\documentclass[11pt]{article}
\pdfoutput=1
\usepackage[utf8]{inputenc}
\usepackage[T1]{fontenc}
\usepackage{lmodern}
\usepackage{microtype}
\usepackage[margin=1in]{geometry}
\usepackage[shortlabels]{enumitem}

\usepackage{double-to-bicat-macros}
\numberwithin{equation}{section}

\usepackage[breaklinks,hidelinks,bookmarksnumbered]{hyperref}
\usepackage[capitalize,noabbrev,nameinlink]{cleveref}

\usepackage{thmtools}
\declaretheorem[within=section]{theorem}
\declaretheorem[sibling=theorem]{proposition}
\declaretheorem[sibling=theorem]{corollary}
\declaretheorem[sibling=theorem]{lemma}
\declaretheorem[sibling=theorem,style=definition]{definition}
\declaretheorem[sibling=theorem,style=definition,qed=\qedsymbol]{construction}
\declaretheorem[sibling=theorem,style=remark,qed=\qedsymbol]{remark}
\declaretheorem[sibling=theorem,style=remark,qed=\qedsymbol]{example}

\usepackage[style=alphabetic,maxbibnames=10]{biblatex}
\renewbibmacro{in:}{}

\title{Transposing cartesian and other structure in double categories}
\author{Evan Patterson}
\date{}

\begin{document}
\maketitle

\begin{abstract}
  The cartesian structure possessed by relations, spans, profunctors, and other
  such morphisms is elegantly expressed by universal properties in double
  categories. Though cartesian double categories were inspired in part by the
  older program of cartesian bicategories, the precise relationship between the
  double-categorical and bicategorical approaches has so far remained
  mysterious, except in special cases. We provide a formal connection by showing
  that every double category with iso-strong finite products, and in particular
  every cartesian equipment, has an underlying cartesian bicategory. To do so,
  we develop broadly applicable techniques for transposing natural
  transformations and adjunctions between double categories, extending a line of
  previous work rooted in the concepts of companions and conjoints.
\end{abstract}

\section{Introduction}
\label{sec:introduction}

The quest to axiomatize the elusive cartesian structure possessed by relations,
spans, and profunctors has been a long one. Bicategories of such morphisms have
an evident set-theoretic cartesian product, but it is not the bicategorical
product. What, then, is its category-theoretic nature? Ideally, the
set-theoretic product would be characterized by a universal property.

An early attempt to solve this puzzle was made by Carboni and Walters
\cite{carboni1987}, who axiomatized \emph{`a bicategory of relations'} as a
locally posetal, symmetric monoidal bicategory in which each object carries the
structure of a commutative comonoid, subject to several axioms. Due in part to
the uncertain status of symmetric monoidal bicategories at the time, Carboni and
Walters were able to define a cartesian bicategory only in the locally posetal
case. Two decades later, Carboni, Kelly, Walters, and Wood proposed a general
definition of a \emph{cartesian bicategory} \cite{carboni2008}. Their later
approach has the virtue of making it more obvious that being cartesian is a
\emph{property} of, not a structure on, a bicategory, yet the definition is a
complicated one that takes much of the paper to state completely.

Further progress depended on realizing that the cartesian structure of
relations, spans, and profunctors is most simply expressed by a universal
property not in bicategories, but in double categories. Building on the double
limits introduced by Grandis and Paré \cite{grandis1999}, Aleiferi defined a
\emph{cartesian double category} to be a double category with binary and nullary
double products \cite{aleiferi2018}. Recently, reviving an idea by Paré
\cite{pare2009}, the author developed a stronger notion of a \emph{double
  category with finite products} \cite{patterson2024}, capturing products of
various shapes, including local products. Behind both approaches is the insight
that, by using double adjunctions \cite{grandis2004}, products in double
categories can be defined as right adjoints to diagonals, analogously to
products in categories.

Yet a lack of clarity still prevails as few formal connections have been made
between cartesian bicategories and double categories with finite products.
Lambert has shown that the underlying bicategory of any \emph{locally posetal}
cartesian equipment is cartesian \cite[\mbox{Proposition 3.1}]{lambert2022}. The
general situation is entangled with issues of bicategorical coherence, and no
formal results are known. Our original impetus for this work was to close this
gap. We do so by showing that the underlying bicategory of any double category
with iso-strong finite products is cartesian
(\cref{thm:underlying-cartesian-bicategory}). It follows that the underlying
bicategory of a cartesian equipment is cartesian
(\cref{cor:underlying-cartesian-bicategory}).

The technique that we use to prove this result is, perhaps, more interesting
than the result itself, as it sheds light on the relationship between double
categories and bicategories generally. In fact, only in the paper's final
section do we study cartesian and cocartesian structure. In the rest of the
paper, we explore a more nebulous question: how can structure in double
categories be ``transposed'' from the strict (2-categorical) direction to the
weak (bicategorical) direction?

Systematic methods to transpose structure in double categories were first
devised by Garner and Gurski \cite{garner2009} and by Shulman
\cite{shulman2010}, with antecedents in earlier work by Shulman
\cite[\mbox{Appendix B}]{shulman2008} and improvements made later by Hansen and
Shulman \cite{hansen2019}. All of these works seek expedient ways to construct
tricategorical structures that avoid long calculations verifying coherence
axioms. Garner and Gurski construct tricategories from ``locally-double
bicategories,'' whereas Hansen and Shulman construct symmetric monoidal
bicategories from symmetric monoidal double categories. At the heart of such
constructions are companions assumed to exist in a base double category.

Companions and conjoints are the basic tool for transposing morphisms in a
double category \cite{brown1976,grandis2004,shulman2008}. A \emph{companion} of
an arrow in a double category is a universal choice of proarrow with the same
domain and codomain; a \emph{conjoint} is similar, except that it has the
opposite orientation. Companions and conjoints play an indispensable role in
this work as in previous ones.

The utility of companions as a general tool is vastly increased by lifting the
property of having a companion from individual morphisms in a double category to
natural transformations between double functors. The companion of a natural
transformation, if it exists, should be a transformation of another kind, whose
components at objects are proarrows rather than arrows
\cite{grandis1999,grandis2019}. We will call them \emph{protransformations}
(\cref{def:lax-protransformation}); they have also be called ``loose
transformations'' or, under our orientation convention, ``horizontal
transformations,'' and they generalize the pseudo natural transformations
familiar from bicategory theory.

But there is a twist: the companion of a natural transformation whose component
arrows have companions is usually only an \emph{oplax} protransformation
(\cref{thm:companion-transformation}). Dually, the conjoint of a natural
transformation is usually only a \emph{lax} protransformation. These generalize
the oplax and lax natural transformations from bicategory theory. The appearance
of laxness is not an accident but a fundamental aspect of transposing structure
in double categories, and we will see that it explains why cartesian
bicategories take their distinctive form. Under very special conditions, the
companion or conjoint of a natural transformation is a pseudo protransformation
(\cref{cor:companion-transformation-pseudo}), as also recently noticed by
Gambino, Garner, and Vasilakopoulou \cite[\S{3}]{gambino2022}.

Having found an environment in which to exhibit companions of natural
transformations, we can apply the universal property of companions to deduce
powerful results for transposing structure with a minimum of calculation. Most
importantly, we can transpose a double adjunction, turning the (strict) triangle
identities that hold between natural transformations into triangle identities
between oplax protransformations that hold only up to coherent isomorphism
(\cref{thm:transpose-adjunction}). Due to the oplaxity, the latter is not quite
a biadjunction in the standard sense but is rather a kind of \emph{lax
  adjunction} \cite{gray1974}.

From here, the passage to cartesian bicategories is quite direct. A cartesian
bicategory does \emph{not} obtain its cartesian product as a right biadjoint to
the diagonal; if it did, the cartesian product would simply be the bicategorical
product, the failure of which was the impetus to invent cartesian bicategories
in the first place. But Trimble has shown that the axioms for a cartesian
bicategory can be reformulated using lax adjunctions, such that the cartesian
product is a right lax adjoint to the diagonal, built out of pseudofunctors,
\emph{oplax} natural transformations, and invertible modifications
(\cref{def:cartesian-bicategory}, following \cite{trimble2009}). Thus, from a
double category with finite products, we can produce a cartesian bicategory
simply by transposing the double adjunction between diagonals and products into
a lax adjunction (\cref{thm:underlying-cartesian-bicategory}).

It appears that all of the known, genuine\footnote{By ``genuine,'' we mean
  cartesian bicategories that do not already have finite products in the
  standard 2-categorical or bicategorical senses.} examples of cartesian
bicategories arise from this construction. Important examples include:
\begin{itemize}[noitemsep]
  \item $\Rel(\cat{S})$, the double category of relations in a regular category
    $\cat{S}$;
  \item $\Span(\cat{S})$, the double category of spans in a finitely complete
    category $\cat{S}$;
  \item $\Prof(\cat{S})$, the double category of profunctors internal to a
    finitely complete category $\cat{S}$ with coequalizers that are stable under
    pullback;
  \item $\catV\text{-}\mathbb{M}\mathsf{at}$, the double category of matrices
    valued in an (infinitary) distributive category $\catV$;
  \item $\catV\text{-}\Prof$, the double category of profunctors enriched over a
    distributive category $\catV$ with coequalizers preserved by products in
    each variable.
\end{itemize}
Since all of these double categories are cartesian equipments, their underlying
bicategories become cartesian bicategories by transposing the double adjunctions
(\cref{cor:underlying-cartesian-bicategory}).

The asymmetry between products and coproducts in a typical double category such
as that of relations or spans is magnified upon passage to the underlying
bicategory. Using the terminology of \cite{patterson2024}, products in a double
category tend to be at most \emph{iso-strong}, meaning essentially that parallel
products of proarrows commute with external composition, as in a cartesian
double category. By contrast, coproducts in a well-behaved double category are
\emph{strong}, meaning that arbitrary span-indexed coproducts of proarrows
commute with external composition. That is enough to imply that the coprojection
and codiagonal transformations each have companions and conjoints that are
\emph{pseudo} protransformations, hence that the double adjunction defining the
coproducts transposes to give simultaneous bicategorical products and coproducts
(\cref{thm:direct-sums}), also known as \emph{direct sums}
\cite{lack2010bicategories}. For instance, the bicategories of relations and of
spans inherit their direct sums, based on set-theoretic coproducts, in this way.

Our outlook is that double categories offer a unifying language for
two-dimensional category theory. Recent works have shown that two-dimensional
products and coproducts, when defined using double adjunctions, attain universal
properties that are simple to state and use \cite{aleiferi2018,patterson2024}.
In this work, we have shown that the bicategorical formulations, found in
cartesian bicategories and bicategories with direct sums, are obtained by
transposition. But even this bicategorical structure is better viewed as
double-categorical because the oplax or lax protransformations furnishing it
have universal properties as companions or conjoints---universal properties that
become invisible upon passing to the underlying oplax or lax natural
transformations. In the future, one might hope to simplify the construction of
other complicated bicategorical structures using similar techniques.

\paragraph{Outline.} The plan of the paper is as follows. We begin in
\cref{sec:protransformations} by defining lax, oplax, and pseudo
protransformations between lax double functors and assembling them into the
proarrows of double categories (\cref{sec:protransformation-double-cat}). Within
these double categories, we recognize certain oplax or lax protransformations as
companions or conjoints of natural transformations
(\cref{sec:companion-transformations}), thus assigning a universal property to a
procedure for transposing natural transformations. In
\cref{sec:double-to-bicat}, we observe that transposing natural transformations
is pseudofunctorial and also yields purely bicategorical results as a corollary
(\cref{sec:transposing-transformations}). We then upgrade our method to
transpose a double adjunction into a kind of lax adjunction
(\cref{sec:transposing-adjunctions}). Finally, in
\cref{sec:double-to-cartesian-bicat}, we apply these results, first, to double
categories with finite products (\cref{sec:structure-proarrows}), exhibiting
their underlying bicategories as cartesian
(\cref{sec:transposing-cartesian-structure}), and, dually, to double categories
with finite coproducts (\cref{sec:transposing-cocartesian-structure}), under
stronger assumptions exhibiting underlying bicategories with direct sums.
Throughout the paper, we draw on the theory of companions and conjoints in a
double category. This theory is carefully reviewed in \cref{sec:companions},
including a sequence of lemmas on reshaping cells that have often been left
implicit in the literature (\cref{sec:sliding}) and a way of formulating
companions as a biadjunction (\cref{sec:companions-biadjoint}).

\paragraph{Conventions.} Categories $\cat{C}, \cat{D}, \dots$ are written
in sans-serif font, 2-categories and bicategories $\bicat{B}, \bicat{C}, \dots$
in bold font, and double categories $\dbl{D}, \dbl{E}, \dots$ in blackboard bold
font. The composite of morphisms $x \xto{f} y \xto{g} z$ in a category is
written variously in applicative order as $g \circ f$ or in diagrammatic order as
$f \cdot g$. The composite of proarrows $x \xproto{m} y \xproto{n} z$ in a double
category is always written in diagrammatic order as $m \odot n$.

Double categories and double functors are assumed to be \emph{pseudo} unless
otherwise stated. However, to avoid cumbersome notation, associators and unitors
are elided in pasting diagrams when it is clear from typing constraints where
they must be inserted. The \emph{``internal''} category structure of a double
category $\dbl{D}$ refers to the underlying categories $\dbl{D}_0$ and
$\dbl{D}_1$, whereas the \emph{``external''} category structure refers to the
pseudocategory with composition
$\odot: \dbl{D}_1 \times_{\dbl{D}_0} \dbl{D}_1 \to \dbl{D}_1$ and identities
$\id: \dbl{D}_0 \to \dbl{D}_1$.

\paragraph{Acknowledgments.} I am grateful to Nathanael Arkor for insightful
comments on the first draft of this paper, which helped to streamline the
presentation in \cref{sec:companions-biadjoint} and to construct direct sums in
\cref{sec:transposing-cocartesian-structure}, among other improvements.

\section{Natural transformations and protransformations}
\label{sec:protransformations}

A double category can be defined as a pseudocategory internal to $\Cat$, the
2-category of categories. Likewise, a lax, colax, or pseudo double functor is
respectively a lax, colax, or pseudo functor internal to $\Cat$. Having adopted
this perspective, one must sooner or later also consider natural transformations
internal to $\Cat$. These are not the usual (tight) natural transformations
between double functors but a transposed notion, whose components on objects are
proarrows rather than arrows. We will call them ``protransformations;'' another
plausible name is ``loose transformations.''

\subsection{Double categories of protransformations}
\label{sec:protransformation-double-cat}

In this section, we define lax protransformations between lax double functors
and show how they assemble into the proarrows of a double category of lax
functors. Although these notions are straightforward laxifications of pseudo
notions long present in the literature on double categories
\cite{grandis1999,grandis2019}, it seems worthwhile to give a detailed account
since they are central to the present paper and we will need to reference the
various axioms and operations.

\begin{definition}[Protransformation] \label{def:lax-protransformation}
  Let $F,G: \dbl{D} \to \dbl{E}$ be lax functors between double categories. A
  \define{lax}, \define{oplax}, or \define{pseudo protransformation}
  $\tau: F \proTo G$ is respectively a lax, oplax, or pseudo natural
  transformation, internal to $\Cat$, from $F$ to $G$.
\end{definition}

We will immediately unwind this conceptual definition, but for the general
definitions of pseudocategories, pseudofunctors, and pseudo natural
transformations, see \cite{martins-ferreira2006}.

Equivalently, a \define{lax protransformation} $\tau: F \proTo G$ is seen to
consist of
\begin{itemize}
  \item for each object $x \in \dbl{D}$, a proarrow $\tau_x: Fx \proto Gx$ in
    $\dbl{E}$, the \define{component} of $\tau$ at $x$;
  \item for each arrow $f: x \to y$ in $\dbl{D}$, a cell $\tau_f$ in $\dbl{E}$
    of the form
    \begin{equation*}
      \begin{tikzcd}
        Fx & Gx \\
        Fy & Gy
        \arrow["Ff"', from=1-1, to=2-1]
        \arrow[""{name=0, anchor=center, inner sep=0}, "{\tau_x}", "\shortmid"{marking}, from=1-1, to=1-2]
        \arrow["Gf", from=1-2, to=2-2]
        \arrow[""{name=1, anchor=center, inner sep=0}, "{\tau_y}"', "\shortmid"{marking}, from=2-1, to=2-2]
        \arrow["{\tau_f}"{description}, draw=none, from=0, to=1]
      \end{tikzcd},
    \end{equation*}
    called the \define{component} of $\tau$ at $f$;
  \item for each proarrow $m: x \proto y$ in $\dbl{D}$, a globular cell $\tau_m$
    in $\dbl{E}$ of the form
    \begin{equation*}
      \begin{tikzcd}
        Fx & Gx & Gy \\
        Fx & Fy & Gy
        \arrow["Fm"', "\shortmid"{marking}, from=2-1, to=2-2]
        \arrow["{\tau_y}"', "\shortmid"{marking}, from=2-2, to=2-3]
        \arrow["{\tau_x}", "\shortmid"{marking}, from=1-1, to=1-2]
        \arrow["Gm", "\shortmid"{marking}, from=1-2, to=1-3]
        \arrow[Rightarrow, no head, from=1-1, to=2-1]
        \arrow[Rightarrow, no head, from=1-3, to=2-3]
        \arrow["{\tau_m}"{description}, draw=none, from=1-2, to=2-2]
      \end{tikzcd},
    \end{equation*}
    called the \define{naturality comparison} of $\tau$ at $m$.
\end{itemize}
The following axioms must be satisfied.
\begin{itemize}
  \item Functoriality of components: $\tau_{f \cdot g} = \tau_f \cdot \tau_g$ for
    all arrows $x \xto{f} y\xto{g} z$ in $\dbl{D}$, and also
    $\tau_{1_x} = 1_{\tau_x}$ for all objects $x$ in $\dbl{D}$.
  \item Naturality with respect to cells: for every cell
    $\stdInlineCell{\alpha}$ in $\dbl{D}$,
    \begin{equation*}
      \begin{tikzcd}
        Fx & Gx & Gy \\
        Fx & Fy & Gy \\
        Fw & Fz & Gz
        \arrow[""{name=0, anchor=center, inner sep=0}, "Fm", "\shortmid"{marking}, from=2-1, to=2-2]
        \arrow[""{name=1, anchor=center, inner sep=0}, "{\tau_y}", "\shortmid"{marking}, from=2-2, to=2-3]
        \arrow["{\tau_x}", "\shortmid"{marking}, from=1-1, to=1-2]
        \arrow["Gm", "\shortmid"{marking}, from=1-2, to=1-3]
        \arrow[Rightarrow, no head, from=1-1, to=2-1]
        \arrow[Rightarrow, no head, from=1-3, to=2-3]
        \arrow["{\tau_m}"{description}, draw=none, from=1-2, to=2-2]
        \arrow["Ff"', from=2-1, to=3-1]
        \arrow["Fg"{description}, from=2-2, to=3-2]
        \arrow[""{name=2, anchor=center, inner sep=0}, "Fn"', "\shortmid"{marking}, from=3-1, to=3-2]
        \arrow["Gg", from=2-3, to=3-3]
        \arrow[""{name=3, anchor=center, inner sep=0}, "{\tau_z}"', "\shortmid"{marking}, from=3-2, to=3-3]
        \arrow["F\alpha"{description}, draw=none, from=0, to=2]
        \arrow["{\tau_g}"{description}, draw=none, from=1, to=3]
      \end{tikzcd}
      \quad=\quad
      \begin{tikzcd}
        Fx & Gx & Gy \\
        Fw & Gw & Gz \\
        Fw & Fz & Gz
        \arrow["Ff"', from=1-1, to=2-1]
        \arrow[""{name=0, anchor=center, inner sep=0}, "{\tau_x}", "\shortmid"{marking}, from=1-1, to=1-2]
        \arrow[""{name=1, anchor=center, inner sep=0}, "Gm", "\shortmid"{marking}, from=1-2, to=1-3]
        \arrow["Gf"{description}, from=1-2, to=2-2]
        \arrow[""{name=2, anchor=center, inner sep=0}, "{\tau_w}"', "\shortmid"{marking}, from=2-1, to=2-2]
        \arrow[Rightarrow, no head, from=2-1, to=3-1]
        \arrow["Gg", from=1-3, to=2-3]
        \arrow[""{name=3, anchor=center, inner sep=0}, "Gn"', "\shortmid"{marking}, from=2-2, to=2-3]
        \arrow["Fn"', "\shortmid"{marking}, from=3-1, to=3-2]
        \arrow[Rightarrow, no head, from=2-3, to=3-3]
        \arrow["{\tau_z}"', "\shortmid"{marking}, from=3-2, to=3-3]
        \arrow["{\tau_n}"{description}, draw=none, from=2-2, to=3-2]
        \arrow["{\tau_f}"{description}, draw=none, from=0, to=2]
        \arrow["G\alpha"{description}, draw=none, from=1, to=3]
      \end{tikzcd}.
    \end{equation*}
  \item Coherence with respect to external composition: for every pair of
    composable proarrows $m: x \proto y$ and $n: y \proto z$ in $\dbl{D}$,
    \begin{equation*}
      \begin{tikzcd}[row sep=scriptsize]
        Fx & Gx & Gy & Gz \\
        Fx & Fy & Gy & Gz \\
        Fx & Fy & Fz & Gz \\
        Fx && Fz & Fz
        \arrow[""{name=0, anchor=center, inner sep=0}, "Fm"', "\shortmid"{marking}, from=2-1, to=2-2]
        \arrow["{\tau_y}"', "\shortmid"{marking}, from=2-2, to=2-3]
        \arrow["{\tau_x}", "\shortmid"{marking}, from=1-1, to=1-2]
        \arrow["Gm", "\shortmid"{marking}, from=1-2, to=1-3]
        \arrow[Rightarrow, no head, from=1-1, to=2-1]
        \arrow[Rightarrow, no head, from=1-3, to=2-3]
        \arrow["{\tau_m}"{description}, draw=none, from=1-2, to=2-2]
        \arrow[""{name=1, anchor=center, inner sep=0}, "Gn"', "\shortmid"{marking}, from=2-3, to=2-4]
        \arrow[""{name=2, anchor=center, inner sep=0}, "Gn", "\shortmid"{marking}, from=1-3, to=1-4]
        \arrow[Rightarrow, no head, from=1-4, to=2-4]
        \arrow["Fn"', "\shortmid"{marking}, from=3-2, to=3-3]
        \arrow[""{name=3, anchor=center, inner sep=0}, "Fm"', "\shortmid"{marking}, from=3-1, to=3-2]
        \arrow[Rightarrow, no head, from=2-2, to=3-2]
        \arrow[""{name=4, anchor=center, inner sep=0}, "{\tau_z}"', "\shortmid"{marking}, from=3-3, to=3-4]
        \arrow[Rightarrow, no head, from=2-4, to=3-4]
        \arrow["{\tau_n}"{description}, draw=none, from=2-3, to=3-3]
        \arrow[Rightarrow, no head, from=2-1, to=3-1]
        \arrow[""{name=5, anchor=center, inner sep=0}, "{F(m \odot n)}"', "\shortmid"{marking}, from=4-1, to=4-3]
        \arrow[Rightarrow, no head, from=3-1, to=4-1]
        \arrow[Rightarrow, no head, from=3-3, to=4-3]
        \arrow[Rightarrow, no head, from=3-4, to=4-4]
        \arrow[""{name=6, anchor=center, inner sep=0}, "{\tau_z}"', "\shortmid"{marking}, from=4-3, to=4-4]
        \arrow["{1_{Gn}}"{description}, draw=none, from=2, to=1]
        \arrow["{1_{Fm}}"{description}, draw=none, from=0, to=3]
        \arrow["{F_{m,n}}"{description}, draw=none, from=3-2, to=5]
        \arrow["{1_{\tau_z}}"{description}, draw=none, from=4, to=6]
      \end{tikzcd}
      \quad=\quad
      \begin{tikzcd}[row sep=scriptsize]
        Fx & Gx & Gy & Gz \\
        Fx & Gx && Gz \\
        Fx && Fz & Gz
        \arrow[""{name=0, anchor=center, inner sep=0}, "{G(m \odot n)}"', "\shortmid"{marking}, from=2-2, to=2-4]
        \arrow["Gm", "\shortmid"{marking}, from=1-2, to=1-3]
        \arrow["Gn", "\shortmid"{marking}, from=1-3, to=1-4]
        \arrow[Rightarrow, no head, from=1-4, to=2-4]
        \arrow[Rightarrow, no head, from=1-2, to=2-2]
        \arrow[""{name=1, anchor=center, inner sep=0}, "{\tau_x}", "\shortmid"{marking}, from=1-1, to=1-2]
        \arrow[Rightarrow, no head, from=1-1, to=2-1]
        \arrow[""{name=2, anchor=center, inner sep=0}, "{\tau_x}"', "\shortmid"{marking}, from=2-1, to=2-2]
        \arrow["{\tau_z}"', "\shortmid"{marking}, from=3-3, to=3-4]
        \arrow["{F(m \odot n)}"', "\shortmid"{marking}, from=3-1, to=3-3]
        \arrow[""{name=3, anchor=center, inner sep=0}, Rightarrow, no head, from=2-4, to=3-4]
        \arrow[""{name=4, anchor=center, inner sep=0}, Rightarrow, no head, from=2-1, to=3-1]
        \arrow["{1_{\tau_x}}"{description}, draw=none, from=1, to=2]
        \arrow["{\tau_{m \odot n}}"{description}, draw=none, from=4, to=3]
        \arrow["{G_{m,n}}"{description, pos=0.4}, draw=none, from=1-3, to=0]
      \end{tikzcd}.
    \end{equation*}
  \item Coherence with respect to external identities: for every object $x$ in
    $\dbl{D}$,
    \begin{equation*}
      \begin{tikzcd}[row sep=scriptsize]
        Fx & Gx & Gx \\
        Fx & Gx & Gx \\
        Fx & Fx & Gx
        \arrow[""{name=0, anchor=center, inner sep=0}, "{G \mathrm{id}_x}"', "\shortmid"{marking}, from=2-2, to=2-3]
        \arrow[""{name=1, anchor=center, inner sep=0}, "{\tau_x}"', "\shortmid"{marking}, from=2-1, to=2-2]
        \arrow["{F \mathrm{id}_x}"', "\shortmid"{marking}, from=3-1, to=3-2]
        \arrow["{\tau_x}"', "\shortmid"{marking}, from=3-2, to=3-3]
        \arrow[Rightarrow, no head, from=2-1, to=3-1]
        \arrow[Rightarrow, no head, from=2-3, to=3-3]
        \arrow["{\tau_{\mathrm{id}_x}}"{description}, draw=none, from=2-2, to=3-2]
        \arrow[""{name=2, anchor=center, inner sep=0}, "{\mathrm{id}_{Gx}}", "\shortmid"{marking}, from=1-2, to=1-3]
        \arrow[Rightarrow, no head, from=1-2, to=2-2]
        \arrow[Rightarrow, no head, from=1-3, to=2-3]
        \arrow[""{name=3, anchor=center, inner sep=0}, "{\tau_x}", "\shortmid"{marking}, from=1-1, to=1-2]
        \arrow[Rightarrow, no head, from=1-1, to=2-1]
        \arrow["{G_x}"{description}, draw=none, from=2, to=0]
        \arrow["{1_{\tau_x}}"{description}, draw=none, from=3, to=1]
      \end{tikzcd}
      \quad=\quad
      \begin{tikzcd}[row sep=scriptsize]
        Fx & Fx & Gx \\
        Fx & Fx & Gx
        \arrow[""{name=0, anchor=center, inner sep=0}, "{F \mathrm{id}_x}"', "\shortmid"{marking}, from=2-1, to=2-2]
        \arrow[""{name=1, anchor=center, inner sep=0}, "{\tau_x}"', "\shortmid"{marking}, from=2-2, to=2-3]
        \arrow[Rightarrow, no head, from=1-1, to=2-1]
        \arrow[Rightarrow, no head, from=1-2, to=2-2]
        \arrow[Rightarrow, no head, from=1-3, to=2-3]
        \arrow[""{name=2, anchor=center, inner sep=0}, "{\mathrm{id}_{Fx}}", "\shortmid"{marking}, from=1-1, to=1-2]
        \arrow[""{name=3, anchor=center, inner sep=0}, "{\tau_x}", "\shortmid"{marking}, from=1-2, to=1-3]
        \arrow["{F_x}"{description}, draw=none, from=2, to=0]
        \arrow["{1_{\tau_x}}"{description}, draw=none, from=3, to=1]
      \end{tikzcd}.
    \end{equation*}
\end{itemize}
A \define{oplax} protransformation $\tau: F \proTo G$ is defined similarly,
except that the orientation of the naturality comparisons is reversed, giving
them the form
\begin{equation*}
  \begin{tikzcd}
    Fx & Fy & Gy \\
    Fx & Gx & Gy
    \arrow["Fm", "\shortmid"{marking}, from=1-1, to=1-2]
    \arrow["{\tau_y}", "\shortmid"{marking}, from=1-2, to=1-3]
    \arrow["{\tau_x}"', "\shortmid"{marking}, from=2-1, to=2-2]
    \arrow["Gm"', "\shortmid"{marking}, from=2-2, to=2-3]
    \arrow[Rightarrow, no head, from=1-1, to=2-1]
    \arrow[Rightarrow, no head, from=1-3, to=2-3]
    \arrow["{\tau_m}"{description}, draw=none, from=1-2, to=2-2]
  \end{tikzcd}
\end{equation*}
for each proarrow $m: x \proto y$ and changing the above axioms accordingly.
Finally, a lax or oplax protransformation $\tau$ is \define{pseudo} if every
comparison cell $\tau_m$ is invertible, and is \define{strict} if every comparison
is an identity.

Unless otherwise stated, protransformations, like double categories and double
functors, \emph{are assumed to be pseudo}. Pseudo protransformations were first
defined, under a different name, by Grandis and Paré; see
\cite[\S{7.4}]{grandis1999} or \cite[\mbox{Definition 3.8.2}]{grandis2019}.

Double functors between a fixed pair of double categories, and
protransformations of those, generally do not form a category. Since their
components are proarrows, protransformations have a composition that is
associative and unital only up to isomorphism. But protransformations are the
proarrows of a double category. The arrows of this double category are natural
transformations and the cells are \emph{modifications}.

\begin{definition}[Modification] \label{def:modification}
  Let $F,G,H,K: \dbl{D} \to \dbl{E}$ be lax functors between double categories.
  A \define{modification}
  \begin{equation*}
    \begin{tikzcd}
      F & G \\
      H & K
      \arrow["\alpha"', from=1-1, to=2-1]
      \arrow["\beta", from=1-2, to=2-2]
      \arrow[""{name=0, anchor=center, inner sep=0}, "\sigma", "\shortmid"{marking}, from=1-1, to=1-2]
      \arrow[""{name=1, anchor=center, inner sep=0}, "\tau"', "\shortmid"{marking}, from=2-1, to=2-2]
      \arrow["\mu"{description}, draw=none, from=0, to=1]
    \end{tikzcd}
  \end{equation*}
  bounded by lax protransformations $\sigma: F \proTo G$ and $\tau: H \proTo K$ and
  natural transformations $\alpha: F \To H$ and $\beta: G \To K$ consists of, for each
  object $x \in \dbl{D}$, a cell in $\dbl{E}$ of
  the form
  \begin{equation*}
    \begin{tikzcd}
      Fx & Gx \\
      Hx & Kx
      \arrow["{\alpha_x}"', from=1-1, to=2-1]
      \arrow["{\beta_x}", from=1-2, to=2-2]
      \arrow[""{name=0, anchor=center, inner sep=0}, "{\sigma_x}", "\shortmid"{marking}, from=1-1, to=1-2]
      \arrow[""{name=1, anchor=center, inner sep=0}, "{\tau_x}"', "\shortmid"{marking}, from=2-1, to=2-2]
      \arrow["{\mu_x}"{description}, draw=none, from=0, to=1]
    \end{tikzcd},
  \end{equation*}
  the \define{component} of $\mu$ at $x$. Two axioms must be satisfied.
  \begin{itemize}
    \item Internal equivariance: for every arrow $f: x \to y$ in $\dbl{D}$,
      \begin{equation*}
        \begin{tikzcd}
          Fx & Gx \\
          Hx & Kx \\
          Hy & Kf
          \arrow["{\alpha_x}"', from=1-1, to=2-1]
          \arrow["{\beta_x}", from=1-2, to=2-2]
          \arrow[""{name=0, anchor=center, inner sep=0}, "{\sigma_x}", "\shortmid"{marking}, from=1-1, to=1-2]
          \arrow[""{name=1, anchor=center, inner sep=0}, "{\tau_x}", "\shortmid"{marking}, from=2-1, to=2-2]
          \arrow["Hf"', from=2-1, to=3-1]
          \arrow["Kf", from=2-2, to=3-2]
          \arrow[""{name=2, anchor=center, inner sep=0}, "{\tau_y}"', "\shortmid"{marking}, from=3-1, to=3-2]
          \arrow["{\mu_x}"{description, pos=0.4}, draw=none, from=0, to=1]
          \arrow["{\tau_f}"{description}, draw=none, from=1, to=2]
        \end{tikzcd}
        \quad=\quad
        \begin{tikzcd}
          Fx & Gx \\
          Fy & Gy \\
          Hy & Ky
          \arrow["{\alpha_y}"', from=2-1, to=3-1]
          \arrow["{\beta_y}", from=2-2, to=3-2]
          \arrow[""{name=0, anchor=center, inner sep=0}, "{\sigma_y}", "\shortmid"{marking}, from=2-1, to=2-2]
          \arrow[""{name=1, anchor=center, inner sep=0}, "{\tau_y}"', "\shortmid"{marking}, from=3-1, to=3-2]
          \arrow["Ff"', from=1-1, to=2-1]
          \arrow["Gf", from=1-2, to=2-2]
          \arrow[""{name=2, anchor=center, inner sep=0}, "{\sigma_x}", "\shortmid"{marking}, from=1-1, to=1-2]
          \arrow["{\mu_y}"{description}, draw=none, from=0, to=1]
          \arrow["{\sigma_f}"{description, pos=0.4}, draw=none, from=2, to=0]
        \end{tikzcd}.
      \end{equation*}
    \item External equivariance: for every proarrow $m: x \proto y$ in $\dbl{D}$,
      \begin{equation*}
        \begin{tikzcd}
          Fx & Gx & Gy \\
          Hx & Kx & Ky \\
          Hx & Hy & Ky
          \arrow["{\alpha_x}"', from=1-1, to=2-1]
          \arrow["{\beta_x}"{description}, from=1-2, to=2-2]
          \arrow[""{name=0, anchor=center, inner sep=0}, "{\sigma_x}", "\shortmid"{marking}, from=1-1, to=1-2]
          \arrow[""{name=1, anchor=center, inner sep=0}, "{\tau_x}"', "\shortmid"{marking}, from=2-1, to=2-2]
          \arrow[""{name=2, anchor=center, inner sep=0}, "Km"', "\shortmid"{marking}, from=2-2, to=2-3]
          \arrow["Hm"', "\shortmid"{marking}, from=3-1, to=3-2]
          \arrow["{\tau_y}"', "\shortmid"{marking}, from=3-2, to=3-3]
          \arrow[""{name=3, anchor=center, inner sep=0}, "Gm", "\shortmid"{marking}, from=1-2, to=1-3]
          \arrow["{\beta_y}", from=1-3, to=2-3]
          \arrow[Rightarrow, no head, from=2-1, to=3-1]
          \arrow[Rightarrow, no head, from=2-3, to=3-3]
          \arrow["{\tau_m}"{description}, draw=none, from=2-2, to=3-2]
          \arrow["{\mu_x}"{description}, draw=none, from=0, to=1]
          \arrow["{\beta_m}"{description}, draw=none, from=3, to=2]
        \end{tikzcd}
        \quad=\quad
        \begin{tikzcd}
          Fx & Gx & Gy \\
          Fx & Fy & Gy \\
          Hx & Hy & Ky
          \arrow["{\sigma_x}", "\shortmid"{marking}, from=1-1, to=1-2]
          \arrow["Gm", "\shortmid"{marking}, from=1-2, to=1-3]
          \arrow[Rightarrow, no head, from=1-1, to=2-1]
          \arrow[Rightarrow, no head, from=1-3, to=2-3]
          \arrow[""{name=0, anchor=center, inner sep=0}, "Fm", "\shortmid"{marking}, from=2-1, to=2-2]
          \arrow[""{name=1, anchor=center, inner sep=0}, "{\sigma_y}", "\shortmid"{marking}, from=2-2, to=2-3]
          \arrow["{\sigma_m}"{description}, draw=none, from=1-2, to=2-2]
          \arrow["{\alpha_x}"', from=2-1, to=3-1]
          \arrow["{\alpha_y}"{description}, from=2-2, to=3-2]
          \arrow[""{name=2, anchor=center, inner sep=0}, "Hm"', "\shortmid"{marking}, from=3-1, to=3-2]
          \arrow[""{name=3, anchor=center, inner sep=0}, "{\tau_y}"', "\shortmid"{marking}, from=3-2, to=3-3]
          \arrow["{\beta_y}", from=2-3, to=3-3]
          \arrow["{\alpha_m}"{description}, draw=none, from=0, to=2]
          \arrow["{\mu_y}"{description}, draw=none, from=1, to=3]
        \end{tikzcd}.
      \end{equation*}
  \end{itemize}
  Modifications between oplax protransformations are defined in nearly the same
  way, adjusting only the second axiom. Finally, modifications between pseudo
  protransformations are equivalently defined by viewing them as modifications
  between either lax or oplax transformations.
\end{definition}

Modifications bounded by two natural transformations, optionally allowed to be
pseudo, and by two pseudo protransformations were first defined by Grandis and
Paré; see \cite[\S{7.4}]{grandis1999} or \cite[\mbox{Definition
  3.8.3}]{grandis2019}. We prefer natural transformations to be strict, since
they belong to the strict direction of the double categories, but we allow
protransformations to be lax or oplax, since even strict natural transformations
give rise to lax or oplax protransformations, as we will see.

Lax double functors, natural transformations, protransformations, and
modifications form a double category, as asserted in \cite[\S{7.4}]{grandis1999}
and \cite[\mbox{Theorem 3.8.4}]{grandis2019}. The same is true if the
protransformations are allowed to be either lax or oplax. For ease of reference,
we state the construction in all three cases.

\begin{construction}[Double categories of protransformations]
  Let $\dbl{D}$ and $\dbl{E}$ be double categories. Then there are double
  categories
  \begin{equation*}
    \LaxLax(\dbl{D},\dbl{E}), \qquad
    \LaxOpl(\dbl{D},\dbl{E}), \qquad\text{and}\qquad
    \LaxPs(\dbl{D},\dbl{E})
  \end{equation*}
  whose
  \begin{itemize}[noitemsep]
    \item objects are lax double functors $\dbl{D} \to \dbl{E}$;
    \item arrows are natural transformations;
    \item proarrows are lax, oplax, or pseudo protransformations, respectively;
    \item cells are modifications.
  \end{itemize}
  In all three double categories, the composite $\sigma \odot \tau: F \proTo H$
  of protransformations $\sigma: F \proTo G$ and $\tau: G \proTo H$ is defined
  componentwise in $\dbl{E}$, so that
  $(\sigma \odot \tau)_x \coloneqq \sigma_x \odot \tau_x$ for each object
  $x \in \dbl{D}$ and
  \begin{equation*}
    \begin{tikzcd}
      Fx & Hx \\
      Fy & Hy
      \arrow["Ff"', from=1-1, to=2-1]
      \arrow["Hf", from=1-2, to=2-2]
      \arrow[""{name=0, anchor=center, inner sep=0}, "{(\sigma \odot \tau)_x}", "\shortmid"{marking}, from=1-1, to=1-2]
      \arrow[""{name=1, anchor=center, inner sep=0}, "{(\sigma \odot \tau)_y}"', "\shortmid"{marking}, from=2-1, to=2-2]
      \arrow["{(\sigma \odot \tau)_f}"{description}, draw=none, from=0, to=1]
    \end{tikzcd}
    \quad\coloneqq\quad
    \begin{tikzcd}
      Fx & Gx & Hx \\
      Fy & Gy & Hy
      \arrow[""{name=0, anchor=center, inner sep=0}, "{\sigma_x}", "\shortmid"{marking}, from=1-1, to=1-2]
      \arrow[""{name=1, anchor=center, inner sep=0}, "{\tau_x}", "\shortmid"{marking}, from=1-2, to=1-3]
      \arrow["Ff"', from=1-1, to=2-1]
      \arrow["Gf"{description}, from=1-2, to=2-2]
      \arrow["Hf", from=1-3, to=2-3]
      \arrow[""{name=2, anchor=center, inner sep=0}, "{\sigma_y}"', "\shortmid"{marking}, from=2-1, to=2-2]
      \arrow[""{name=3, anchor=center, inner sep=0}, "{\tau_y}"', "\shortmid"{marking}, from=2-2, to=2-3]
      \arrow["{\sigma_f}"{description}, draw=none, from=0, to=2]
      \arrow["{\tau_f}"{description}, draw=none, from=1, to=3]
    \end{tikzcd}
  \end{equation*}
  for each arrow $f: x \to y$ in $\dbl{D}$, and (in the lax case) is defined on
  naturality comparisons by
  \begin{equation*}
    \begin{tikzcd}
      Fx & Hx & Hy \\
      Fx & Fy & Hy
      \arrow["Hm", "\shortmid"{marking}, from=1-2, to=1-3]
      \arrow["Fm"', "\shortmid"{marking}, from=2-1, to=2-2]
      \arrow["{(\sigma \odot \tau)_x}", "\shortmid"{marking}, from=1-1, to=1-2]
      \arrow["{(\sigma \odot \tau)_y}"', "\shortmid"{marking}, from=2-2, to=2-3]
      \arrow[Rightarrow, no head, from=1-1, to=2-1]
      \arrow[Rightarrow, no head, from=1-3, to=2-3]
      \arrow["{(\sigma \odot \tau)_m}"{description}, draw=none, from=1-2, to=2-2]
    \end{tikzcd}
    \quad\coloneqq\quad
    \begin{tikzcd}[row sep=scriptsize]
      Fx & Gx & Hx & Hy \\
      Fx & Gx & Gy & Hy \\
      Fx & Fy & Gy & Hy
      \arrow[""{name=0, anchor=center, inner sep=0}, "{\sigma_x}"', "\shortmid"{marking}, from=2-1, to=2-2]
      \arrow["Gm", "\shortmid"{marking}, from=2-2, to=2-3]
      \arrow["{\tau_x}", "\shortmid"{marking}, from=1-2, to=1-3]
      \arrow["Hm", "\shortmid"{marking}, from=1-3, to=1-4]
      \arrow[""{name=1, anchor=center, inner sep=0}, "{\tau_z}", "\shortmid"{marking}, from=2-3, to=2-4]
      \arrow[Rightarrow, no head, from=1-2, to=2-2]
      \arrow[Rightarrow, no head, from=1-1, to=2-1]
      \arrow["Fm"', "\shortmid"{marking}, from=3-1, to=3-2]
      \arrow[""{name=2, anchor=center, inner sep=0}, "{\sigma_x}", "\shortmid"{marking}, from=1-1, to=1-2]
      \arrow["{\sigma_y}"', "\shortmid"{marking}, from=3-2, to=3-3]
      \arrow[""{name=3, anchor=center, inner sep=0}, "{\tau_z}"', "\shortmid"{marking}, from=3-3, to=3-4]
      \arrow[Rightarrow, no head, from=1-4, to=2-4]
      \arrow[Rightarrow, no head, from=2-3, to=3-3]
      \arrow[Rightarrow, no head, from=2-4, to=3-4]
      \arrow[Rightarrow, no head, from=2-1, to=3-1]
      \arrow["{\tau_m}"{description}, draw=none, from=1-3, to=2-3]
      \arrow["{\sigma_m}"{description}, draw=none, from=2-2, to=3-2]
      \arrow["{1_{\sigma_x}}"{description}, draw=none, from=2, to=0]
      \arrow["{1_{\tau_z}}"{description}, draw=none, from=1, to=3]
    \end{tikzcd}
  \end{equation*}
  for each proarrow $m: x \proto y$ in $\dbl{D}$. The identity protransformation
  $\id_F$ has components $(\id_F)_x \coloneqq \id_{Fx}$ and
  $(\id_F)_f \coloneqq \id_{Ff}$ and its (invertible) naturality comparisons are
  given by unitors in $\dbl{E}$. Modifications compose componentwise in
  $\dbl{E}$, in both directions. Finally, the associator and unitor
  modifications are also defined componentwise by associators and unitors in
  $\dbl{E}$.

  Lax and oplax protransformations are dual in that there is an isomorphism of
  double categories
  \begin{equation*}
    \LaxOpl(\dbl{D}, \dbl{E})^\rev \cong \LaxLax(\dbl{D}^\rev, \dbl{E}^\rev),
  \end{equation*}
  where $(-)^\rev$ is the reversal duality from \cref{rem:duality}.
\end{construction}

Protransformations, like natural transformations, can be pre-whiskered with lax
functors. While that might seem unsurprising, it is worth examining with some
care since protransformations generally \emph{cannot} be post-whiskered with lax
functors, only with pseudo ones.

\begin{construction}[Pre-whiskering]
  Let $\tau: G \proTo H: \dbl{D} \to \dbl{E}$ be a protransformation, possibly lax or
  oplax, between lax double functors. The \define{pre-whiskering} of $\tau$ with
  another lax functor $F: \dbl{C} \to \dbl{D}$ is a protransformation
  \begin{equation*}
    \tau*F: G \circ F \proTo H \circ F: \dbl{C} \to \dbl{E}
  \end{equation*}
  of the same kind, defined to have components $(\tau * F)_x \coloneq \tau_{Fx}$ and
  $(\tau * F)_f \coloneqq \tau_{Ff}$ at each object $x$ and arrow $f$ in $\dbl{C}$ and
  naturality comparisons $(\tau * F)_m \coloneqq \tau_{Fm}$ for each proarrow $m$ in
  $\dbl{C}$.

  Similarly, given lax double functors $G,H,K,L: \dbl{D} \to \dbl{E}$, a
  modification as on the left
  \begin{equation*}
    \begin{tikzcd}
      G & H \\
      K & L
      \arrow["\alpha"', from=1-1, to=2-1]
      \arrow["\beta", from=1-2, to=2-2]
      \arrow[""{name=0, anchor=center, inner sep=0}, "\sigma", "\shortmid"{marking}, from=1-1, to=1-2]
      \arrow[""{name=1, anchor=center, inner sep=0}, "\tau"', "\shortmid"{marking}, from=2-1, to=2-2]
      \arrow["\mu"{description}, draw=none, from=0, to=1]
    \end{tikzcd}
    \qquad\leadsto\qquad
    \begin{tikzcd}
      {G \circ F} & {H \circ F} \\
      {K \circ F} & {L \circ F}
      \arrow["{\alpha*F}"', from=1-1, to=2-1]
      \arrow["{\beta*F}", from=1-2, to=2-2]
      \arrow[""{name=0, anchor=center, inner sep=0}, "{\sigma*F}", "\shortmid"{marking}, from=1-1, to=1-2]
      \arrow[""{name=1, anchor=center, inner sep=0}, "{\tau*F}"', "\shortmid"{marking}, from=2-1, to=2-2]
      \arrow["{\mu*F}"{description}, draw=none, from=0, to=1]
    \end{tikzcd}
  \end{equation*}
  has a \define{pre-whiskering} with a lax functor $F: \dbl{C} \to \dbl{D}$, the
  modification on the right defined to have components
  $(\mu * F)_x \coloneqq \mu_{Fx}$ at each object $x \in \dbl{C}$.
\end{construction}
\begin{proof}[Proof of well-definedness]
  We need to check that pre-whiskerings of protransformations and modifications
  are well-defined. The proofs are essentially the same as in the bicategorical
  setting \cite[\mbox{Lemma 11.1.5}]{johnson2021} but we record them anyway.

  Given, say, a lax protransformation $\tau: G \proTo H$ between lax double
  functors $G, H: \dbl{D} \to \dbl{E}$, the pre-whiskering $\tau*F$ has functorial
  components $\tau_{Ff}$ by the functoriality of $F_0: \dbl{C}_0 \to \dbl{D}_0$, and
  the naturality of the pre-whiskering $\tau*F$ at a cell $\alpha$ in $\dbl{C}$
  follows immediately from the naturality of $\tau$ at the cell $F\alpha$ in $\dbl{D}$.
  The two coherence axioms are not quite as immediate. To prove coherence with
  respect to external composition, fix proarrows $x \xproto{m} y \xproto{n} z$
  in $\dbl{C}$ and calculate
  \begingroup
  \allowdisplaybreaks
  \begin{align*}
    \begin{tikzcd}[ampersand replacement=\&,sep=scriptsize]
      GFx \& HFx \& HFy \& HFz \\
      GFx \& GFy \& HFy \& HFz \\
      GFx \& GFy \& GFz \& HFz \\
      GFx \&\& GFz \\
      GFx \&\& GFz \& HFz
      \arrow[""{name=0, anchor=center, inner sep=0}, "GFm"', "\shortmid"{marking}, from=2-1, to=2-2]
      \arrow["{\tau_{Fy}}"', "\shortmid"{marking}, from=2-2, to=2-3]
      \arrow["{\tau_{Fx}}", "\shortmid"{marking}, from=1-1, to=1-2]
      \arrow["HFm", "\shortmid"{marking}, from=1-2, to=1-3]
      \arrow[Rightarrow, no head, from=1-1, to=2-1]
      \arrow[Rightarrow, no head, from=1-3, to=2-3]
      \arrow["{\tau_{Fm}}"{description}, draw=none, from=1-2, to=2-2]
      \arrow[""{name=1, anchor=center, inner sep=0}, "HFn"', "\shortmid"{marking}, from=2-3, to=2-4]
      \arrow[""{name=2, anchor=center, inner sep=0}, "HFn", "\shortmid"{marking}, from=1-3, to=1-4]
      \arrow[Rightarrow, no head, from=1-4, to=2-4]
      \arrow["GFn"', "\shortmid"{marking}, from=3-2, to=3-3]
      \arrow[""{name=3, anchor=center, inner sep=0}, "GFm"', "\shortmid"{marking}, from=3-1, to=3-2]
      \arrow[Rightarrow, no head, from=2-2, to=3-2]
      \arrow[""{name=4, anchor=center, inner sep=0}, "{\tau_{Fz}}"', "\shortmid"{marking}, from=3-3, to=3-4]
      \arrow[Rightarrow, no head, from=2-4, to=3-4]
      \arrow["{\tau_{Fn}}"{description}, draw=none, from=2-3, to=3-3]
      \arrow[Rightarrow, no head, from=2-1, to=3-1]
      \arrow[""{name=5, anchor=center, inner sep=0}, "{G(Fm \odot Fn)}"', "\shortmid"{marking}, from=4-1, to=4-3]
      \arrow[Rightarrow, no head, from=3-1, to=4-1]
      \arrow[Rightarrow, no head, from=3-3, to=4-3]
      \arrow[Rightarrow, no head, from=4-1, to=5-1]
      \arrow[Rightarrow, no head, from=4-3, to=5-3]
      \arrow[Rightarrow, no head, from=3-4, to=5-4]
      \arrow[""{name=6, anchor=center, inner sep=0}, "{\tau_{Fz}}"', "\shortmid"{marking}, from=5-3, to=5-4]
      \arrow[""{name=7, anchor=center, inner sep=0}, "{GF(m \odot n)}"', "\shortmid"{marking}, from=5-1, to=5-3]
      \arrow["1"{description}, draw=none, from=2, to=1]
      \arrow["1"{description}, draw=none, from=0, to=3]
      \arrow["{G_{Fm,Fn}}"{description}, draw=none, from=3-2, to=5]
      \arrow["1"{description}, draw=none, from=4, to=6]
      \arrow["{G(F_{m,n})}"{description, pos=0.6}, draw=none, from=5, to=7]
    \end{tikzcd}
    &=
    \begin{tikzcd}[ampersand replacement=\&,sep=scriptsize]
      GFx \& HFx \& HFy \& HFz \\
      GFx \& HFx \&\& HFz \\
      GFx \&\& GFz \& HFz \\
      GFx \&\& GFz \& HFz
      \arrow[""{name=0, anchor=center, inner sep=0}, "{\tau_{Fx}}", "\shortmid"{marking}, from=1-1, to=1-2]
      \arrow["HFm", "\shortmid"{marking}, from=1-2, to=1-3]
      \arrow["HFn", "\shortmid"{marking}, from=1-3, to=1-4]
      \arrow[""{name=1, anchor=center, inner sep=0}, "{G(Fm \odot Fn)}"', "\shortmid"{marking}, from=3-1, to=3-3]
      \arrow[Rightarrow, no head, from=3-1, to=4-1]
      \arrow[Rightarrow, no head, from=3-3, to=4-3]
      \arrow[""{name=2, anchor=center, inner sep=0}, "{\tau_{Fz}}"', "\shortmid"{marking}, from=4-3, to=4-4]
      \arrow[""{name=3, anchor=center, inner sep=0}, "{GF(m \odot n)}"', "\shortmid"{marking}, from=4-1, to=4-3]
      \arrow[""{name=4, anchor=center, inner sep=0}, "{\tau_{Fx}}"', "\shortmid"{marking}, from=2-1, to=2-2]
      \arrow[""{name=5, anchor=center, inner sep=0}, "{H(Fm \odot Fn)}"', "\shortmid"{marking}, from=2-2, to=2-4]
      \arrow[Rightarrow, no head, from=1-1, to=2-1]
      \arrow[Rightarrow, no head, from=1-2, to=2-2]
      \arrow[Rightarrow, no head, from=2-4, to=3-4]
      \arrow[Rightarrow, no head, from=2-1, to=3-1]
      \arrow[""{name=6, anchor=center, inner sep=0}, "{\tau_{Fz}}", "\shortmid"{marking}, from=3-3, to=3-4]
      \arrow[Rightarrow, no head, from=1-4, to=2-4]
      \arrow[Rightarrow, no head, from=3-4, to=4-4]
      \arrow["{G(F_{m,n})}"{description, pos=0.6}, draw=none, from=1, to=3]
      \arrow["1"{description}, draw=none, from=0, to=4]
      \arrow["1"{description}, draw=none, from=6, to=2]
      \arrow["{\tau_{Fm \odot Fn}}"{description}, draw=none, from=5, to=1]
      \arrow["{H_{Fm,Fn}}"{description, pos=0.4}, draw=none, from=1-3, to=5]
    \end{tikzcd} \\
    &=
    \begin{tikzcd}[ampersand replacement=\&,sep=scriptsize]
      GFx \& HFx \& HFy \& HFz \\
      \& HFx \&\& HFz \\
      GFx \& HFx \&\& HFz \\
      GFx \&\& GFz \& HFz
      \arrow[""{name=0, anchor=center, inner sep=0}, "{\tau_{Fx}}", "\shortmid"{marking}, from=1-1, to=1-2]
      \arrow["HFm", "\shortmid"{marking}, from=1-2, to=1-3]
      \arrow["HFn", "\shortmid"{marking}, from=1-3, to=1-4]
      \arrow["{\tau_{Fz}}"', "\shortmid"{marking}, from=4-3, to=4-4]
      \arrow["{GF(m \odot n)}"', "\shortmid"{marking}, from=4-1, to=4-3]
      \arrow[""{name=1, anchor=center, inner sep=0}, "{H(Fm \odot Fn)}"', "\shortmid"{marking}, from=2-2, to=2-4]
      \arrow[Rightarrow, no head, from=1-2, to=2-2]
      \arrow[Rightarrow, no head, from=1-4, to=2-4]
      \arrow[""{name=2, anchor=center, inner sep=0}, "{HF(m \odot n)}"', "\shortmid"{marking}, from=3-2, to=3-4]
      \arrow[Rightarrow, no head, from=2-2, to=3-2]
      \arrow[Rightarrow, no head, from=2-4, to=3-4]
      \arrow[Rightarrow, no head, from=1-1, to=3-1]
      \arrow[""{name=3, anchor=center, inner sep=0}, "{\tau_{Fx}}", "\shortmid"{marking}, from=3-1, to=3-2]
      \arrow[Rightarrow, no head, from=3-1, to=4-1]
      \arrow[Rightarrow, no head, from=3-4, to=4-4]
      \arrow["{\tau_{F(m \odot n)}}"{description}, draw=none, from=3-2, to=4-3]
      \arrow["{H_{Fm,Fn}}"{description, pos=0.4}, draw=none, from=1-3, to=1]
      \arrow["{H(F_{m,n})}"{description, pos=0.6}, draw=none, from=1, to=2]
      \arrow["1"{description}, draw=none, from=0, to=3]
    \end{tikzcd},
  \end{align*}
  \endgroup
  where the first equation is the coherence of $\tau$ at the proarrows
  $Fx \xproto{Fm} Fy \xproto{Fn} Fz$ in $\dbl{D}$ and the second equation is the
  naturality of $\tau$ with respect to the cell $F_{m,n}$. Coherence with respect
  to external identities is proved similarly.

  Finally, the equivariance axioms of a pre-whiskered modification $\mu*F$ at an
  arrow $f$ or a proarrow $m$ in $\dbl{C}$ follow immediately from those of the
  original modification $\mu$ at $Ff$ or $Fm$, respectively.
\end{proof}

\begin{lemma}[Pre-whiskering is functorial]
  \label{lem:pre-whiskering-functoriality}
  Given double categories $\dbl{D}$ and $\dbl{E}$, pre-whiskering with a lax
  double functor $F: \dbl{C} \to \dbl{D}$ defines a \emph{strict} double functor
  \begin{equation*}
    (-) * F: \Lax_p(\dbl{D},\dbl{E}) \to \Lax_p(\dbl{C},\dbl{E}),
  \end{equation*}
  whenever `$p$' is replaced with any of `lax', `oplax', or `pseudo'.
\end{lemma}
\begin{proof}
  Immediate since $(-) * F$ acts simply by reindexing components and comparisons.
\end{proof}

Turning to post-whiskering, we caution again that protransformations cannot be
post-whiskered with lax double functors, only with pseudo ones. This problem is
well known for (lax, oplax, or pseudo) natural transformations in the
bicategorical setting \cite[\S{11.1}]{johnson2021}; in fact, it is a good reason
to consider double categories in the first place, since ordinary natural
transformations between lax double functors do not suffer from this problem
\cite{shulman2009}.

\begin{construction}[Post-whiskering]
  Let $\tau: F \proTo G: \dbl{C} \to \dbl{D}$ be a lax protransformation between lax
  double functors. The \define{post-whiskering} of $\tau$ with a double functor
  $H: \dbl{D} \to \dbl{E}$ is the lax protransformation
  \begin{equation*}
    H*\tau: H \circ F \proTo H \circ G: \dbl{C} \to \dbl{E}
  \end{equation*}
  with components $(H*\tau)_x \coloneqq H(\tau_x)$ and $(H*\tau)_f \coloneqq H(\tau_f)$ at
  each object $x$ and arrow $f$ in $\dbl{C}$ and with natural comparisons
  \begin{equation*}
    \begin{tikzcd}
      HFx & HGx & HGy \\
      HFx & HFy & HGy
      \arrow["HFm"', "\shortmid"{marking}, from=2-1, to=2-2]
      \arrow["{(H*\tau)_y}"', "\shortmid"{marking}, from=2-2, to=2-3]
      \arrow["{(H*\tau)_x}", "\shortmid"{marking}, from=1-1, to=1-2]
      \arrow["HGm", "\shortmid"{marking}, from=1-2, to=1-3]
      \arrow[Rightarrow, no head, from=1-1, to=2-1]
      \arrow[Rightarrow, no head, from=1-3, to=2-3]
      \arrow["{(H*\tau)_m}"{description}, draw=none, from=1-2, to=2-2]
    \end{tikzcd}
    \quad\coloneqq\quad
    \begin{tikzcd}
      HFx & Gx & Gy \\
      HFx && HGy \\
      HFx && HGy \\
      HFx & HFy & HGy
      \arrow["HFm"', "\shortmid"{marking}, from=4-1, to=4-2]
      \arrow["{H\tau_y}"', "\shortmid"{marking}, from=4-2, to=4-3]
      \arrow["{H\tau_x}", "\shortmid"{marking}, from=1-1, to=1-2]
      \arrow["HGm", "\shortmid"{marking}, from=1-2, to=1-3]
      \arrow[""{name=0, anchor=center, inner sep=0}, "{H(\tau_x \odot Gm)}", "\shortmid"{marking}, from=2-1, to=2-3]
      \arrow[Rightarrow, no head, from=1-1, to=2-1]
      \arrow[Rightarrow, no head, from=1-3, to=2-3]
      \arrow[Rightarrow, no head, from=2-1, to=3-1]
      \arrow[Rightarrow, no head, from=2-3, to=3-3]
      \arrow[""{name=1, anchor=center, inner sep=0}, "{H(Fm \odot \tau_y)}"', "\shortmid"{marking}, from=3-1, to=3-3]
      \arrow[Rightarrow, no head, from=3-1, to=4-1]
      \arrow[Rightarrow, no head, from=3-3, to=4-3]
      \arrow["{H_{\tau_x,Gm}}"{description, pos=0.3}, draw=none, from=1-2, to=0]
      \arrow["{H\tau_m}"{description}, draw=none, from=0, to=1]
      \arrow["{H_{Fm,\tau_y}^{-1}}"{description, pos=0.7}, draw=none, from=1, to=4-2]
    \end{tikzcd}
  \end{equation*}
  for each proarrow $m: x \proto y$ in $\dbl{C}$. Post-whiskering of an oplax or
  pseudo protransformation is defined similarly.

  The \define{post-whiskering} $H*\mu$ of a modification $\mu$ is defined by
  applying $H$ componentwise, so that $(H*\mu)_x \coloneqq H(\mu_x)$ for each
  object $x \in \dbl{C}$.
\end{construction}

We omit the proof that post-whiskerings of protransformations are well-defined,
as it is again essentially the same as in the bicategorical setting
\cite[\mbox{Lemma 11.1.6}]{johnson2021}.

\begin{lemma}[Post-whiskering is functorial]
  \label{lem:post-whiskering-functoriality}
  Given double categories $\dbl{C}$ and $\dbl{D}$, post-whiskering with a
  (pseudo) double functor $H: \dbl{D} \to \dbl{E}$ defines a double functor
  \begin{equation*}
    H * (-): \Lax_p(\dbl{C},\dbl{D}) \to \Lax_p(\dbl{C},\dbl{E}),
  \end{equation*}
  whenever `$p$' is replaced with any of `lax', `oplax', or `pseudo'.
\end{lemma}

\subsection{Companion and conjoint transformations}
\label{sec:companion-transformations}

Among other possible uses, double categories of protransformations are
environments in which to find companions and conjoints of natural
transformations between double functors.

\begin{theorem}[Companions and conjoints of natural transformations]
  \label{thm:companion-transformation}
  A natural transformation $\alpha: F \To G$ between lax double functors
  $F,G: \dbl{D} \to \dbl{E}$ has a companion in $\LaxOpl(\dbl{D},\dbl{E})$, which
  is an oplax protransformation $\alpha_!: F \proTo G$, if and only if each
  component arrow $\alpha_x: Fx \to Gx$ has a companion in $\dbl{E}$.

  Dually, the natural transformation $\alpha: F \To G$ has a conjoint in
  $\LaxLax(\dbl{D},\dbl{E})$, which is a lax protransformation
  $\alpha^*: G \proTo F$, if and only if each component $\alpha_x$ has a conjoint in
  $\dbl{E}$.
\end{theorem}
\begin{proof}
  Supposing each component of $\alpha: F \To G$ has a companion, we define an
  oplax protransformation $\alpha_!: F \proTo G$ as follows. The component of
  $\alpha_!$ at an object $x \in \dbl{D}$ is any choice of companion
  $(\alpha_!)_x \coloneqq (\alpha_x)_!: Fx \proto Gx$ of the corresponding
  component $\alpha_x: Fx \to Gx$ of $\alpha$. The component of $\alpha_!$ at an
  arrow $f: x \to y$ in $\dbl{D}$ is defined by sliding arrows
  (\cref{lem:sliding-general}) in the identity cell
  \begin{equation*}
    \begin{tikzcd}[row sep=scriptsize]
      Fx & Fx \\
      Fy & Gx \\
      Gy & Gy
      \arrow["Ff"', from=1-1, to=2-1]
      \arrow["{\alpha_y}"', from=2-1, to=3-1]
      \arrow[""{name=0, anchor=center, inner sep=0}, "{\id_{Fx}}", "\shortmid"{marking}, from=1-1, to=1-2]
      \arrow["{\alpha_x}", from=1-2, to=2-2]
      \arrow["Gf", from=2-2, to=3-2]
      \arrow[""{name=1, anchor=center, inner sep=0}, "{\id_{Gy}}"', "\shortmid"{marking}, from=3-1, to=3-2]
      \arrow["\id"{description}, draw=none, from=0, to=1]
    \end{tikzcd}
    \qquad\leadsto\qquad
    \begin{tikzcd}
      Fx & Gx \\
      Fy & Gy
      \arrow[""{name=0, anchor=center, inner sep=0}, "{(\alpha_!)_x}", "\shortmid"{marking}, from=1-1, to=1-2]
      \arrow["Ff"', from=1-1, to=2-1]
      \arrow["Gf", from=1-2, to=2-2]
      \arrow[""{name=1, anchor=center, inner sep=0}, "{(\alpha_!)_y}"', "\shortmid"{marking}, from=2-1, to=2-2]
      \arrow["{(\alpha_!)_f}"{description}, draw=none, from=0, to=1]
    \end{tikzcd}
  \end{equation*}
  on the left, induced by the naturality square of $\alpha$ at $f$. The
  naturality comparison of $\alpha_!$ at a proarrow $m: x \proto y$ in $\dbl{D}$
  is defined by sliding arrows (\cref{lem:sliding-globular}) in the component of
  $\alpha$ at $m$:
  \begin{equation} \label{eq:companion-naturality-comparison}
    \begin{tikzcd}
      Fx & Fy \\
      Gx & Gy
      \arrow[""{name=0, anchor=center, inner sep=0}, "Fm", "\shortmid"{marking}, from=1-1, to=1-2]
      \arrow["{\alpha_x}"', from=1-1, to=2-1]
      \arrow["{\alpha_y}", from=1-2, to=2-2]
      \arrow[""{name=1, anchor=center, inner sep=0}, "Gm"', "\shortmid"{marking}, from=2-1, to=2-2]
      \arrow["{\alpha_m}"{description}, draw=none, from=0, to=1]
    \end{tikzcd}
    \qquad\leadsto\qquad
    \begin{tikzcd}
      Fx & Fy & Gy \\
      Fx & Gx & Gy
      \arrow["Fm", "\shortmid"{marking}, from=1-1, to=1-2]
      \arrow["Gm"', "\shortmid"{marking}, from=2-2, to=2-3]
      \arrow["{(\alpha_!)_x}"', "\shortmid"{marking}, from=2-1, to=2-2]
      \arrow["{(\alpha_!)_y}", "\shortmid"{marking}, from=1-2, to=1-3]
      \arrow["{(\alpha_!)_m}"{description}, draw=none, from=1-2, to=2-2]
      \arrow[Rightarrow, no head, from=1-1, to=2-1]
      \arrow[Rightarrow, no head, from=1-3, to=2-3]
    \end{tikzcd}.
  \end{equation}
  We thus have the data of an oplax protransformation $\alpha_!: F \proTo G$.

  The axioms now follow straightforwardly from the functoriality and naturality
  of sliding. First, for any arrows $x \xto{f} y \xto{g} z$ in $\dbl{D}$, the
  equation between external identities
  \begin{equation*}
    \begin{tikzcd}[row sep=scriptsize]
      Fx & Fx \\
      Fz & Gx \\
      Gz & Gz
      \arrow["{F(f \cdot g)}"', from=1-1, to=2-1]
      \arrow["{\alpha_z}"', from=2-1, to=3-1]
      \arrow[""{name=0, anchor=center, inner sep=0}, "{\id_{Fx}}", "\shortmid"{marking}, from=1-1, to=1-2]
      \arrow["{\alpha_x}", from=1-2, to=2-2]
      \arrow["{G(f \cdot g)}", from=2-2, to=3-2]
      \arrow[""{name=1, anchor=center, inner sep=0}, "{\id_{Gz}}"', "\shortmid"{marking}, from=3-1, to=3-2]
      \arrow["\id"{description}, draw=none, from=0, to=1]
    \end{tikzcd}
    \quad=\quad
    \begin{tikzcd}[row sep=scriptsize]
      Fx & Fx & Fx \\
      Fy & Fy & Gx \\
      Fz & Gy & Gy \\
      Gz & Gz & Gz
      \arrow[""{name=0, anchor=center, inner sep=0}, "{\id_{Fx}}", "\shortmid"{marking}, from=1-1, to=1-2]
      \arrow[""{name=1, anchor=center, inner sep=0}, "{\id_{Gz}}"', "\shortmid"{marking}, from=4-1, to=4-2]
      \arrow["Fg"', from=2-1, to=3-1]
      \arrow["{\alpha_z}"', from=3-1, to=4-1]
      \arrow["{\alpha_y}"', from=2-2, to=3-2]
      \arrow["Gg"', from=3-2, to=4-2]
      \arrow[""{name=2, anchor=center, inner sep=0}, "{\id_{Fy}}"', "\shortmid"{marking}, from=2-1, to=2-2]
      \arrow["Ff"', from=1-1, to=2-1]
      \arrow["Ff", from=1-2, to=2-2]
      \arrow[""{name=3, anchor=center, inner sep=0}, "{\id_{Gy}}", "\shortmid"{marking}, from=3-2, to=3-3]
      \arrow[""{name=4, anchor=center, inner sep=0}, "{\id_{Fx}}", "\shortmid"{marking}, from=1-2, to=1-3]
      \arrow["{\alpha_x}", from=1-3, to=2-3]
      \arrow["Gf", from=2-3, to=3-3]
      \arrow["Gg", from=3-3, to=4-3]
      \arrow[""{name=5, anchor=center, inner sep=0}, "{\id_{Gz}}"', "\shortmid"{marking}, from=4-2, to=4-3]
      \arrow["{\id_{Ff}}"{description}, draw=none, from=0, to=2]
      \arrow["\id"{description}, draw=none, from=2, to=1]
      \arrow["\id"{description}, draw=none, from=4, to=3]
      \arrow["{\id_{Gg}}"{description}, draw=none, from=3, to=5]
    \end{tikzcd}
  \end{equation*}
  implies that $(\alpha_!)_{f \cdot g} = (\alpha_!)_f \cdot (\alpha_!)_g$, by
  \cref{lem:sliding-functoriality-internal-special}. Also,
  $(\alpha_!)_{1_x} = 1_{(\alpha_!)_x}$ by the same lemma, proving the functoriality of the
  components of $\alpha_!$. The coherence axioms of $\alpha$ imply those of $\alpha_!$.
  Specifically, for any proarrows $x \xproto{m} y \xproto{n} z$ in $\dbl{D}$,
  the coherence equation
  \begin{equation*}
    \begin{tikzcd}[row sep=scriptsize]
      Fx & Fy & Fz \\
      Gx & Gy & Gz \\
      Gx && Gz
      \arrow[""{name=0, anchor=center, inner sep=0}, "Fm", "\shortmid"{marking}, from=1-1, to=1-2]
      \arrow[""{name=1, anchor=center, inner sep=0}, "Fn", "\shortmid"{marking}, from=1-2, to=1-3]
      \arrow[""{name=2, anchor=center, inner sep=0}, "{G(m \odot n)}"', "\shortmid"{marking}, from=3-1, to=3-3]
      \arrow[Rightarrow, no head, from=2-1, to=3-1]
      \arrow[Rightarrow, no head, from=2-3, to=3-3]
      \arrow[""{name=3, anchor=center, inner sep=0}, "Gm"', "\shortmid"{marking}, from=2-1, to=2-2]
      \arrow[""{name=4, anchor=center, inner sep=0}, "Gn"', "\shortmid"{marking}, from=2-2, to=2-3]
      \arrow["{\alpha_x}"', from=1-1, to=2-1]
      \arrow["{\alpha_z}", from=1-3, to=2-3]
      \arrow["{\alpha_y}"{description}, from=1-2, to=2-2]
      \arrow["{\alpha_m}"{description}, draw=none, from=0, to=3]
      \arrow["{\alpha_n}"{description}, draw=none, from=1, to=4]
      \arrow["{G_{m,n}}"{description}, draw=none, from=2-2, to=2]
    \end{tikzcd}
    \quad=\quad
    \begin{tikzcd}[row sep=scriptsize]
      Fx & Fy & Fz \\
      Fx && Fz \\
      Gx && Gz
      \arrow["Fm", "\shortmid"{marking}, from=1-1, to=1-2]
      \arrow["Fn", "\shortmid"{marking}, from=1-2, to=1-3]
      \arrow[""{name=0, anchor=center, inner sep=0}, "{F(m \odot n)}"', "\shortmid"{marking}, from=2-1, to=2-3]
      \arrow[Rightarrow, no head, from=1-1, to=2-1]
      \arrow[Rightarrow, no head, from=1-3, to=2-3]
      \arrow[""{name=1, anchor=center, inner sep=0}, "{G(m \odot n)}"', "\shortmid"{marking}, from=3-1, to=3-3]
      \arrow["{\alpha_z}"', from=2-3, to=3-3]
      \arrow["{\alpha_x}"', from=2-1, to=3-1]
      \arrow["{F_{m,n}}"{description}, draw=none, from=1-2, to=0]
      \arrow["{\alpha_{m \odot n}}"{description, pos=0.6}, draw=none, from=0, to=1]
    \end{tikzcd}
  \end{equation*}
  induces the corresponding coherence equation
  \begin{equation*}
    \begin{tikzcd}[row sep=scriptsize]
      Fx & Fy & Fz & Gz \\
      Fx & Fy & Gy & Gz \\
      Fx & Gx & Gy & Gz \\
      Fx & Gx && Gz
      \arrow[""{name=0, anchor=center, inner sep=0}, "{G(m \odot n)}"', "\shortmid"{marking}, from=4-2, to=4-4]
      \arrow[Rightarrow, no head, from=3-2, to=4-2]
      \arrow[Rightarrow, no head, from=3-4, to=4-4]
      \arrow["Gm"', "\shortmid"{marking}, from=3-2, to=3-3]
      \arrow[""{name=1, anchor=center, inner sep=0}, "Gn"', "\shortmid"{marking}, from=3-3, to=3-4]
      \arrow[""{name=2, anchor=center, inner sep=0}, "{(\alpha_!)_x}", "\shortmid"{marking}, from=3-1, to=3-2]
      \arrow["{(\alpha_!)_y}"', "\shortmid"{marking}, from=2-2, to=2-3]
      \arrow[Rightarrow, no head, from=2-1, to=3-1]
      \arrow[Rightarrow, no head, from=2-3, to=3-3]
      \arrow[""{name=3, anchor=center, inner sep=0}, "Fm"', "\shortmid"{marking}, from=2-1, to=2-2]
      \arrow[""{name=4, anchor=center, inner sep=0}, "Gn", "\shortmid"{marking}, from=2-3, to=2-4]
      \arrow[Rightarrow, no head, from=2-4, to=3-4]
      \arrow[""{name=5, anchor=center, inner sep=0}, "{(\alpha_!)_x}"', "\shortmid"{marking}, from=4-1, to=4-2]
      \arrow[Rightarrow, no head, from=3-1, to=4-1]
      \arrow["{(\alpha_!)_m}"{description}, draw=none, from=2-2, to=3-2]
      \arrow["{(\alpha_!)_z}", "\shortmid"{marking}, from=1-3, to=1-4]
      \arrow[Rightarrow, no head, from=1-1, to=2-1]
      \arrow[Rightarrow, no head, from=1-2, to=2-2]
      \arrow[Rightarrow, no head, from=1-4, to=2-4]
      \arrow["Fn", "\shortmid"{marking}, from=1-2, to=1-3]
      \arrow[""{name=6, anchor=center, inner sep=0}, "Fm", "\shortmid"{marking}, from=1-1, to=1-2]
      \arrow["{(\alpha_!)_n}"{description}, draw=none, from=1-3, to=2-3]
      \arrow["{G_{m,n}}"{description}, draw=none, from=3-3, to=0]
      \arrow["{1_{(\alpha_!)_x}}"{description}, draw=none, from=2, to=5]
      \arrow["{1_{Gn}}"{description}, draw=none, from=4, to=1]
      \arrow["{1_{Fm}}"{description}, draw=none, from=6, to=3]
    \end{tikzcd}
    \quad=\quad
    \begin{tikzcd}[row sep=scriptsize]
      Fx & Fy & Fz & Gz \\
      Fx && Fz & Gz \\
      Fx & Gx && Gz
      \arrow["Fm", "\shortmid"{marking}, from=1-1, to=1-2]
      \arrow["Fn", "\shortmid"{marking}, from=1-2, to=1-3]
      \arrow[""{name=0, anchor=center, inner sep=0}, "{F(m \odot n)}"', "\shortmid"{marking}, from=2-1, to=2-3]
      \arrow[Rightarrow, no head, from=1-1, to=2-1]
      \arrow[Rightarrow, no head, from=1-3, to=2-3]
      \arrow[""{name=1, anchor=center, inner sep=0}, "{(\alpha_!)_z}"', "\shortmid"{marking}, from=2-3, to=2-4]
      \arrow[Rightarrow, no head, from=2-1, to=3-1]
      \arrow[Rightarrow, no head, from=1-4, to=2-4]
      \arrow[""{name=2, anchor=center, inner sep=0}, "{(\alpha_!)_z}", "\shortmid"{marking}, from=1-3, to=1-4]
      \arrow[""{name=3, anchor=center, inner sep=0}, "{G(m \odot n)}"', "\shortmid"{marking}, from=3-2, to=3-4]
      \arrow[Rightarrow, no head, from=2-4, to=3-4]
      \arrow["{(\alpha_!)_x}"', from=3-1, to=3-2]
      \arrow["{F_{m,n}}"{description}, draw=none, from=1-2, to=0]
      \arrow["{1_{(\alpha_!)_z}}"{description}, draw=none, from=2, to=1]
      \arrow["{(\alpha_!)_{m \odot n}}"{description}, draw=none, from=0, to=3]
    \end{tikzcd}
  \end{equation*}
  by \cref{lem:sliding-functoriality-external-globular,lem:sliding-naturality}. The
  coherence axiom for identity proarrows is proved similarly. To prove the final
  axiom of naturality with respect to cells, let $\stdInlineCell{\gamma}$ be a
  cell in $\dbl{D}$. We have the naturality equation for $\alpha$:
  \begin{equation*}
    \begin{tikzcd}[row sep=scriptsize]
      Fx & Fx & Fy \\
      Fw & Gx & Gy \\
      Gw & Gw & Gz
      \arrow[""{name=0, anchor=center, inner sep=0}, "Fm", "\shortmid"{marking}, from=1-2, to=1-3]
      \arrow[""{name=1, anchor=center, inner sep=0}, "Gm", "\shortmid"{marking}, from=2-2, to=2-3]
      \arrow["{\alpha_x}"', from=1-2, to=2-2]
      \arrow["{\alpha_y}", from=1-3, to=2-3]
      \arrow[""{name=2, anchor=center, inner sep=0}, "Gn"', "\shortmid"{marking}, from=3-2, to=3-3]
      \arrow["Gf"', from=2-2, to=3-2]
      \arrow["Gg", from=2-3, to=3-3]
      \arrow["Ff"', from=1-1, to=2-1]
      \arrow[""{name=3, anchor=center, inner sep=0}, "{\id_{Fx}}", "\shortmid"{marking}, from=1-1, to=1-2]
      \arrow[""{name=4, anchor=center, inner sep=0}, "{\id_{Gw}}"', "\shortmid"{marking}, from=3-1, to=3-2]
      \arrow["{\alpha_w}"', from=2-1, to=3-1]
      \arrow["{\alpha_m}"{description, pos=0.4}, draw=none, from=0, to=1]
      \arrow["G\gamma"{description}, draw=none, from=1, to=2]
      \arrow["\id"{description}, draw=none, from=3, to=4]
    \end{tikzcd}
    \quad=\quad
    \begin{tikzcd}[row sep=scriptsize]
      Fx & Fy & Fy \\
      Fw & Fz & Gy \\
      Gw & Gz & Gz
      \arrow[""{name=0, anchor=center, inner sep=0}, "Fm", "\shortmid"{marking}, from=1-1, to=1-2]
      \arrow["Ff"', from=1-1, to=2-1]
      \arrow["Fg", from=1-2, to=2-2]
      \arrow[""{name=1, anchor=center, inner sep=0}, "Fn", "\shortmid"{marking}, from=2-1, to=2-2]
      \arrow["{\alpha_w}"', from=2-1, to=3-1]
      \arrow["{\alpha_z}", from=2-2, to=3-2]
      \arrow[""{name=2, anchor=center, inner sep=0}, "Gn"', "\shortmid"{marking}, from=3-1, to=3-2]
      \arrow[""{name=3, anchor=center, inner sep=0}, "{\id_{Fy}}", "\shortmid"{marking}, from=1-2, to=1-3]
      \arrow[""{name=4, anchor=center, inner sep=0}, "{\id_{Gz}}"', "\shortmid"{marking}, from=3-2, to=3-3]
      \arrow["{\alpha_y}", from=1-3, to=2-3]
      \arrow["Gg", from=2-3, to=3-3]
      \arrow["F\gamma"{description, pos=0.4}, draw=none, from=0, to=1]
      \arrow["{\alpha_n}"{description}, draw=none, from=1, to=2]
      \arrow["\id"{description}, draw=none, from=3, to=4]
    \end{tikzcd}.
  \end{equation*}
  Compose both sides on the left with $\id_{Ff} \cdot \eta_w$, where $\eta_w$ is
  the unit cell for $(\alpha_w)_!$, and on the right with
  $\varepsilon_y \cdot \id_{Gg}$, where $\varepsilon_y$ is the counit for
  $(\alpha_y)_!$. Also, insert the identity
  $\id_{\alpha_x} = \eta_x \cdot \varepsilon_x$ in the middle of the left-hand
  side and insert the identity $\id_{\alpha_z} = \eta_z \cdot \varepsilon_z$ in
  the middle of the right-hand side. From the correspondence in
  \cref{eq:sliding-companion}, we obtain the naturality equation
  \begin{equation*}
    \begin{tikzcd}[row sep=scriptsize]
      Fx & Fy & Gy \\
      Fx & Gx & Gy \\
      Fw & Gw & Gz
      \arrow["Fm", "\shortmid"{marking}, from=1-1, to=1-2]
      \arrow["{(\alpha_!)_y}", "\shortmid"{marking}, from=1-2, to=1-3]
      \arrow[""{name=0, anchor=center, inner sep=0}, "{(\alpha_!)_x}", "\shortmid"{marking}, from=2-1, to=2-2]
      \arrow[Rightarrow, no head, from=1-1, to=2-1]
      \arrow[Rightarrow, no head, from=1-3, to=2-3]
      \arrow[""{name=1, anchor=center, inner sep=0}, "Gm", "\shortmid"{marking}, from=2-2, to=2-3]
      \arrow["Ff"', from=2-1, to=3-1]
      \arrow["Gf"{description}, from=2-2, to=3-2]
      \arrow[""{name=2, anchor=center, inner sep=0}, "{(\alpha_!)_w}"', "\shortmid"{marking}, from=3-1, to=3-2]
      \arrow["Gg", from=2-3, to=3-3]
      \arrow[""{name=3, anchor=center, inner sep=0}, "Gn"', "\shortmid"{marking}, from=3-2, to=3-3]
      \arrow["{(\alpha_!)_m}"{description}, draw=none, from=1-2, to=2-2]
      \arrow["{(\alpha_!)_f}"{description}, draw=none, from=0, to=2]
      \arrow["G\gamma"{description}, draw=none, from=1, to=3]
    \end{tikzcd}
    \quad=\quad
    \begin{tikzcd}[row sep=scriptsize]
      Fx & Fy & Gy \\
      Fw & Fz & Gz \\
      Fw & Gw & Gz
      \arrow[""{name=0, anchor=center, inner sep=0}, "Fm", "\shortmid"{marking}, from=1-1, to=1-2]
      \arrow[""{name=1, anchor=center, inner sep=0}, "{(\alpha_!)_y}", "\shortmid"{marking}, from=1-2, to=1-3]
      \arrow["{(\alpha_!)_w}"', "\shortmid"{marking}, from=3-1, to=3-2]
      \arrow["Gn"', "\shortmid"{marking}, from=3-2, to=3-3]
      \arrow["Ff"', from=1-1, to=2-1]
      \arrow["Fg"{description}, from=1-2, to=2-2]
      \arrow[""{name=2, anchor=center, inner sep=0}, "Fn"', "\shortmid"{marking}, from=2-1, to=2-2]
      \arrow["Gg", from=1-3, to=2-3]
      \arrow[""{name=3, anchor=center, inner sep=0}, "{(\alpha_!)_z}"', "\shortmid"{marking}, from=2-2, to=2-3]
      \arrow[Rightarrow, no head, from=2-3, to=3-3]
      \arrow[Rightarrow, no head, from=2-1, to=3-1]
      \arrow["{(\alpha_!)_n}"{description}, draw=none, from=2-2, to=3-2]
      \arrow["F\gamma"{description}, draw=none, from=0, to=2]
      \arrow["{(\alpha_!)_g}"{description}, draw=none, from=1, to=3]
    \end{tikzcd}.
  \end{equation*}
  So $\alpha_!$ is a well-defined oplax protransformation.

  The unit and counit for the companion pair $(\alpha, \alpha_!)$ in
  $\LaxOpl(\dbl{D},\dbl{E})$ are the modifications
  \begin{equation*}
    \begin{tikzcd}
      F & F \\
      F & G
      \arrow[""{name=0, anchor=center, inner sep=0}, "{\alpha_!}"', "\shortmid"{marking}, from=2-1, to=2-2]
      \arrow["\alpha", from=1-2, to=2-2]
      \arrow[""{name=1, anchor=center, inner sep=0}, "{\id_F}", "\shortmid"{marking}, from=1-1, to=1-2]
      \arrow[Rightarrow, no head, from=1-1, to=2-1]
      \arrow["\eta"{description}, draw=none, from=1, to=0]
    \end{tikzcd}
    \qquad\text{and}\qquad
    \begin{tikzcd}
      F & G \\
      G & G
      \arrow["\alpha"', from=1-1, to=2-1]
      \arrow[""{name=0, anchor=center, inner sep=0}, "{\alpha_!}", "\shortmid"{marking}, from=1-1, to=1-2]
      \arrow[""{name=1, anchor=center, inner sep=0}, "{\id_G}"', "\shortmid"{marking}, from=2-1, to=2-2]
      \arrow[Rightarrow, no head, from=1-2, to=2-2]
      \arrow["\varepsilon"{description}, draw=none, from=0, to=1]
    \end{tikzcd}
  \end{equation*}
  whose components are the units and counits
  \begin{equation*}
    \begin{tikzcd}
      Fx & Fx \\
      Fx & Gx
      \arrow[Rightarrow, no head, from=1-1, to=2-1]
      \arrow[""{name=0, anchor=center, inner sep=0}, "{\id_{Fx}}", "\shortmid"{marking}, from=1-1, to=1-2]
      \arrow["{\alpha_x}", from=1-2, to=2-2]
      \arrow[""{name=1, anchor=center, inner sep=0}, "{(\alpha_x)_!}"', "\shortmid"{marking}, from=2-1, to=2-2]
      \arrow["{\eta_x}"{description}, draw=none, from=0, to=1]
    \end{tikzcd}
    \qquad\text{and}\qquad
    \begin{tikzcd}
      Fx & Gx \\
      Gx & Gx
      \arrow["{\alpha_x}"', from=1-1, to=2-1]
      \arrow[""{name=0, anchor=center, inner sep=0}, "{(\alpha_x)_!}", "\shortmid"{marking}, from=1-1, to=1-2]
      \arrow[Rightarrow, no head, from=1-2, to=2-2]
      \arrow[""{name=1, anchor=center, inner sep=0}, "{\id_{Gx}}"', "\shortmid"{marking}, from=2-1, to=2-2]
      \arrow["{\varepsilon_x}"{description}, draw=none, from=0, to=1]
    \end{tikzcd}
  \end{equation*}
  for the companion pairs $(\alpha_x, (\alpha_x)_!)$ in $\dbl{E}$, for each
  object $x \in \dbl{D}$. Since modifications compose componentwise, the two
  axioms for $\eta$ and $\varepsilon$ follow immediately from those for each
  $\eta_x$ and $\varepsilon_x$, so we just need to show that the modifications
  are well defined.

  We prove the two equivariance axioms for the modification $\eta$; the proofs
  for $\varepsilon$ are dual. The first equivariance axiom states that for every
  arrow $f: x \to y$ in $\dbl{D}$,
  \begin{equation*}
    \begin{tikzcd}
      Fx & Fx \\
      Fx & Gx \\
      Fy & Gy
      \arrow[Rightarrow, no head, from=1-1, to=2-1]
      \arrow[""{name=0, anchor=center, inner sep=0}, "{\id_{Fx}}", "\shortmid"{marking}, from=1-1, to=1-2]
      \arrow["{\alpha_x}", from=1-2, to=2-2]
      \arrow[""{name=1, anchor=center, inner sep=0}, "{(\alpha_x)_!}", "\shortmid"{marking}, from=2-1, to=2-2]
      \arrow["Gf", from=2-2, to=3-2]
      \arrow["Ff"', from=2-1, to=3-1]
      \arrow[""{name=2, anchor=center, inner sep=0}, "{(\alpha_y)_!}"', "\shortmid"{marking}, from=3-1, to=3-2]
      \arrow["{\eta_x}"{description, pos=0.4}, draw=none, from=0, to=1]
      \arrow["{(\alpha_!)_f}"{description}, draw=none, from=1, to=2]
    \end{tikzcd}
    \quad=\quad
    \begin{tikzcd}
      Fx & Fx \\
      Fy & Fy \\
      Fy & Gy
      \arrow[""{name=0, anchor=center, inner sep=0}, "{\id_{Fx}}", "\shortmid"{marking}, from=1-1, to=1-2]
      \arrow[""{name=1, anchor=center, inner sep=0}, "{\id_{Fy}}", "\shortmid"{marking}, from=2-1, to=2-2]
      \arrow["Ff"', from=1-1, to=2-1]
      \arrow["Ff", from=1-2, to=2-2]
      \arrow["{\alpha_y}", from=2-2, to=3-2]
      \arrow[Rightarrow, no head, from=2-1, to=3-1]
      \arrow[""{name=2, anchor=center, inner sep=0}, "{(\alpha_y)_!}"', "\shortmid"{marking}, from=3-1, to=3-2]
      \arrow["{\eta_y}"{description}, draw=none, from=1, to=2]
      \arrow["{\id_{Ff}}"{description, pos=0.4}, draw=none, from=0, to=1]
    \end{tikzcd}.
  \end{equation*}
  Post-composing both sides with counit $\varepsilon_y$ and using the inverse
  correspondence in \cref{eq:sliding-companion-inv} gives the true equation
  \begin{equation*}
    \begin{tikzcd}[row sep=scriptsize]
      Fx & Fx \\
      Fy & Gx \\
      Gy & Gy
      \arrow["Ff"', from=1-1, to=2-1]
      \arrow["{\alpha_y}"', from=2-1, to=3-1]
      \arrow[""{name=0, anchor=center, inner sep=0}, "{\id_{Fx}}", "\shortmid"{marking}, from=1-1, to=1-2]
      \arrow["{\alpha_x}", from=1-2, to=2-2]
      \arrow["Gf", from=2-2, to=3-2]
      \arrow[""{name=1, anchor=center, inner sep=0}, "{\id_{Gy}}"', "\shortmid"{marking}, from=3-1, to=3-2]
      \arrow["\id"{description}, draw=none, from=0, to=1]
    \end{tikzcd}
    \quad=\quad
    \begin{tikzcd}[row sep=scriptsize]
      Fx & Fx \\
      Fy & Fy \\
      Gy & Gy
      \arrow[""{name=0, anchor=center, inner sep=0}, "{\id_{Fx}}", "\shortmid"{marking}, from=1-1, to=1-2]
      \arrow[""{name=1, anchor=center, inner sep=0}, "{\id_{Fy}}", "\shortmid"{marking}, from=2-1, to=2-2]
      \arrow["Ff"', from=1-1, to=2-1]
      \arrow["Ff", from=1-2, to=2-2]
      \arrow["{\alpha_y}", from=2-2, to=3-2]
      \arrow["{\alpha_y}"', from=2-1, to=3-1]
      \arrow[""{name=2, anchor=center, inner sep=0}, "{\id_{Gy}}"', "\shortmid"{marking}, from=3-1, to=3-2]
      \arrow["{\id_{Ff}}"{description, pos=0.4}, draw=none, from=0, to=1]
      \arrow["{\id_{\alpha_y}}"{description}, draw=none, from=1, to=2]
    \end{tikzcd}.
  \end{equation*}
  Therefore, by the universal property possessed by $\varepsilon_y$ as a
  restriction cell, the original equation holds too. The other equivariance
  axiom states that for every proarrow $m: x \proto y$ in $\dbl{D}$,
  \begin{equation*}
    \begin{tikzcd}
      Fx & Fx & Fy \\
      Fx & Gx & Gy
      \arrow[Rightarrow, no head, from=1-1, to=2-1]
      \arrow[""{name=0, anchor=center, inner sep=0}, "{\id_{Fx}}", "\shortmid"{marking}, from=1-1, to=1-2]
      \arrow[""{name=1, anchor=center, inner sep=0}, "Fm", "\shortmid"{marking}, from=1-2, to=1-3]
      \arrow["{\alpha_x}"{description}, from=1-2, to=2-2]
      \arrow[""{name=2, anchor=center, inner sep=0}, "{(\alpha_x)_!}"', "\shortmid"{marking}, from=2-1, to=2-2]
      \arrow["{\alpha_y}", from=1-3, to=2-3]
      \arrow[""{name=3, anchor=center, inner sep=0}, "Gm"', "\shortmid"{marking}, from=2-2, to=2-3]
      \arrow["{\alpha_m}"{description}, draw=none, from=1, to=3]
      \arrow["{\eta_x}"{description}, draw=none, from=0, to=2]
    \end{tikzcd}
    \quad=\quad
    \begin{tikzcd}[row sep=scriptsize]
      Fx & Fy & Fy \\
      Fx & Fy & Gy \\
      Fx & Gx & Gy
      \arrow[""{name=0, anchor=center, inner sep=0}, "Fm", "\shortmid"{marking}, from=1-1, to=1-2]
      \arrow["{(\alpha_x)_!}"', "\shortmid"{marking}, from=3-1, to=3-2]
      \arrow["Gm"', "\shortmid"{marking}, from=3-2, to=3-3]
      \arrow[""{name=1, anchor=center, inner sep=0}, "{\id_{Fy}}", "\shortmid"{marking}, from=1-2, to=1-3]
      \arrow["{\alpha_y}", from=1-3, to=2-3]
      \arrow[Rightarrow, no head, from=1-2, to=2-2]
      \arrow[Rightarrow, no head, from=1-1, to=2-1]
      \arrow[""{name=2, anchor=center, inner sep=0}, "Fm"', "\shortmid"{marking}, from=2-1, to=2-2]
      \arrow[""{name=3, anchor=center, inner sep=0}, "{(\alpha_y)_!}"', "\shortmid"{marking}, from=2-2, to=2-3]
      \arrow[Rightarrow, no head, from=2-1, to=3-1]
      \arrow[Rightarrow, no head, from=2-3, to=3-3]
      \arrow["{(\alpha_!)_m}"{description}, draw=none, from=2-2, to=3-2]
      \arrow["{\eta_y}"{description}, draw=none, from=1, to=3]
      \arrow["{1_{Fm}}"{description}, draw=none, from=0, to=2]
    \end{tikzcd}.
  \end{equation*}
  Post-composing both sides with the cell $\varepsilon_x \odot 1_{Gm}$ and again
  using the inverse correspondence in \cref{eq:sliding-companion-inv} gives the
  true equation $\alpha_m = \alpha_m$. Since $\varepsilon_x \odot 1_{Gm}$ is a
  restriction cell, namely the restriction of $Gm$ along $\alpha_x$ and
  $1_{Gy}$, the universal property of the restriction implies that the original
  equation holds too.

  This completes the proof that $(\alpha, \alpha_!)$ is a companion pair in
  $\LaxOpl(\dbl{D},\dbl{E})$, establishing the harder direction of the first
  equivalence. Conversely, if $(\alpha, \alpha_!)$ is a companion pair in
  $\LaxOpl(\dbl{D},\dbl{E})$ with binding modifications $\eta$ and
  $\varepsilon$, then for each object $x \in \dbl{D}$, the components $\eta_x$
  and $\varepsilon_x$ are immediately seen to make $(\alpha_x, (\alpha_!)_x)$
  into a companion pair in $\dbl{E}$.

  The second equivalence, about conjoints, is dual in view of \cref{rem:duality}
  and the isomorphism
  $\LaxLax(\dbl{D},\dbl{E})^\rev \cong \LaxOpl(\dbl{D}^\rev, \dbl{E}^\rev)$. For the
  sake of concreteness, we note that if the natural transformation $\alpha: F \To G$
  has components with conjoints in $\dbl{E}$, then the conjoint lax
  protransformation $\alpha^*: G \proTo F$ has components
  $(\alpha^*)_x \coloneqq (\alpha_x)^*: Gx \proto Fx$ at each object $x \in \dbl{D}$ and has
  naturality comparisons
  \begin{equation} \label{eq:conjoint-naturality-comparison}
    \begin{tikzcd}
      Fx & Fy \\
      Gx & Gy
      \arrow[""{name=0, anchor=center, inner sep=0}, "Fm", "\shortmid"{marking}, from=1-1, to=1-2]
      \arrow["{\alpha_x}"', from=1-1, to=2-1]
      \arrow["{\alpha_y}", from=1-2, to=2-2]
      \arrow[""{name=1, anchor=center, inner sep=0}, "Gm"', "\shortmid"{marking}, from=2-1, to=2-2]
      \arrow["{\alpha_m}"{description}, draw=none, from=0, to=1]
    \end{tikzcd}
    \qquad\leadsto\qquad
    \begin{tikzcd}
      Gx & Fx & Fy \\
      Gx & Gy & Fy
      \arrow["Fm", "\shortmid"{marking}, from=1-2, to=1-3]
      \arrow["{\alpha_x^*}", "\shortmid"{marking}, from=1-1, to=1-2]
      \arrow["Gm"', "\shortmid"{marking}, from=2-1, to=2-2]
      \arrow["{\alpha_y^*}"', "\shortmid"{marking}, from=2-2, to=2-3]
      \arrow[Rightarrow, no head, from=1-1, to=2-1]
      \arrow[Rightarrow, no head, from=1-3, to=2-3]
      \arrow["{\alpha_m^*}"{description}, draw=none, from=1-2, to=2-2]
    \end{tikzcd}
  \end{equation}
  induced by \cref{lem:sliding-globular} for each proarrow $m: x \proto y$ in
  $\dbl{D}$.
\end{proof}

\begin{corollary}[Companions and conjoints, pseudo case]
  \label{cor:companion-transformation-pseudo}
  A natural transformation $\alpha: F \To G$ between lax double functors
  $F, G: \dbl{D} \to \dbl{E}$ has a companion in $\LaxPs(\dbl{D},\dbl{E})$, which
  is a (pseudo) protransformation $\alpha_!: F \proTo G$, if and only if each
  component arrow $\alpha_x$ has a companion and each component cell $\alpha_m$ is a
  \emph{commuter} (\cref{def:commuters}).

  Dually, the natural transformation $\alpha: F \To G$ has a conjoint in
  $\LaxPs(\dbl{D},\dbl{E})$, which is a protransformation
  $\alpha^*: G \proTo F$, if and only if each component arrow $\alpha_x$ has a
  conjoint and each component cell $\alpha_m$ is a \emph{cocommuter}.
\end{corollary}
\begin{proof}
  This follows immediately from \cref{thm:companion-transformation} and
  \cref{eq:companion-naturality-comparison,eq:conjoint-naturality-comparison}
  defining the naturality comparisons of the (op)lax protransformations.
\end{proof}

The preceding corollary was stated recently as \cite[\mbox{Proposition
  3.10}]{gambino2022}, with only a sketch of a proof. The next corollary is a
further specialization.

\begin{corollary}[Companions and conjoints of natural isomorphisms]
  \label{cor:companion-natural-isomorphism}
  Suppose $\alpha: F \To G$ is a natural \emph{isomorphism} between lax double
  functors $F, G: \dbl{D} \to \dbl{E}$ such that each component $\alpha_x$ has a
  companion. Then $\alpha$ has a companion $\alpha_!: F \proTo G$ in
  $\LaxPs(\dbl{D},\dbl{E})$. If, in addition, each component $\alpha_x$ has a
  conjoint, then $\alpha$ has a conjoint $\alpha^*: G \proTo F$, and the two
  protransformations form an adjoint equivalence $\alpha_! \dashv \alpha^*$ in
  $\LaxPs(\dbl{D},\dbl{E})$.
\end{corollary}
\begin{proof}
  Since $\alpha$ is a natural isomorphism, each component cell $\alpha_m$ is an
  isomorphism and hence is a commuter by \cref{lem:isos-are-commuters}. Thus,
  $\alpha$ has a companion in $\LaxPs(\dbl{D},\dbl{E})$ by
  \cref{cor:companion-transformation-pseudo}. If, moreover, each component arrow
  $\alpha_x$ has a conjoint, then each component cell $\alpha_m$ is a cocommuter and $\alpha$
  has a conjoint in $\LaxPs(\dbl{D},\dbl{E})$. That $\alpha_!$ and $\alpha^*$ then form an
  adjoint equivalence is a general fact about the companion and the conjoint of
  an isomorphism \cite[\mbox{Lemma 3.21}]{shulman2010}.
\end{proof}

The following lemma about natural transformations is a useful source of commuter
cells. Recall that a lax double functor is \define{normal} if its identity
comparison cells are invertible.

\begin{lemma}[Components at companions are commuters]
  \label{lem:components-at-companions-are-commuters}
  Suppose $\alpha: F \To G$ is a natural transformation between \emph{normal} lax
  double functors $F,G: \dbl{D} \to \dbl{E}$. If $f: x \to y$ is an arrow in
  $\dbl{D}$ with a companion $f_!: x \proto y$, then under sliding
  (\cref{lem:sliding-general}) the component $\alpha_{f_!}$ corresponds to the
  external identity on the naturality square for $f$:
  \begin{equation*}
    \begin{tikzcd}
      Fx & Fy \\
      Gx & Gy
      \arrow[""{name=0, anchor=center, inner sep=0}, "{F(f_!)}", "\shortmid"{marking}, from=1-1, to=1-2]
      \arrow["{\alpha_x}"', from=1-1, to=2-1]
      \arrow["{\alpha_y}", from=1-2, to=2-2]
      \arrow[""{name=1, anchor=center, inner sep=0}, "{G(f_!)}"', "\shortmid"{marking}, from=2-1, to=2-2]
      \arrow["{\alpha_{f_!}}"{description}, draw=none, from=0, to=1]
    \end{tikzcd}
    \qquad\leftrightsquigarrow\qquad
    \begin{tikzcd}[row sep=scriptsize]
      Fx & Fx \\
      Gx & Fy \\
      Gy & Gy
      \arrow["Ff", from=1-2, to=2-2]
      \arrow[""{name=0, anchor=center, inner sep=0}, "{\id_{Fx}}", "\shortmid"{marking}, from=1-1, to=1-2]
      \arrow["{\alpha_y}", from=2-2, to=3-2]
      \arrow["{\alpha_x}"', from=1-1, to=2-1]
      \arrow["Gf"', from=2-1, to=3-1]
      \arrow[""{name=1, anchor=center, inner sep=0}, "{\id_{Gy}}"', "\shortmid"{marking}, from=3-1, to=3-2]
      \arrow["\id"{description}, draw=none, from=0, to=1]
    \end{tikzcd}.
  \end{equation*}
  Moreover, if the components $\alpha_x$ and $\alpha_y$ also have companions, then the
  cell $\alpha_{f_!}$ is a commuter. In fact, the reshaped cell $(\alpha_{f_!})_!$
  (\cref{lem:sliding-globular}) is the canonical isomorphism between companions
  of $Ff \cdot \alpha_y = \alpha_x \cdot Gf$.
\end{lemma}
\begin{proof}
  Only as a convenience, assume that $F$ and $G$ are \define{unitary} lax
  functors, i.e., strictly preserve external identities. So, if $(f, f_!, \eta, \varepsilon)$
  is a companion pair, then its image under $F$ is again a companion pair
  $(Ff, Ff_!, F\eta, F\epsilon)$ without inserting identity comparisons, and likewise for
  the image under $G$. Now calculate
  \begin{equation*}
    \begin{tikzcd}
      Fx & Fx \\
      Fx & Fy \\
      Gx & Gy \\
      Gy & Gy
      \arrow[""{name=0, anchor=center, inner sep=0}, "{F f_!}", "\shortmid"{marking}, from=2-1, to=2-2]
      \arrow["{\alpha_x}"', from=2-1, to=3-1]
      \arrow["{\alpha_y}", from=2-2, to=3-2]
      \arrow[""{name=1, anchor=center, inner sep=0}, "{G f_!}"', "\shortmid"{marking}, from=3-1, to=3-2]
      \arrow[Rightarrow, no head, from=1-1, to=2-1]
      \arrow["Ff", from=1-2, to=2-2]
      \arrow[""{name=2, anchor=center, inner sep=0}, "\shortmid"{marking}, Rightarrow, no head, from=1-1, to=1-2]
      \arrow["Gf"', from=3-1, to=4-1]
      \arrow[Rightarrow, no head, from=3-2, to=4-2]
      \arrow[""{name=3, anchor=center, inner sep=0}, "\shortmid"{marking}, Rightarrow, no head, from=4-1, to=4-2]
      \arrow["{\alpha_{f_!}}"{description}, draw=none, from=0, to=1]
      \arrow["F\eta"{description, pos=0.4}, draw=none, from=2, to=0]
      \arrow["G\varepsilon"{description, pos=0.6}, draw=none, from=1, to=3]
    \end{tikzcd}
    =
    \begin{tikzcd}
      Fx & Fx \\
      Fx & Fy \\
      Fy & Fy \\
      Gy & Gy
      \arrow[""{name=0, anchor=center, inner sep=0}, "{F f_!}", "\shortmid"{marking}, from=2-1, to=2-2]
      \arrow[Rightarrow, no head, from=1-1, to=2-1]
      \arrow["Ff", from=1-2, to=2-2]
      \arrow[""{name=1, anchor=center, inner sep=0}, "\shortmid"{marking}, Rightarrow, no head, from=1-1, to=1-2]
      \arrow[""{name=2, anchor=center, inner sep=0}, "\shortmid"{marking}, Rightarrow, no head, from=4-1, to=4-2]
      \arrow["Ff"', from=2-1, to=3-1]
      \arrow[Rightarrow, no head, from=2-2, to=3-2]
      \arrow[""{name=3, anchor=center, inner sep=0}, "\shortmid"{marking}, Rightarrow, no head, from=3-1, to=3-2]
      \arrow["{\alpha_y}"', from=3-1, to=4-1]
      \arrow["{\alpha_y}", from=3-2, to=4-2]
      \arrow["F\eta"{description, pos=0.4}, draw=none, from=1, to=0]
      \arrow["F\varepsilon"{description}, draw=none, from=0, to=3]
      \arrow["{\alpha_{\id_y}}"{description}, draw=none, from=3, to=2]
    \end{tikzcd}
    =
    \begin{tikzcd}
      Fx & Fx \\
      Fy & Fy \\
      Gy & Gy
      \arrow[""{name=0, anchor=center, inner sep=0}, "\shortmid"{marking}, Rightarrow, no head, from=2-1, to=2-2]
      \arrow["Ff", from=1-2, to=2-2]
      \arrow[""{name=1, anchor=center, inner sep=0}, "\shortmid"{marking}, Rightarrow, no head, from=1-1, to=1-2]
      \arrow[""{name=2, anchor=center, inner sep=0}, "\shortmid"{marking}, Rightarrow, no head, from=3-1, to=3-2]
      \arrow["Ff"', from=1-1, to=2-1]
      \arrow["{\alpha_y}"', from=2-1, to=3-1]
      \arrow["{\alpha_y}", from=2-2, to=3-2]
      \arrow["{F\id_f}"{description}, draw=none, from=1, to=0]
      \arrow["{\alpha_{\id_y}}"{description}, draw=none, from=0, to=2]
    \end{tikzcd}
    =
    \begin{tikzcd}
      Fx & Fx \\
      Fy & Fy \\
      Gy & Gy
      \arrow[""{name=0, anchor=center, inner sep=0}, "\shortmid"{marking}, Rightarrow, no head, from=2-1, to=2-2]
      \arrow["Ff", from=1-2, to=2-2]
      \arrow[""{name=1, anchor=center, inner sep=0}, "\shortmid"{marking}, Rightarrow, no head, from=1-1, to=1-2]
      \arrow[""{name=2, anchor=center, inner sep=0}, "\shortmid"{marking}, Rightarrow, no head, from=3-1, to=3-2]
      \arrow["Ff"', from=1-1, to=2-1]
      \arrow["{\alpha_y}"', from=2-1, to=3-1]
      \arrow["{\alpha_y}", from=2-2, to=3-2]
      \arrow["{\id_{Ff}}"{description}, draw=none, from=1, to=0]
      \arrow["{\id_{\alpha_y}}"{description}, draw=none, from=0, to=2]
    \end{tikzcd}
    = \id_{Ff \cdot \alpha_y}.
  \end{equation*}
  Performing the calculation in the other direction exhibits the right-hand side
  as equal to $\id_{\alpha_x \cdot Gf}$. Either way, we find that $\alpha_{f_!}$
  corresponds under sliding to an external identity.

  Moreover, if $\alpha_x$ and $\alpha_y$ have companions, then the reshaped cell
  $(\alpha_{f_!})_!$ is the canonical isomorphism between companions, as seen by
  comparing with the formula in \cite[\mbox{Lemma 3.8}]{shulman2010}.
\end{proof}

\section{From double categories to bicategories}
\label{sec:double-to-bicat}

Combining results from the previous two sections, we obtain procedures for
transposing structure in double categories. This includes as a special case
transferring structure from a double category to its underlying bicategory. Much
of the work has already been done in \cref{thm:companion-transformation} and its
corollaries by recognizing the companion of a natural transformation between lax
double functors as an oplax or pseudo protransformation. We now simply apply
generalities about companions from \cref{sec:companions} to this situation.

\subsection{Transposing natural transformations}
\label{sec:transposing-transformations}

Given double categories $\dbl{D}$ and $\dbl{E}$, there is a strict double
category $\LaxComp(\dbl{D},\dbl{E})$ having as objects, lax double functors
$\dbl{D} \to \dbl{E}$; as arrows, natural transformations whose component arrows
have companions; as proarrows, arbitrary natural transformations; and as cells
$\inlineCell{F}{G}{H}{K}{\gamma}{\delta}{\alpha}{\beta}{\mu}$, modifications of
the form $\mu: \beta \circ \gamma \Tto \delta \circ \alpha$. Define the strict
double category $\LaxCC(\dbl{D},\dbl{E})$ in the same way, except that its
arrows are natural transformations whose component arrows have companions
\emph{and} whose component cells are commuters.\footnote{Natural transformation
  satisfying both of these conditions are said to have \define{loosely strong
    companions} by Hansen and Shulman \cite[\mbox{Definition
    4.10}]{hansen2019}.}

\begin{theorem}
  \label{thm:transpose-transformation}
  For any double categories $\dbl{D}$ and $\dbl{E}$, there are double functors
  \begin{equation*}
    \LaxComp(\dbl{D},\dbl{E})^\top \to \LaxOpl(\dbl{D},\dbl{E})
    \qquad\text{and}\qquad
    \LaxCC(\dbl{D},\dbl{E})^\top \to \LaxPs(\dbl{D},\dbl{E})
  \end{equation*}
  that act as identities on objects and arrows, and send proarrows (which are
  natural transformations satisfying extra properties) to oplax or pseudo
  protransformations, respectively.
\end{theorem}
\begin{proof}
  Using the biadjunction $(-)^\top \dashv \Comp$ in
  \cref{thm:companions-biadjunction}, the two components of the counit at the
  double categories $\LaxOpl(\dbl{D},\dbl{E})$ and $\LaxPs(\dbl{D},\dbl{E})$ are
  double functors
  \begin{equation*}
    \Comp(\LaxOpl(\dbl{D},\dbl{E}))^\top \to \LaxOpl(\dbl{D},\dbl{E})
    \qquad\text{and}\qquad
    \Comp(\LaxPs(\dbl{D},\dbl{E}))^\top \to \LaxPs(\dbl{D},\dbl{E})
  \end{equation*}
  with the claimed properties. But \cref{thm:companion-transformation} and
  \cref{cor:companion-transformation-pseudo} say precisely that
  \begin{equation*}
    \Comp(\LaxOpl(\dbl{D},\dbl{E})) = \LaxComp(\dbl{D},\dbl{E})
    \qquad\text{and}\qquad
    \Comp(\LaxPs(\dbl{D},\dbl{E})) = \LaxCC(\dbl{D},\dbl{E}).
    \qedhere
  \end{equation*}
\end{proof}

Deducing a result about bicategories is, from a high level, as simple as passing
from double categories to their underlying bicategories. But this procedure is
not without subtleties since, notoriously, there is no 2-category (or
3-category) of bicategories, lax functors, and lax, oplax, or pseudo natural
transformations \cite{shulman2009}. \emph{A fortiori}, there is no 2-category of
double categories, lax functors, and lax, oplax, or pseudo protransformations,
and so there can be no forgetful 2-functor from the latter to the former. Let us
take a closer look at the passage from double categories to bicategories.

\begin{construction}[Underlying bicategory]
  \label{def:underlying-bicategory}
  First of all, there \emph{are} mere categories $\DblLaxOne$ and $\BicatLaxOne$
  that have double categories and bicategories as objects, respectively, and
  have lax functors as morphisms. The \define{underlying bicategory} functor
  \begin{equation*}
    \UBicat: \DblLaxOne \to \BicatLaxOne
  \end{equation*}
  sends
  \begin{itemize}
    \item each (pseudo) double category $\dbl{D}$ to its underlying bicategory
      $\UBicat(\dbl{D})$, comprising the objects, proarrows, and globular cells
      of $\dbl{D}$; and
    \item each lax double functor to the lax functor between underlying
      bicategories that has the same action on objects, proarrows, and globular
      cells and the same comparison cells.
  \end{itemize}
  This forgetful functor restricts on pseudofunctors to a functor
  $\UBicat: \DblOne \to \BicatOne$; see, for instance, \cite[\mbox{Theorem
    4.1}]{shulman2010}.

  Turning to protransformations, notice that the double-categorical notion of a
  lax protransformation generalizes the bicategorical notion of a lax natural
  transformation. To be more precise, when bicategories $\bicat{B}$ and
  $\bicat{C}$ are regarded as double categories with only identity arrows, a lax
  or oplax protransformation (\cref{def:lax-protransformation}) between lax
  functors $F,G: \bicat{B} \to \bicat{C}$ is exactly a lax or oplax natural
  transformation between $F,G$, as defined in bicategory theory
  \cite[\S\S{4.2--4.3}]{johnson2021}. Moreover, a modification between
  protransformations (\cref{def:modification}) with identity source and target
  is precisely a modification in the bicategorical sense
  \cite[\S{4.4}]{johnson2021}.

  Conversely, for fixed double categories $\dbl{D}$ and $\dbl{E}$, there are
  forgetful 2-functors
  \begin{equation*}
    \UBicat: \begin{cases}
      \LaxLaxTwo(\dbl{D},\dbl{E}) \to
        \LaxLaxTwo(\UBicat(\dbl{D}),\UBicat(\dbl{E})) \\
      \LaxOplTwo(\dbl{D},\dbl{E}) \to
        \LaxOplTwo(\UBicat(\dbl{D}),\UBicat(\dbl{E})) \\
      \LaxPsTwo(\dbl{D},\dbl{E}) \to
        \LaxPsTwo(\UBicat(\dbl{D}),\UBicat(\dbl{E}))
    \end{cases}
  \end{equation*}
  from the 2-categories of lax double functors $\dbl{D} \to \dbl{E}$, (lax, oplax,
  or pseudo) protransformations, and modifications to the 2-categories of lax
  functors $\UBicat(\dbl{D}) \to \UBicat(\dbl{E})$, (lax, oplax, or pseudo)
  natural transformations, and modifications. The action on lax functors is as
  above, whereas the action on protransformations is simply to forget the
  component cells, while keeping the component proarrows and the naturality
  comparisons.
\end{construction}

We need a final bit of notation to state the next result. Given double
categories $\dbl{D}$ and $\dbl{E}$, let $\LaxCompTwo(\dbl{D},\dbl{E})$ be the
2-category of lax double functors $\dbl{D} \to \dbl{E}$, natural transformations
whose component arrows have companions, and modifications. Let
$\LaxCCTwo(\dbl{D},\dbl{E})$ be the 2-category with the same objects and cells,
but with morphisms being the natural transformations whose component arrows have
companions \emph{and} whose component cells are commuters.

\begin{corollary}
  For any double categories $\dbl{D}$ and $\dbl{E}$, there are pseudo\-functors
  \begin{equation*}
    \LaxCompTwo(\dbl{D},\dbl{E})^\co \to
      \LaxOplTwo(\UBicat(\dbl{D}),\UBicat(\dbl{E}))
    \qquad\text{and}\qquad
    \LaxCCTwo(\dbl{D},\dbl{E})^\co \to
      \LaxPsTwo(\UBicat(\dbl{D}),\UBicat(\dbl{E}))
  \end{equation*}
  that send
  \begin{itemize}[nosep]
    \item lax double functors $\dbl{D} \to \dbl{E}$ to the lax functors between
      the underlying bicategories;
    \item natural transformations (satisfying extra properties) to oplax or
      pseudo natural transformations, respectively; and
    \item modifications to modifications, reversing the orientation.
  \end{itemize}
\end{corollary}
\begin{proof}
  Applying the forgetful functor $\UBicat: \DblOne \to \BicatOne$ to the double
  functors from \cref{thm:transpose-transformation} gives pseudofunctors
  \begin{equation*}
    \LaxCompTwo(\dbl{D},\dbl{E})^\co \to \LaxOplTwo(\dbl{D},\dbl{E})
    \qquad\text{and}\qquad
    \LaxCCTwo(\dbl{D},\dbl{E})^\co \to \LaxPsTwo(\dbl{D},\dbl{E}).
  \end{equation*}
  To complete the proof, post-compose these with the forgetful 2-functors
  \begin{equation*}
    \LaxOplTwo(\dbl{D},\dbl{E}) \xto{\UBicat}
      \LaxOplTwo(\UBicat(\dbl{D}),\UBicat(\dbl{E}))
    \qquad\text{and}\qquad
    \LaxPsTwo(\dbl{D},\dbl{E}) \xto{\UBicat}
      \LaxPsTwo(\UBicat(\dbl{D}),\UBicat(\dbl{E})).
    \qedhere
  \end{equation*}
\end{proof}

The corollary slightly strengthens Hansen and Shulman's \cite[\mbox{Theorem
  4.6}]{hansen2019}, which is proved directly without passing through
protransformations; see also Shulman's earlier \cite[\mbox{Theorem
  4.6}]{shulman2010}. We think the abstract perspective offered here is valuable
even when the extra flexibility of our \cref{thm:transpose-transformation} is
not needed.

\subsection{Transposing adjunctions}
\label{sec:transposing-adjunctions}

If natural transformations between double functors can be transposed, it stands
to reason that double adjunctions can be too. And they can, but not without a
few subtleties. First, since the companion of a natural transformation is
generally an \emph{oplax} protransformation, we cannot have an ordinary
biadjunction but must have something looser. There are many such notions,
depending on whether the transformations are lax, oplax, or pseudo and whether
the modifications are invertible, but they are all described by the same axioms,
due to Gray \cite[\S{I.7}]{gray1974}. We call all such situations \define{lax
  adjunctions} and will make clear from context what kinds of cells are
involved.

\begin{theorem} \label{thm:transpose-adjunction} Suppose
  $(\eta,\varepsilon): F \dashv G: \dbl{D} \tofrom \dbl{E}$ is a double adjunction such that $F$
  and $G$ are pseudo and the component arrows of the unit and counit
  \begin{equation*}
    \eta: 1_{\dbl{D}} \To G \circ F
    \qquad\text{and}\qquad
    \varepsilon: F \circ G \To 1_{\dbl{E}}
  \end{equation*}
  have companions in $\dbl{D}$ and $\dbl{E}$, respectively. Then this data
  extends to a lax adjunction
  \begin{equation*}
    (\eta_!,\varepsilon_!,s,t): F \dashv G: \dbl{D} \tofrom \dbl{E}
  \end{equation*}
  comprising double functors, oplax protransformations, and invertible
  modifications.

  If, moreover, the component cells of the unit and counit are commuters, then
  there is a biadjunction comprising double functors, protransformations, and
  invertible modfications.
\end{theorem}
\begin{proof}
  By assumption, the triangle identities
  \begin{equation*}
    \begin{tikzcd}
      F & {F \circ G \circ F} \\
      & F
      \arrow["{1_F}"', from=1-1, to=2-2]
      \arrow["{F*\eta}", from=1-1, to=1-2]
      \arrow["{\varepsilon*F}", from=1-2, to=2-2]
    \end{tikzcd}
    \qquad\text{and}\qquad
    \begin{tikzcd}
      G & {G \circ F \circ G} \\
      & G
      \arrow["{\eta*G}", from=1-1, to=1-2]
      \arrow["{G*\varepsilon}", from=1-2, to=2-2]
      \arrow["{1_G}"', from=1-1, to=2-2]
    \end{tikzcd}
  \end{equation*}
  hold on the nose and, by \cref{thm:companion-transformation}, the unit and
  counit have companions $\eta_!$ and $\varepsilon_!$ in
  $\LaxOpl(\dbl{D},\dbl{D})$ and $\LaxOpl(\dbl{E},\dbl{E})$, respectively. Now,
  since double functors preserve companions,
  \cref{lem:pre-whiskering-functoriality} implies that the pre-whiskerings
  $\eta_! * G$ and $\varepsilon_! * F$ of companions are companions of the pre-whiskerings
  $\eta * G$ and $\varepsilon * F$. Similarly, by \cref{lem:post-whiskering-functoriality},
  the post-whiskerings $F * \eta_!$ and $G * \varepsilon_!$ of companions are companions of
  the post-whiskerings $F * \eta$ and $G * \varepsilon$. Thus, there are unique globular
  isomorphisms
  \begin{equation*}
    \begin{tikzcd}
      F & {F \circ G \circ F} \\
      & F
      \arrow[""{name=0, anchor=center, inner sep=0}, "{1_F}"', "\shortmid"{marking}, from=1-1, to=2-2]
      \arrow["{F * \eta_!}", "\shortmid"{marking}, from=1-1, to=1-2]
      \arrow["{\varepsilon_! * F}", "\shortmid"{marking}, from=1-2, to=2-2]
      \arrow["s", shorten <=2pt, Rightarrow, from=0, to=1-2]
    \end{tikzcd}
    \qquad\text{and}\qquad
    \begin{tikzcd}
      G & {G \circ F \circ G} \\
      & G
      \arrow["{\eta_! * G}", "\shortmid"{marking}, from=1-1, to=1-2]
      \arrow["{G * \varepsilon_!}", "\shortmid"{marking}, from=1-2, to=2-2]
      \arrow[""{name=0, anchor=center, inner sep=0}, "{1_G}"', "\shortmid"{marking}, from=1-1, to=2-2]
      \arrow["t"', shorten >=2pt, Rightarrow, from=1-2, to=0]
    \end{tikzcd},
  \end{equation*}
  the \define{triangulators}, witnessing the functoriality and essential
  uniqueness of companions \cite[\mbox{Lemmas 3.8} and
  \mbox{3.12-13}]{shulman2010}.

  It remains to show that the triangulators satisfy the two coherence axioms of
  a lax adjunction \cite[\mbox{Definition I.7.1}]{gray1974}, sometimes called
  the \define{swallowtail identities} \cite{baez2003}. The first of these is
  \begin{equation*}
    \begin{tikzcd}
      {1_{\dbl{D}}} && {G \circ F} \\
      {G \circ F} && {G \circ F \circ G \circ F} \\
      &&& {G \circ F}
      \arrow["{\eta_!}", "\shortmid"{marking}, from=1-1, to=1-3]
      \arrow["{\eta_!}"', "\shortmid"{marking}, from=1-1, to=2-1]
      \arrow["{\eta_!*GF}"', "\shortmid"{marking}, from=2-1, to=2-3]
      \arrow["{GF*\eta_!}"', "\shortmid"{marking}, from=1-3, to=2-3]
      \arrow["{(\eta_!)_{\eta_!}}"', shorten <=23pt, shorten >=23pt, Rightarrow, from=1-3, to=2-1]
      \arrow[""{name=0, anchor=center, inner sep=0}, "1"', "\shortmid"{marking}, curve={height=24pt}, from=2-1, to=3-4]
      \arrow["{G * \varepsilon_! * F}"{description}, from=2-3, to=3-4]
      \arrow[""{name=1, anchor=center, inner sep=0}, "1", "\shortmid"{marking}, curve={height=-30pt}, from=1-3, to=3-4]
      \arrow["{t*F}"', shorten <=8pt, shorten >=8pt, Rightarrow, from=2-3, to=0]
      \arrow["{G*s}", shorten <=8pt, shorten >=8pt, Rightarrow, from=1, to=2-3]
    \end{tikzcd}
    \quad=\quad
    1_{\eta_!},
  \end{equation*}
  which becomes possibly more intelligible when expressed in components as
  \begin{equation*}
    \begin{tikzcd}
      x && GFx \\
      GFx && GFGFx \\
      &&& GFx
      \arrow["{(\eta_x)_!}", "\shortmid"{marking}, from=1-1, to=1-3]
      \arrow["{(\eta_x)_!}"', "\shortmid"{marking}, from=1-1, to=2-1]
      \arrow["{(\eta_{GFx})_!}"', "\shortmid"{marking}, from=2-1, to=2-3]
      \arrow["{GF(\eta_x)_!}"', "\shortmid"{marking}, from=1-3, to=2-3]
      \arrow["{(\eta_{(\eta_x)_!})_!}"', shorten <=20pt, shorten >=20pt, Rightarrow, from=1-3, to=2-1]
      \arrow[""{name=0, anchor=center, inner sep=0}, "1"', "\shortmid"{marking}, curve={height=24pt}, from=2-1, to=3-4]
      \arrow["{G (\varepsilon_{Fx})_!}"{description}, from=2-3, to=3-4]
      \arrow[""{name=1, anchor=center, inner sep=0}, "1", "\shortmid"{marking}, curve={height=-30pt}, from=1-3, to=3-4]
      \arrow["{t_{Fx}}"', shorten <=8pt, shorten >=8pt, Rightarrow, from=2-3, to=0]
      \arrow["{Gs_x}", shorten <=9pt, shorten >=9pt, Rightarrow, from=1, to=2-3]
    \end{tikzcd}
    \quad=\quad
    1_{(\eta_x)_!},
    \qquad x \in \dbl{D}.
  \end{equation*}
  Either by \cref{lem:components-at-companions-are-commuters} or by
  construction, all of these globular cells in $\dbl{E}$ are canonical
  isomorphisms between choices of companions, hence the equation holds by the
  uniqueness of those isomorphisms. The other swallowtail identity is proved
  analogously.

  Under the further assumption that the unit $\eta$ and counit $\varepsilon$
  have component cells that are commuters, their companions $\eta_!$ and
  $\varepsilon_!$ are pseudo protransformations by
  \cref{cor:companion-transformation-pseudo} and we obtain a genuine
  biadjunction.
\end{proof}

\begin{corollary} \label{cor:transpose-adjunction}
  Suppose $(\eta,\varepsilon): F \dashv G: \dbl{D} \tofrom \dbl{E}$ is a double adjunction such
  that $F$ and $G$ are pseudo and $\eta$ and $\varepsilon$ have component arrows with
  companions. Then, denoting the underlying bicategories of $\dbl{D}$ and
  $\dbl{E}$ by $\bicat{D}$ and $\bicat{E}$, there is a lax adjunction
  $(\eta_!,\varepsilon_!,s,t): F \dashv G: \bicat{D} \tofrom \bicat{E}$ comprising pseudofunctors,
  oplax natural transformations, and invertible modifications.

  If, moreover, $\eta$ and $\varepsilon$ have component cells that are commuters, there is a
  biadjunction comprising pseudofunctors, pseudo natural transformations, and
  invertible modifications.
\end{corollary}
\begin{proof}
  Follows from the theorem by passing to bicategories and their cells
  (\cref{def:underlying-bicategory}).
\end{proof}

These results reveal another subtlety about transposing adjunctions. In its most
general form, a double adjunction is between a colax double functor on the left
and a lax double functor on the right \cite{grandis2004,grandis2019}. Yet we
cannot transpose colax-lax or even pseudo-lax double adjunctions, i.e.,
adjunctions in the 2-category $\DblLax$, but must restrict still further to
pseudo-pseudo double adjunctions, i.e., adjunctions in the 2-category $\Dbl$.
The problem is that protransformations cannot be post-whiskered by lax functors,
so the data of a lax adjunction seems to not even make sense in the more general
contexts. Having a simple but flexible theory of two-dimensional adjunctions is
an important virtue of double categories that is largely lost when passing to
bicategories.

\section{From double categories with products to cartesian bicategories}
\label{sec:double-to-cartesian-bicat}

As an application of the theory developed so far, we show how to transpose the
cartesian structure possessed by a double category with products, where we use
``products'' in the generalized sense considered first by Paré \cite{pare2009}
and later by the author \cite{patterson2024}. This line of reasoning will
culminate in a proof that every double category with finite products, and in
particular every cartesian equipment, has an underlying cartesian bicategory
\cite{carboni2008}. In this section, we assume acquaintance with
double-categorical products \cite{patterson2024}. We avoid assuming much about
cartesian bicategories by taking advantage of an alternative axiomatization due
to Trimble \cite{trimble2009}.

\subsection{Structure proarrows and comonoid homomorphisms}
\label{sec:structure-proarrows}

First, we recall our notation for products and the structure maps between them
\cite[\S{8}]{patterson2024}. The product of an $I$-indexed family of objects
$\vec{x}$ in a double category, assuming it exists, is denoted
$\Pi\vec{x} \coloneqq \prod_{i \in I} x_i$. An $I$-indexed family $\vec{x}$ can be
reindexed along any function $f_0: J \to I$, yielding a $J$-indexed family of
objects $f_0^*(\vec{x}) \coloneqq \vec{x} \circ f_0$, comprising the objects
$x_{f_0(j)}$ for each $j \in J$. The universal property of the product
$\Pi f_0^* \vec{x}$ then furnishes a structure map, the unique arrow
$\Pi(f_0)_{\vec{x}}: \Pi\vec{x} \to \Pi f_0^* \vec{x}$ satisfying
$\pi_j \circ \Pi(f_0)_{\vec{x}} = \pi_{f_0(j)}$ for each $j \in J$. Similarly, an
$I$-indexed parallel family of proarrows $\vec{m}: \vec{x} \proto \vec{y}$ has a
reindexing along a function $f_0: J \to I$, which is a $J$-indexed parallel family
of proarrows $f_0^*(\vec{m}): f_0^*(\vec{x}) \proto f_0^*(\vec{y})$. The
universal property of the product again gives a structure map, now a cell of the
form:
\begin{equation*}
  \begin{tikzcd}
    {\Pi\vec{x}} & {\Pi\vec{y}} \\
    {\Pi f_0^* \vec{x}} & {\Pi f_0^* \vec{y}}
    \arrow[""{name=0, anchor=center, inner sep=0}, "{\Pi\vec{m}}", "\shortmid"{marking}, from=1-1, to=1-2]
    \arrow["{\Pi(f_0)_{\vec{x}}}"', from=1-1, to=2-1]
    \arrow[""{name=1, anchor=center, inner sep=0}, "{\Pi f_0^* \vec{m}}"', "\shortmid"{marking}, from=2-1, to=2-2]
    \arrow["{\Pi(f_0)_{\vec{y}}}", from=1-2, to=2-2]
    \arrow["{\Pi(f_0)_{\vec{m}}}"{description}, draw=none, from=0, to=1]
  \end{tikzcd}.
\end{equation*}
Such arrows and cells between products can be transposed, by a straightforward
application of \cref{thm:companion-transformation}.

\begin{proposition}[Structure proarrows between products]
  \label{prop:structure-proarrows}
  Let $\dbl{D}$ be a double category with normal lax (finite) products, and let
  $f_0: J \to I$ be a function between (finite) sets. Then for each $I$-indexed
  family of objects $\vec{x}$ in $\dbl{D}$, the structure arrow
  $\Pi(f_0)_{\vec{x}}: \Pi\vec{x} \to \Pi f_0^*\vec{x}$ has a companion and a conjoint,
  \begin{equation*}
    (\Pi(f_0)_{\vec{x}})_!: \Pi\vec{x} \proto \Pi f_0^* \vec{x}
    \qquad\text{and}\qquad
    (\Pi(f_0)_{\vec{x}})^*: \Pi f_0^* \vec{x} \proto \Pi\vec{x}.
  \end{equation*}
  For each $I$-indexed parallel family of proarrows
  $\vec{m}: \vec{x} \proto \vec{y}$ in $\dbl{D}$, the structure cell
  $\Pi(f_0)_{\vec{m}}$ can be reshaped (via \cref{lem:sliding-globular}) into
  cells
  \begin{equation*}
    \begin{tikzcd}
      {\Pi\vec{x}} & {\Pi\vec{y}} & {\Pi f_0^* \vec{y}} \\
      {\Pi\vec{x}} & {\Pi f_0^* \vec{x}} & {\Pi f_0^* \vec{y}}
      \arrow["{\Pi\vec{m}}", "\shortmid"{marking}, from=1-1, to=1-2]
      \arrow["{\Pi f_0^* \vec{m}}"', "\shortmid"{marking}, from=2-2, to=2-3]
      \arrow["{(\Pi(f_0)_{\vec{x}})_!}"', "\shortmid"{marking}, from=2-1, to=2-2]
      \arrow["{(\Pi(f_0)_{\vec{y}})_!}", "\shortmid"{marking}, from=1-2, to=1-3]
      \arrow[Rightarrow, no head, from=1-1, to=2-1]
      \arrow[Rightarrow, no head, from=1-3, to=2-3]
      \arrow["{(\Pi(f_0)_{\vec{m}})_!}"{description}, draw=none, from=1-2, to=2-2]
    \end{tikzcd}
    \qquad\text{and}\qquad
    \begin{tikzcd}
      {\Pi f_0^* \vec{x}} & {\Pi\vec{x}} & {\Pi\vec{y}} \\
      {\Pi f_0^* \vec{x}} & {\Pi f_0^* \vec{y}} & {\Pi\vec{y}}
      \arrow["{\Pi\vec{m}}", "\shortmid"{marking}, from=1-2, to=1-3]
      \arrow["{\Pi f_0^* \vec{m}}"', "\shortmid"{marking}, from=2-1, to=2-2]
      \arrow["{\Pi(f_0)_{\vec{y}}^*}"', "\shortmid"{marking}, from=2-2, to=2-3]
      \arrow["{\Pi(f_0)_{\vec{x}}^*}", "\shortmid"{marking}, from=1-1, to=1-2]
      \arrow[Rightarrow, no head, from=1-3, to=2-3]
      \arrow["{\Pi(f_0)_{\vec{m}}^*}"{description}, draw=none, from=1-2, to=2-2]
      \arrow[Rightarrow, no head, from=1-1, to=2-1]
    \end{tikzcd},
  \end{equation*}
  which together form the components of oplax and
  lax protransformations
  \begin{equation*}
    \begin{tikzcd}
      {\dbl{D}^I} && {\dbl{D}^J} \\
      & {\dbl{D}}
      \arrow[""{name=0, anchor=center, inner sep=0}, "\Pi"', from=1-1, to=2-2]
      \arrow[""{name=1, anchor=center, inner sep=0}, "\Pi", from=1-3, to=2-2]
      \arrow["{\dbl{D}^{f_0}}", from=1-1, to=1-3]
      \arrow["{\Pi(f_0)_!}", "\shortmid"{marking}, shift left, shorten <=12pt, shorten >=12pt, Rightarrow, from=0, to=1]
    \end{tikzcd}
    \qquad\text{and}\qquad
    \begin{tikzcd}
      {\dbl{D}^I} && {\dbl{D}^J} \\
      & {\dbl{D}}
      \arrow[""{name=0, anchor=center, inner sep=0}, "\Pi"', from=1-1, to=2-2]
      \arrow[""{name=1, anchor=center, inner sep=0}, "\Pi", from=1-3, to=2-2]
      \arrow["{\dbl{D}^{f_0}}", from=1-1, to=1-3]
      \arrow["{\Pi(f_0)^*}"', "\shortmid"{marking}, shift right, shorten <=12pt, shorten >=12pt, Rightarrow, from=1, to=0]
    \end{tikzcd}.
  \end{equation*}
  Moreover, whenever $f_0$ is a bijection, the two protransformations are pseudo
  and form an adjoint equivalence $\Pi(f_0)_! \dashv \Pi(f_0)^*$ in
  $\LaxPs(\dbl{D}^I,\dbl{D})$.
\end{proposition}
\begin{proof}
  Fixing a function $f: J \to I$, the structure arrows and cells,
  $\Pi(f_0)_{\vec{x}}$ and $\Pi(f_0)_{\vec{m}}$, form the components of a natural
  transformation
  \begin{equation*}
    \begin{tikzcd}
      {\dbl{D}^I} && {\dbl{D}^J} \\
      & {\dbl{D}}
      \arrow[""{name=0, anchor=center, inner sep=0}, "\Pi"', from=1-1, to=2-2]
      \arrow[""{name=1, anchor=center, inner sep=0}, "\Pi", from=1-3, to=2-2]
      \arrow["{\dbl{D}^{f_0}}", from=1-1, to=1-3]
      \arrow["{\Pi(f_0)}", shift left, shorten <=12pt, shorten >=12pt, Rightarrow, from=0, to=1]
    \end{tikzcd}
  \end{equation*}
  between lax double functors. Indeed, for any arrows $f: \vec{x} \to \vec{y}$ in
  $\dbl{D}^I$, the naturality square
  \begin{equation*}
    \begin{tikzcd}[column sep=large]
      {\Pi\vec{x}} && {\Pi\vec{y}} \\
      {\Pi f_0^*\vec{x}} && {\Pi f_0^* \vec{y}}
      \arrow["{\Pi(f_0)_{\vec{x}}}"', from=1-1, to=2-1]
      \arrow["{\Pi f \coloneqq \Pi_{i \in I} f_i}", from=1-1, to=1-3]
      \arrow["{\Pi(f_0)_{\vec{y}}}", from=1-3, to=2-3]
      \arrow["{\Pi(f_0,f)}"{description}, dashed, from=1-1, to=2-3]
      \arrow["{\Pi f_0^* f \coloneqq \Pi_{j \in J} f_{f_0(j)}}"', from=2-1, to=2-3]
    \end{tikzcd}
  \end{equation*}
  commutes, its common composite being the unique arrow
  $\Pi(f_0,f): \Pi\vec{x} \to \Pi f_0^* \vec{y}$ in $\dbl{D}$ satisfying
  $\pi_j \circ \Pi(f_0,f) = f_{f_0(j)} \circ \pi_{f_0(j)}$ for each $j \in J$. Naturality with
  respect to cells is perfectly analogous.

  Now, since $\dbl{D}$ has \emph{normal} lax products, each component arrow
  $\Pi(f_0)_{\vec{x}}$ has both a companion and a conjoint \cite[\mbox{Corollary
    8.4}]{patterson2024}. Therefore, by \cref{thm:companion-transformation}, the
  natural transformation $\Pi(f_0)$ itself has a companion and a conjoint, which
  are oplax and lax protransformations, respectively. Moreover, when $f_0$ is a
  bijection, $\Pi(f_0)$ is a natural isomorphism and hence, by
  \cref{cor:companion-natural-isomorphism}, the induced protransformations are
  pseudo and form an adjoint equivalence.
\end{proof}

The proposition has several important special cases, applicable in any double
category $\dbl{D}$ with normal lax finite products.
\begin{itemize}
  \item \emph{Diagonals and codiagonals}: taking the unique function
    $f_0: 2 \xto{!} 1$, the diagonal $\Delta_x: x \to x^2$ at an object $x$ has
    a companion and a conjoint,
    \begin{equation*}
      \delta_x \coloneqq (\Delta_x)_!: x \proto x^2
      \qquad\text{and}\qquad
      \delta_x^* \coloneqq (\Delta_x)^*: x^2 \proto x,
    \end{equation*}
    and the diagonal $\Delta_m: m \to m^2$ at a proarrow $m: x \proto y$ has reshapings
    into cells
    \begin{equation} \label{eq:diagonal-proarrow-comparison}
      \begin{tikzcd}[row sep=scriptsize]
        x & y & {y^2} \\
        x & {x^2} & {y^2}
        \arrow["m", "\shortmid"{marking}, from=1-1, to=1-2]
        \arrow["{\delta_y}", "\shortmid"{marking}, from=1-2, to=1-3]
        \arrow[Rightarrow, no head, from=1-1, to=2-1]
        \arrow[Rightarrow, no head, from=1-3, to=2-3]
        \arrow["{\delta_x}"', "\shortmid"{marking}, from=2-1, to=2-2]
        \arrow["{m^2}"', "\shortmid"{marking}, from=2-2, to=2-3]
        \arrow["{\delta_m}"{description}, draw=none, from=1-2, to=2-2]
      \end{tikzcd}
      \qquad\text{and}\qquad
      \begin{tikzcd}[row sep=scriptsize]
        {x^2} & x & y \\
        {x^2} & {y^2} & y
        \arrow["m", "\shortmid"{marking}, from=1-2, to=1-3]
        \arrow["{m^2}"', "\shortmid"{marking}, from=2-1, to=2-2]
        \arrow["{\delta_x^*}", "\shortmid"{marking}, from=1-1, to=1-2]
        \arrow["{\delta_m^*}"{description}, draw=none, from=1-2, to=2-2]
        \arrow["{\delta_y^*}"', "\shortmid"{marking}, from=2-2, to=2-3]
        \arrow[Rightarrow, no head, from=1-1, to=2-1]
        \arrow[Rightarrow, no head, from=1-3, to=2-3]
      \end{tikzcd},
    \end{equation}
    which together form the components of oplax and lax protransformations
    \begin{equation*}
      \delta: 1_{\dbl{D}} \proTo \times \circ \Delta_{\dbl{D}}
      \qquad\text{and}\qquad
      \delta^*:  \times \circ \Delta_{\dbl{D}} \proTo 1_{\dbl{D}}.
    \end{equation*}
  \item \emph{Deletion and creation}: taking the unique function $f_0: 0 \xto{!} 1$,
    the deletion map $!_x: x \to 1$ at an object $x$ has a companion and a
    conjoint,
    \begin{equation*}
      \varepsilon_x \coloneqq (!_x)_!: x \proto 1
      \qquad\text{and}\qquad
      \varepsilon_x^* \coloneqq (!_x)^*: 1 \proto x,
    \end{equation*}
    and the deletion cell $!_m: m \to 1$ at a proarrow $m: x \proto y$ has
    reshapings into cells
    \begin{equation}  \label{eq:deletion-proarrow-comparison}
      \begin{tikzcd}[row sep=scriptsize]
        x & y & 1 \\
        x && 1
        \arrow["m", "\shortmid"{marking}, from=1-1, to=1-2]
        \arrow["{\varepsilon_y}", "\shortmid"{marking}, from=1-2, to=1-3]
        \arrow[Rightarrow, no head, from=1-1, to=2-1]
        \arrow[""{name=0, anchor=center, inner sep=0}, "{\varepsilon_x}"', "\shortmid"{marking}, from=2-1, to=2-3]
        \arrow[Rightarrow, no head, from=1-3, to=2-3]
        \arrow["{\varepsilon_m}"{description, pos=0.4}, draw=none, from=1-2, to=0]
      \end{tikzcd}
      \qquad\text{and}\qquad
      \begin{tikzcd}[row sep=scriptsize]
        1 & x & y \\
        1 && y
        \arrow["{\varepsilon_x^*}", "\shortmid"{marking}, from=1-1, to=1-2]
        \arrow[""{name=0, anchor=center, inner sep=0}, "{\varepsilon_y^*}"', "\shortmid"{marking}, from=2-1, to=2-3]
        \arrow[Rightarrow, no head, from=1-3, to=2-3]
        \arrow[Rightarrow, no head, from=1-1, to=2-1]
        \arrow["m", from=1-2, to=1-3]
        \arrow["{\varepsilon_m^*}"{description, pos=0.4}, draw=none, from=1-2, to=0]
      \end{tikzcd},
    \end{equation}
    which together form the components of oplax and lax protransformations
    \begin{equation*}
      \varepsilon: 1_{\dbl{D}} \proTo {I_{\dbl{D}}} \circ {!_{\dbl{D}}}
      \qquad\text{and}\qquad
      \varepsilon^*: {I_{\dbl{D}}} \circ {!_{\dbl{D}}} \proTo 1_{\dbl{D}}.
    \end{equation*}
  \item \emph{Symmetries}: taking the swap function $f_0: 2 \xto{\cong} 2$, the symmetry
    isomorphism $\sigma_{x,x'}: x \times x' \to x' \times x$ at a pair of
    objects $x$ and $x'$ has a companion and a conjoint,
    \begin{equation*}
      \tau_{x,x'} \coloneqq (\sigma_{x,x'})_!: x \times x' \proto x' \times x
      \qquad\text{and}\qquad
      \tau_{x,x'}^* \coloneqq (\sigma_{x,x'})^*: x' \times x \proto x \times x',
    \end{equation*}
    and the symmetry isomorphism $\sigma_{m,m'}: m \times m' \to m' \times m$ at a pair of
    proarrows $m: x \proto y$ and $m': x' \proto y'$ has reshapings into
    invertible cells
    \begin{equation*}
      \begin{tikzcd}[row sep=scriptsize]
        {x \times x'} & {y \times y'} & {y' \times y} \\
        {x \times x'} & {x' \times x} & {y' \times y}
        \arrow["{m \times m'}", "\shortmid"{marking}, from=1-1, to=1-2]
        \arrow["{\tau_{y,y'}}", "\shortmid"{marking}, from=1-2, to=1-3]
        \arrow[Rightarrow, no head, from=1-1, to=2-1]
        \arrow[Rightarrow, no head, from=1-3, to=2-3]
        \arrow["{\tau_{x,x'}}"', "\shortmid"{marking}, from=2-1, to=2-2]
        \arrow["{m' \times m}"', "\shortmid"{marking}, from=2-2, to=2-3]
        \arrow["{\tau_{m,m'}}"{description}, draw=none, from=1-2, to=2-2]
      \end{tikzcd}
      \qquad\text{and}\qquad
      \begin{tikzcd}
        {x' \times x} & {x \times x'} & {y \times y'} \\
        {x' \times x} & {y' \times y} & {y \times y'}
        \arrow["{m \times m'}", "\shortmid"{marking}, from=1-2, to=1-3]
        \arrow["{m' \times m}"', "\shortmid"{marking}, from=2-1, to=2-2]
        \arrow["{\tau_{x,x'}^*}", "\shortmid"{marking}, from=1-1, to=1-2]
        \arrow["{\tau_{m,m'}^*}"{description}, draw=none, from=1-2, to=2-2]
        \arrow[Rightarrow, no head, from=1-3, to=2-3]
        \arrow[Rightarrow, no head, from=1-1, to=2-1]
        \arrow["{\tau_{y,y'}^*}"', "\shortmid"{marking}, from=2-2, to=2-3]
      \end{tikzcd}.
    \end{equation*}
    These are the components of two protransformations, forming an adjoint
    equivalence
    \begin{equation*}
      \begin{tikzcd}
        \times & {{\times} \circ {\sigma_{\dbl{D},\dbl{D}}}}
        \arrow[""{name=0, anchor=center, inner sep=0}, "\tau", "\shortmid"{marking}, curve={height=-24pt}, Rightarrow, from=1-1, to=1-2]
        \arrow[""{name=1, anchor=center, inner sep=0}, "{\tau^*}", "\shortmid"{marking}, curve={height=-24pt}, Rightarrow, from=1-2, to=1-1]
        \arrow["\dashv"{anchor=center, rotate=-90}, draw=none, from=0, to=1]
      \end{tikzcd}
      \qquad\text{in}\qquad \LaxPs(\dbl{D}^2,\dbl{D}).
    \end{equation*}
\end{itemize}

By combining these special cases, we obtain the comonoid and monoid structures
that play such a central role in Carboni and Walters' original axioms for a
(locally posetal) cartesian bicategory \cite{carboni1987}. Any object $x$ in a
double category $\dbl{D}$ with lax finite products canonically has the structure
of a commutative comonoid $(x,\Delta_x,!_x)$ in $\dbl{D}_0$. When $\dbl{D}$ has
\emph{normal} lax finite products, these can be transposed into a commutative
comonoid $(x,\delta_x,\varepsilon_x)$ and a commutative monoid $(x,\delta_x^*,\varepsilon_x^*)$ in the
underlying bicategory $\UBicat(\dbl{D})$, where the commutative (co)monoid laws
now hold only up to isomorphism, by the pseudofunctoriality of companions and
conjoints. The comonoid and monoid structures are also adjoint in the sense that
$\delta_x \dashv \delta_x^*$ and $\varepsilon_x \dashv \varepsilon_x^*$ in $\UBicat(\dbl{D})$. Finally, every proarrow
$m$ in $\dbl{D}$ is automatically an ``oplax comonoid homomorphism,'' witnessed by
the comparison cells $\delta_m$ and $\varepsilon_m$, and a ``lax monoid homomorphism,'' witnessed
by the comparison cells $\delta_m^*$ and $\varepsilon_m^*$. In the locally posetal case, these
are essentially the axioms of a cartesian bicategory \cite[\mbox{Definition
  1.2}]{carboni1987}.\footnote{Such observations are easily cast into the form
  of a theorem but do not on their own suffice to give a cartesian bicategory in
  the general case. For now, we focus on a precise treatment of comonoid and
  monoid homomorphisms in a double category with products, an interesting topic
  in its own right.}

\begin{definition}[Comonoid homomorphisms]
  A proarrow $m$ in a double category with normal lax finite products is a
  \define{comonoid homomorphism} if the cells \mbox{$\Delta_m: m \to m^2$} and
  \mbox{$!_m: m \to 1$} are commuters, i.e., the globular cells $\delta_m$ and $\varepsilon_m$ in
  \cref{eq:diagonal-proarrow-comparison,eq:deletion-proarrow-comparison} are
  invertible.

  Dually, a proarrow $m$ is a \define{monoid homomorphism} if the cells $\Delta_m$
  and $!_m$ are cocommuters, i.e., the globular cells $\delta_m^*$ and $\varepsilon_m^*$ are
  invertible.
\end{definition}

\begin{proposition}[Companions are comonoid homomorphisms]
  Any companion of an arrow in a double category with normal lax finite products
  is a comonoid homomorphism, and any conjoint is a monoid homomorphism.
\end{proposition}
\begin{proof}
  Let $\dbl{D}$ be a double category with normal lax finite products. Taking the
  functions $f_0: 2 \xto{!} 1$ and $g_0: 0 \xto{!} 1$ in the proof of
  \cref{prop:structure-proarrows} gives natural transformations
  \begin{equation*}
    \Delta \coloneqq \Pi(f_0): 1_{\dbl{D}} \To \times \circ \Delta_{\dbl{D}}
    \qquad\text{and}\qquad
    ! \coloneqq \Pi(g_0): 1_{\dbl{D}} \To 1 \circ {!_{\dbl{D}}}
  \end{equation*}
  between normal lax functors. Now apply
  \cref{lem:components-at-companions-are-commuters} and its dual for conjoints.
\end{proof}

\begin{remark}[Comonoid homomorphism not a companion]
  The converse is false: a comonoid homomorphism in a double category with
  products need not be a companion, nor a map (in the sense recalled in
  \cref{def:map} below). A counterexample in the double category of
  boolean-valued profunctors, which is even locally posetal, has been given by
  Todd Trimble \cite{trimble2023mo}.
\end{remark}

Commutative comonoids and comonoid homomorphisms, which figure so prominently in
the original axioms for a locally posetal cartesian bicategory
\cite{carboni1987}, make no appearance in the axioms for a general cartesian
bicategory proposed much later \cite{carboni2008}. The central concept is now
that of a \emph{map}.

\begin{definition}[Map] \label{def:map}
  A morphism in a bicategory is a \define{map} if it has a right adjoint.
  Similarly, a proarrow in a double category $\dbl{D}$ is \define{map} if it is
  a map in the underlying bicategory of $\dbl{D}$.
\end{definition}

A companion of an arrow in a double category is a map whenever the arrow also
has a conjoint, in which case the conjoint is right adjoint to the companion
(\cref{rem:mates}). Again, the converse is not true: a double category can
possess maps that are not companions. In the double category of profunctors, the
statement that a profunctor $\cat{C} \proto \cat{D}$ has a right adjoint if and
only if it is a companion of a functor $\cat{C} \to \cat{D}$ is equivalent to the
codomain $\cat{D}$ being Cauchy complete \cite[\mbox{Volume 1}, \mbox{Theorem
  7.9.3}]{borceux1994}.

\begin{remark}[Maps as structure versus property]
  In a few important double categories with products, such as $\Span$ and
  $\Rel$, companions, maps, and comonoid homomorphisms all coincide and capture
  the expected notion of ``function-like'' proarrow. This elegant equivalence
  motivates the theory of cartesian bicategories. However, the equivalence is a
  rather special situation. The facts cited above imply that in the double
  category $\Prof$, the three classes of proarrows---companions, maps, and
  comonoid homomorphisms---separate, with only the first exactly capturing the
  profunctors isomorphic to functors. Apparently, it is more robust to treat the
  ``function-like'' morphisms not as ``relation-like'' morphisms satisfying
  special properties but rather as extra structure. That is, of course, what
  happens in a double category.
\end{remark}

\subsection{Transposing cartesian structure}
\label{sec:transposing-cartesian-structure}

To reduce complications with coherence, Carboni, Kelly, Walters, and Wood define
a cartesian bicategory in two stages, taking advantage of bicategorical
universal properties in the first stage. In outline, a bicategory $\bicat{B}$ is
\define{precartesian} when $\bicat{B}$ has finite local products and
$\Map\bicat{B}$, the locally full sub-bicategory of maps, has finite products
(in the sense of bilimits) \cite[\mbox{Definition 3.1}]{carboni2008}. Using
these universal properties, lax functors $\otimes: \bicat{B}^2 \to \bicat{B}$ and
$I: \bicat{1} \to \bicat{B}$ are constructed \cite[\mbox{Theorem
  3.15}]{carboni2008}. Finally, a precartesian bicategory $\bicat{B}$ is defined
to be \define{cartesian} when these lax functors are pseudo
\cite[\mbox{Definition 4.1}]{carboni2008}. This definition, while indirect, has
the advantage of making it immediately clear that being cartesian is a
\emph{property} of a bicategory.

Nevertheless, we find it more convenient to use another definition, due to Todd
Trimble \cite{trimble2009}, of a cartesian \emph{structure} on a bicategory,
closer in spirit to the original Carboni-Walters definition in the locally
posetal case \cite{carboni1987}. On the cited nLab page, Trimble sketches a
proof that such a cartesian structure is essentially unique, and in particular
recovers the universal properties of a cartesian bicategory à la Carboni et al.
When we say ``cartesian bicategory,'' we will mean Trimble's notion; however,
since the equivalence with the original definition has not appeared in print,
the scrupulous reader may prefer to read this as ``Trimble-cartesian
bicategory.''

\begin{definition}[Cartesian bicategory à la Trimble]
  \label{def:cartesian-bicategory}
  A \define{cartesian structure} on a bicategory $\bicat{B}$ consists of:
  \begin{enumerate}[(i),noitemsep]
    \item pseudofunctors $\otimes: \bicat{B}^2 \to \bicat{B}$ and
      $I: \bicat{1} \to \bicat{B}$;
    \item oplax natural transformations
      \begin{equation*}
        \begin{tikzcd}[column sep=small]
          {\bicat{B}} && {\bicat{B}^2} \\
          & {\bicat{B}}
          \arrow["{\Delta_{\bicat{B}}}", from=1-1, to=1-3]
          \arrow[""{name=0, anchor=center, inner sep=0}, "\otimes", from=1-3, to=2-2]
          \arrow[""{name=1, anchor=center, inner sep=0}, "{1_{\bicat{B}}}"', from=1-1, to=2-2]
          \arrow["\delta", shorten <=9pt, shorten >=9pt, Rightarrow, from=1, to=0]
        \end{tikzcd}
        \qquad\qquad
        \begin{tikzcd}[column sep=small]
          {\bicat{B}^2} && {\bicat{B}} \\
          & {\bicat{B}^2}
          \arrow["\otimes", from=1-1, to=1-3]
          \arrow[""{name=0, anchor=center, inner sep=0}, "{1_{\bicat{B}^2}}"', from=1-1, to=2-2]
          \arrow[""{name=1, anchor=center, inner sep=0}, "{\Delta_{\bicat{B}}}", from=1-3, to=2-2]
          \arrow["\pi"', shorten <=9pt, shorten >=9pt, Rightarrow, from=1, to=0]
        \end{tikzcd}
        \qquad\qquad
        \begin{tikzcd}[column sep=small]
          {\bicat{B}} && {\bicat{1}} \\
          & {\bicat{B}}
          \arrow[""{name=0, anchor=center, inner sep=0}, "{1_{\bicat{B}}}"', from=1-1, to=2-2]
          \arrow["{!_{\bicat{B}}}", from=1-1, to=1-3]
          \arrow[""{name=1, anchor=center, inner sep=0}, "I", from=1-3, to=2-2]
          \arrow["\varepsilon", shorten <=9pt, shorten >=9pt, Rightarrow, from=0, to=1]
        \end{tikzcd}
      \end{equation*}
      that are \define{map-valued}, meaning that their components are maps;
    \item invertible modifications
      \begin{equation*}
        \begin{tikzcd}
          {\Delta_{\bicat{B}}} & {\Delta_{\bicat{B}} \circ \otimes \circ \Delta_{\bicat{B}}} \\
          & {\Delta_{\bicat{B}}}
          \arrow[""{name=0, anchor=center, inner sep=0}, "{1_{\Delta_{\bicat{B}}}}"', from=1-1, to=2-2]
          \arrow["{\Delta_{\bicat{B}} * \delta}", from=1-1, to=1-2]
          \arrow["{\pi * \Delta_{\bicat{B}}}", from=1-2, to=2-2]
          \arrow["s", shorten <=5pt, shorten >=5pt, Rightarrow, from=0, to=1-2]
        \end{tikzcd}
        \qquad\qquad
        \begin{tikzcd}
          \otimes & {\otimes \circ \Delta_{\bicat{B}} \circ \otimes} \\
          & \otimes
          \arrow[""{name=0, anchor=center, inner sep=0}, "{1_{\otimes}}"', from=1-1, to=2-2]
          \arrow["{\delta * \otimes}", from=1-1, to=1-2]
          \arrow["{\otimes * \pi}", from=1-2, to=2-2]
          \arrow["t"', shorten <=4pt, shorten >=4pt, Rightarrow, from=1-2, to=0]
        \end{tikzcd}
        \qquad\qquad
        \begin{tikzcd}
          I & {I \circ !_{\bicat{B}} \circ I} \\
          & I
          \arrow["{\varepsilon * I}", from=1-1, to=1-2]
          \arrow["{I*1_{1_{\bicat{B}}}}", from=1-2, to=2-2]
          \arrow[""{name=0, anchor=center, inner sep=0}, "{1_I}"', from=1-1, to=2-2]
          \arrow["u"', shorten <=3pt, shorten >=3pt, Rightarrow, from=1-2, to=0]
        \end{tikzcd},
      \end{equation*}
      the \define{triangulators}, constituting lax adjunctions
      \begin{equation*}
        (\delta, \pi, s, t): \Delta_{\bicat{B}} \dashv \otimes:
          \bicat{B} \tofrom \bicat{B}^2
        \qquad\text{and}\qquad
        (\varepsilon, 1_{1_{\bicat{1}}}, 1_{1_{!_{\bicat{B}}}}, u):\;
          !_{\bicat{B}} \dashv I: \bicat{1} \tofrom \bicat{B}.
      \end{equation*}
  \end{enumerate}
\end{definition}

We can now state and prove the main result of this section. The ``iso-strong''
condition on double products essentially says that parallel products commute
with external composition up to isomorphism, as in a cartesian double category;
for details, see \cite[\S{7}]{patterson2024}.

\begin{theorem}[Underlying cartesian bicategory]
  \label{thm:underlying-cartesian-bicategory}
  The underlying bicategory of a double category with iso-strong finite products
  is cartesian.
\end{theorem}
\begin{proof}
  Let $\dbl{D}$ be a double category with iso-strong finite products. We will
  exhibit a cartesian structure on the underlying bicategory of $\dbl{D}$, a
  structure that is essentially unique \cite{trimble2009}.
  \begin{enumerate}[(i)]
    \item The iso-strong double products in $\dbl{D}$ restrict to double functors
      $\times: \dbl{D}^2 \to \dbl{D}$ and $1: \dbl{1} \to \dbl{D}$. Underlying these are
      pseudofunctors $\otimes: \bicat{D}^2 \to \bicat{D}$ and
      $I: \bicat{1} \to \bicat{D}$, where we write $\bicat{D}$ for the underlying
      bicategory of $\dbl{D}$.

    \item As already observed, taking the functions $f_0: 2 \xto{!} 1$ and
      $g_0 : 0 \xto{!} 1$ in \cref{prop:structure-proarrows} gives
      natural transformations
      \begin{equation*}
        \Delta \coloneqq \Pi(f_0): 1_{\dbl{D}} \To \times \circ \Delta_{\dbl{D}}
        \qquad\text{and}\qquad
        ! \coloneqq \Pi(g_0): 1_{\dbl{D}} \To 1 \circ {!_{\dbl{D}}}
      \end{equation*}
      that have companions, which are map-valued, oplax protransformations
      \begin{equation*}
        \delta \coloneqq \Pi(f_0)_!: 1_{\dbl{D}} \proTo \times \circ \Delta_{\dbl{D}}
        \qquad\text{and}\qquad
        \varepsilon \coloneqq \Pi(g_0)_!: 1_{\dbl{D}} \proTo 1 \circ {!_{\dbl{D}}}.
      \end{equation*}
      Similarly, taking each of the two inclusions $\iota_1, \iota_2: 1 \hookrightarrow 2$, we form
      the natural transformation
      \begin{equation*}
        \Pi \coloneqq (\Pi^1, \Pi^2) \coloneqq (\Pi(\iota_1), \Pi(\iota_2)):
          \Delta_{\dbl{D}} \circ \times \To 1_{\dbl{D}^2}
      \end{equation*}
      whose components are the pairs of projections
      \begin{equation*}
        \Pi_{(x,y)} \coloneqq (\Pi^1_{x,y}, \Pi^2_{x,y}):
          (x \times y, x \times y) \to (x,y),
        \qquad x,y \in \dbl{D}.
      \end{equation*}
      It too has a companion, a map-valued, oplax protransformation
      $\pi \coloneqq \Pi_!: \Delta_{\dbl{D}} \circ \times \proTo 1_{\dbl{D}^2}$. Underlying $\delta$,
      $\varepsilon$, and $\pi$ are the map-valued, oplax natural transformations that we
      need.

    \item Since the double category $\dbl{D}$ has iso-strong finite products, it is,
      in particular, a cartesian double category \cite[\S{7}]{patterson2024}. By
      the definition of the latter \cite{aleiferi2018}, there are double
      adjunctions
      \begin{equation*}
        (\Delta, \Pi): \Delta_{\dbl{D}} \dashv \times: \dbl{D} \tofrom \dbl{D}^2
        \qquad\text{and}\qquad
        (!, 1_{1_{\dbl{1}}}): {!_{\dbl{D}}} \dashv 1: \dbl{D} \tofrom \dbl{1}.
      \end{equation*}
      Applying \cref{cor:transpose-adjunction} to these double adjunctions
      completes the proof. \qedhere
    \end{enumerate}
\end{proof}

\begin{corollary} \label{cor:underlying-cartesian-bicategory}
  The underlying bicategory of a cartesian equipment is cartesian.
\end{corollary}
\begin{proof}
  This follows from the theorem because cartesian equipments have iso-strong
  finite products \cite[\mbox{Corollary 8.7}]{patterson2024}.
\end{proof}

Nearly all of the cartesian bicategories conceived by Carboni, Kelly, Walters,
and Wood \cite[\mbox{Example 3.2}]{carboni2008}, including the prototypical
bicategories of spans and of profunctors, are known to be cartesian equipments
and thus inherit their cartesian structure from this corollary.\footnote{The
  only exception in the list from \cite[\mbox{Example 3.2}]{carboni2008} is the
  2-category of categories with finite products, finite-product-preserving
  functors, and natural transformations. However, this example is better viewed
  as a 2-category than a bicategory. Indeed, it has finite products in the
  strict sense, by general results about algebras of 2-monads
  \cite[\S\S{6.1--6.2}]{lack2010companion}.} Via a more abstract method of
proof, our result strengthens Lambert's result that \emph{locally posetal}
cartesian equipments have underlying cartesian bicategories
\cite[\mbox{Proposition 3.1}]{lambert2022}. An earlier related result is
Lawler's equivalence \cite[\mbox{Proposition 4.2.4}]{lawler2015} between
cartesian bicategories and a special type of cartesian equipment, called
\emph{chordate} after \cite[\mbox{Example 3.2}]{lack2012}.

\subsection{Transposing cocartesian structure}
\label{sec:transposing-cocartesian-structure}

We turn now from cartesian to cocartesian structure. Of course, they are
formally dual, but due to asymmetries in the main examples, namely that they
tend to have strong coproducts but only iso-strong products, we can obtain
stronger results about cocartesian structure. Loosely speaking, a double
category has \emph{strong} coproducts if arbitrary span-indexed coproducts of
proarrows exist and commute with external composition; for details, see
\cite[\S{4}]{patterson2024}.

The following proposition dualizes \cref{prop:structure-proarrows} while
strengthening its hypotheses and conclusions to account for strong coproducts.

\begin{proposition}[Structure maps between strong coproducts]
  \label{prop:structure-proarrows-coproducts}
  Let $\dbl{D}$ be a double category with strong (finite) coproducts. For any
  function $f_0: I \to J$ between (finite) sets, the natural transformation
  \begin{equation*}
    \begin{tikzcd}
      {\dbl{D}^I} && {\dbl{D}^J} \\
      & {\dbl{D}}
      \arrow[""{name=0, anchor=center, inner sep=0}, "\Sigma"', from=1-1, to=2-2]
      \arrow[""{name=1, anchor=center, inner sep=0}, "\Sigma", from=1-3, to=2-2]
      \arrow["{\dbl{D}^{f_0}}"', from=1-3, to=1-1]
      \arrow["{\Sigma(f_0)}", shift left, shorten <=12pt, shorten >=12pt, Rightarrow, from=0, to=1]
    \end{tikzcd}
  \end{equation*}
  induced by the universal property of coproducts has both a companion
  $\Sigma(f_0)_!$ and a conjoint $\Sigma(f_0)^*$ in the double category
  $\LaxPs(\dbl{D}^J, \dbl{D})$.
\end{proposition}
\begin{proof}
  By \cref{cor:companion-transformation-pseudo}, it is equivalent to show that,
  for each $J$-indexed family of objects $\vec{x}$ in $\dbl{D}$, the component
  arrow $\Sigma(f_0)_{\vec{x}}: \Sigma f_0^* \vec{x} \to \Sigma \vec{x}$ has both a companion and
  a conjoint in $\dbl{D}$, and for each $J$-indexed parallel family of proarrows
  $\vec{m}: \vec{x} \proto \vec{y}$ in $\dbl{D}$, the component cell
  \begin{equation*}
    \begin{tikzcd}
      {\Sigma f_0^* \vec{x}} & {\Sigma f_0^* \vec{y}} \\
      {\Sigma\vec{x}} & {\Sigma\vec{y}}
      \arrow["{\Sigma(f_0)_{\vec{x}}}"', from=1-1, to=2-1]
      \arrow["{\Sigma(f_0)_{\vec{y}}}", from=1-2, to=2-2]
      \arrow[""{name=0, anchor=center, inner sep=0}, "{\Sigma\vec{m}}"', "\shortmid"{marking}, from=2-1, to=2-2]
      \arrow[""{name=1, anchor=center, inner sep=0}, "{\Sigma f_0^* \vec{m}}", "\shortmid"{marking}, from=1-1, to=1-2]
      \arrow["{\Sigma(f_0)_{\vec{m}}}"{description}, draw=none, from=1, to=0]
    \end{tikzcd}
  \end{equation*}
  is both a commuter and cocommuter. The existence of the companions and
  conjoints in $\dbl{D}$ follows directly from the dual of \cite[\mbox{Corollary
    8.4}]{patterson2024}.

  We prove that the component cells are commuters; that they are cocommuters is
  proved dually. Given a $J$-indexed family $\vec{m}: \vec{x} \proto \vec{y}$,
  we must show that the reshaped cell
  \begin{equation*}
    \begin{tikzcd}
      {\Sigma f_0^* \vec{x}} & {\Sigma f_0^* \vec{y}} & {\Sigma\vec{y}} \\
      {\Sigma f_0^* \vec{x}} & {\Sigma\vec{x}} & {\Sigma\vec{y}}
      \arrow["{\Sigma\vec{m}}"', "\shortmid"{marking}, from=2-2, to=2-3]
      \arrow["{\Sigma f_0^* \vec{m}}", "\shortmid"{marking}, from=1-1, to=1-2]
      \arrow["{(\Sigma(f_0)_{\vec{x}})_!}"', "\shortmid"{marking}, from=2-1, to=2-2]
      \arrow["{(\Sigma(f_0)_{\vec{y}})_!}", "\shortmid"{marking}, from=1-2, to=1-3]
      \arrow["{(\Sigma(f_0)_{\vec{m}})_!}"{description}, draw=none, from=1-2, to=2-2]
      \arrow[Rightarrow, no head, from=1-3, to=2-3]
      \arrow[Rightarrow, no head, from=1-1, to=2-1]
    \end{tikzcd}
    \quad=\quad
    \begin{tikzcd}
      {\Sigma f_0^* \vec{x}} & {\Sigma f_0^* \vec{x}} & {\Sigma f_0^* \vec{y}} & {\Sigma\vec{y}} \\
      {\Sigma f_0^* \vec{x}} & {\Sigma\vec{x}} & {\Sigma\vec{y}} & {\Sigma\vec{y}}
      \arrow["{\Sigma(f_0)_{\vec{x}}}"{description}, from=1-2, to=2-2]
      \arrow["{\Sigma(f_0)_{\vec{y}}}"{description}, from=1-3, to=2-3]
      \arrow[""{name=0, anchor=center, inner sep=0}, "{\Sigma\vec{m}}"', "\shortmid"{marking}, from=2-2, to=2-3]
      \arrow[""{name=1, anchor=center, inner sep=0}, "{\Sigma f_0^* \vec{m}}", "\shortmid"{marking}, from=1-2, to=1-3]
      \arrow[""{name=2, anchor=center, inner sep=0}, "{(\Sigma(f_0)_{\vec{x}})_!}"', "\shortmid"{marking}, from=2-1, to=2-2]
      \arrow[""{name=3, anchor=center, inner sep=0}, "\shortmid"{marking}, Rightarrow, no head, from=1-1, to=1-2]
      \arrow[Rightarrow, no head, from=1-1, to=2-1]
      \arrow[Rightarrow, no head, from=1-4, to=2-4]
      \arrow[""{name=4, anchor=center, inner sep=0}, "\shortmid"{marking}, Rightarrow, no head, from=2-3, to=2-4]
      \arrow[""{name=5, anchor=center, inner sep=0}, "{(\Sigma(f_0)_{\vec{y}})_!}", "\shortmid"{marking}, from=1-3, to=1-4]
      \arrow["{\Sigma(f_0)_{\vec{m}}}"{description}, draw=none, from=1, to=0]
      \arrow["\eta"{description}, draw=none, from=3, to=2]
      \arrow["\varepsilon"{description}, draw=none, from=5, to=4]
    \end{tikzcd}
  \end{equation*}
  is invertible. Using the notation of \cite[\S{3}]{patterson2024}, we have the
  equation in $\DblFam(\dbl{D})$
  \begin{equation*}
    \begin{tikzcd}
      {(I, f_0^*\vec{x})} & {(I, f_0^*\vec{x})} & {(I, f_0^*\vec{y})} & {(J,\vec{y})} \\
      {(I, f_0^*\vec{x})} & {(J,\vec{x})} & {(J,\vec{y})} & {(J,\vec{y})}
      \arrow[""{name=0, anchor=center, inner sep=0}, "{((f_0)_!, \id)}"', "\shortmid"{marking}, from=2-1, to=2-2]
      \arrow[Rightarrow, no head, from=1-1, to=2-1]
      \arrow[""{name=1, anchor=center, inner sep=0}, "\shortmid"{marking}, Rightarrow, no head, from=1-1, to=1-2]
      \arrow["{(f_0,1)}"{description}, from=1-2, to=2-2]
      \arrow[""{name=2, anchor=center, inner sep=0}, "{(\id_I, f_0^* \vec{m})}", "\shortmid"{marking}, from=1-2, to=1-3]
      \arrow["{(f_0,1)}"{description}, from=1-3, to=2-3]
      \arrow[""{name=3, anchor=center, inner sep=0}, "{(\id_J,\vec{m})}"', "\shortmid"{marking}, from=2-2, to=2-3]
      \arrow[""{name=4, anchor=center, inner sep=0}, "{((f_0)_!, \id)}", "\shortmid"{marking}, from=1-3, to=1-4]
      \arrow[Rightarrow, no head, from=1-4, to=2-4]
      \arrow[""{name=5, anchor=center, inner sep=0}, "\shortmid"{marking}, Rightarrow, no head, from=2-3, to=2-4]
      \arrow["\eta"{description}, draw=none, from=1, to=0]
      \arrow["{(f_0,1)}"{description}, draw=none, from=2, to=3]
      \arrow["\varepsilon"{description}, draw=none, from=4, to=5]
    \end{tikzcd}
    \quad=\quad
    \begin{tikzcd}
      {(I, f_0^*\vec{x})} & {(J,\vec{y})} \\
      {(I, f_0^*\vec{x})} & {(J,\vec{y})}
      \arrow[Rightarrow, no head, from=1-1, to=2-1]
      \arrow[Rightarrow, no head, from=1-2, to=2-2]
      \arrow[""{name=0, anchor=center, inner sep=0}, "{((f_0)_!, f_0^*\vec{m})}", "\shortmid"{marking}, from=1-1, to=1-2]
      \arrow[""{name=1, anchor=center, inner sep=0}, "{((f_0)_!, f_0^*\vec{m})}"', "\shortmid"{marking}, from=2-1, to=2-2]
      \arrow["1"{description}, draw=none, from=0, to=1]
    \end{tikzcd},
  \end{equation*}
  where the cells $\eta = (1_I, 1)$ and $\varepsilon = (f_0, 1)$ are indeed binding cells for
  companions in $\DblFam(\dbl{D})$ by \cite[\mbox{Proposition
    3.8}]{patterson2024}. Therefore, the composite of the images of the cells on
  the left-hand side under the coproduct double functor
  $\Sigma: \DblFam(\dbl{D}) \to \dbl{D}$ is the cell of interest,
  $(\Sigma(f_0)_{\vec{m}})_!$. But the image of the right-hand side is invertible, in
  fact is the identity. So, by the naturality and invertibility of the
  composition comparisons of the double functor $\Sigma: \DblFam(\dbl{D}) \to \dbl{D}$,
  it follows that the cell $(\Sigma(f_0)_{\vec{m}})_!$ is also invertible.
\end{proof}

Recall that a bicategory has \define{direct sums} when it has finite
bicategorical products and coproducts, the initial object is also terminal
(i.e., is a \define{zero object}), and the canonical comparison morphisms
$\sum_{i \in I} x_i \to \prod_{i \in I} x_i$ constructed using the zero object
are equivalences \cite[\S{6}]{lack2010bicategories}. In the proof of the
following theorem, we are able to choose direct sums such that the finite
products and coproducts actually have identical underlying objects.

\begin{theorem} \label{thm:direct-sums}
  The underlying bicategory of a double category with \emph{strong} finite
  coproducts has direct sums, induced from coproducts in the underlying
  2-category by taking companions and conjoints.
\end{theorem}
\begin{proof}
  Suppose $\dbl{D}$ is a double category with strong finite coproducts. Then,
  for any finite set $I$, there is a double adjunction
  \begin{equation*}
    (\iota, \nabla): \Sigma \dashv \Delta_{\dbl{D}}: \dbl{D}^I \tofrom \dbl{D}.
  \end{equation*}
  By \cref{prop:structure-proarrows-coproducts} applied to the functions
  $f_0: 1 \xhookrightarrow{\iota_i} I$, for $i \in I$, and $f_0: I \xto{!} 1$, both of
  the transformations $\iota$ and $\nabla$ have companion and conjoint
  protransformations. In particular, we have an adjunction $\nabla_! \dashv \nabla^*$ between
  the codiagonals and diagonals in the bicategory $\LaxPsTwo(\dbl{D}, \dbl{D})$.

  Furthermore, by \cref{cor:transpose-adjunction} and its dual for conjoints,
  there are biadjunctions
  \begin{equation*}
    (\iota_!, \nabla_!): \oplus \dashv \Delta_{\bicat{D}}: \bicat{D}^I \tofrom \bicat{D}
    \qquad\text{and}\qquad
    (\nabla^*, \iota^*): \Delta_{\bicat{D}} \dashv \oplus: \bicat{D} \tofrom \bicat{D}^I,
  \end{equation*}
  where $\oplus: \bicat{D}^I \to \bicat{D}$ is the pseudofunctor between bicategories
  underlying the coproduct double functor $\Sigma: \dbl{D}^I \to \dbl{D}$. Thus, the
  pseudofunctor $\otimes: \bicat{D}^I \to \bicat{D}$ is a choice of direct sums
  in $\bicat{D}$.
\end{proof}

We caution that the theorem says nothing about cocartesian equipments, which in
general have only iso-strong finite coproducts. Here are two positive examples,
beginning with the ur-example of spans, which have direct sums whenever the base
category has finite coproducts that interact well with pullbacks.

\begin{example}[Spans]
  For any extensive category $\cat{S}$ with pullbacks, the double category of
  spans in $\cat{S}$ has strong finite coproducts \cite[\mbox{Theorem
    4.3}]{patterson2024}. So, by the theorem above, its underlying bicategory of
  spans has direct sums. In this way, we recover a result about bicategories of
  spans proved directly by Lack, Walters, and Wood \cite[\mbox{Theorem
    6.2}]{lack2010bicategories}.
\end{example}

\begin{example}[Matrices]
  For any (infinitary) distributive monoidal category $\catV$, the double
  category of $\catV$-matrices has strong coproducts \cite[\mbox{Proposition
    4.5}]{patterson2024}. Thus, its underlying bicategory of $\catV$-matrices
  has direct sums, giving a categorified ``matrix calculus.''
\end{example}

\begin{remark}[Other notions of coproduct]
  The theorem makes connections with several related ideas in the literature.
  Through his work on proarrow equipments \cite{wood1982,wood1985}, Wood has
  explored structure and axioms on a bicategory that would make it a suitable
  environment for formal category theory. His first three axioms define a
  \emph{proarrow equipment} \cite[\mbox{Axioms 1-3}]{wood1985}, which is
  interchangeable with an equipment in the sense of double categories
  \cite[\mbox{Appendix C}]{shulman2008}. His fourth axiom states that the
  2-category of arrows has finite coproducts, which become bicategorical
  coproducts upon taking companions and bicategorical products upon taking
  conjoints \cite[\mbox{Axiom 4}]{wood1985}. So, it follows from
  \cref{thm:direct-sums} that an equipment with strong finite coproducts
  satisfies Wood's first four axioms. In addition, we obtain an instance of
  Garner and Shulman's ``tight finite coproducts,'' a kind of ``tight collage''
  in an equipment \cite[\mbox{Example 16.14}]{garner2016}. When the double
  category is strict, we likewise have an instance of Lack and Shulman's ``tight
  finite coproducts'' in an $\catF$-category \cite[\S{3.5.1}]{lack2012}.

  Wood's fifth and final axiom \cite[\mbox{Axiom 5}]{wood1985}, as well as
  Garner and Shulman's other kind of tight collage \cite[\mbox{Example
    16.13}]{garner2016}, concern the existence of Kleisli objects and are beyond
  the scope of this paper. Nevertheless, the results presented here lend further
  support to the thesis, under active development by numerous category
  theorists, that double categories are an elegant foundation for formal
  category theory.
\end{remark}

\appendix
\crefalias{section}{appendix}

\section{Companions and conjoints}
\label{sec:companions}

Companions and conjoints provide the basic means to transpose structure in
double categories. As such, they have played a central role in previous works on
extracting structured bicategories from structured double categories
\cite{shulman2010,hansen2019}. In this appendix, we review the definitions of
companions and conjoints and the facts about them that we use throughout the
paper. Most of the results are known, if not always stated explicitly in the
literature. An exception is the final result recognizing companions as a
biadjoint, which refines an adjunction observed earlier by Grandis and Paré
\cite{grandis2004}.

Companions and conjoints were introduced by Grandis and Paré under the names
``orthogonal companions'' and ``orthogonal adjoints'' \cite{grandis2004}, with
antecedents in Brown and Spencer's early work on ``connections'' in a double
category \cite{brown1976}. The close relation between companions, conjoints, and
proarrow equipments was discovered by Shulman \cite{shulman2008}.

\begin{definition}[Companions and conjoints]
  A \define{companion} of an arrow $f: x \to y$ in a double category consists of a
  proarrow $f_!: x \proto y$ and a pair of cells
  \begin{equation*}
    \begin{tikzcd}
      x & x \\
      x & y
      \arrow[""{name=0, anchor=center, inner sep=0}, "{f_!}"', "\shortmid"{marking}, from=2-1, to=2-2]
      \arrow["f", from=1-2, to=2-2]
      \arrow[""{name=1, anchor=center, inner sep=0}, "{\id_x}", "\shortmid"{marking}, from=1-1, to=1-2]
      \arrow[Rightarrow, no head, from=1-1, to=2-1]
      \arrow["\eta"{description}, draw=none, from=1, to=0]
    \end{tikzcd}
    \qquad\text{and}\qquad
    \begin{tikzcd}
      x & y \\
      y & y
      \arrow["f"', from=1-1, to=2-1]
      \arrow[""{name=0, anchor=center, inner sep=0}, "{f_!}", "\shortmid"{marking}, from=1-1, to=1-2]
      \arrow[""{name=1, anchor=center, inner sep=0}, "{\id_y}"', "\shortmid"{marking}, from=2-1, to=2-2]
      \arrow[Rightarrow, no head, from=1-2, to=2-2]
      \arrow["\varepsilon"{description}, draw=none, from=0, to=1]
    \end{tikzcd},
  \end{equation*}
  the \define{unit} and the \define{counit}, satisfying the equations
  \begin{equation*}
    \begin{tikzcd}
      x & x & y \\
      x & y & y
      \arrow["f"', from=1-2, to=2-2]
      \arrow[""{name=0, anchor=center, inner sep=0}, "{f_!}", "\shortmid"{marking}, from=1-2, to=1-3]
      \arrow[""{name=1, anchor=center, inner sep=0}, "{\id_y}"', "\shortmid"{marking}, from=2-2, to=2-3]
      \arrow[Rightarrow, no head, from=1-3, to=2-3]
      \arrow[""{name=2, anchor=center, inner sep=0}, "{f_!}"', "\shortmid"{marking}, from=2-1, to=2-2]
      \arrow[Rightarrow, no head, from=1-1, to=2-1]
      \arrow[""{name=3, anchor=center, inner sep=0}, "{\id_x}", "\shortmid"{marking}, from=1-1, to=1-2]
      \arrow["\varepsilon"{description}, draw=none, from=0, to=1]
      \arrow["\eta"{description}, draw=none, from=3, to=2]
    \end{tikzcd}
    \quad=\quad
    \begin{tikzcd}
      x & y \\
      x & y
      \arrow[""{name=0, anchor=center, inner sep=0}, "{f_!}", "\shortmid"{marking}, from=1-1, to=1-2]
      \arrow[""{name=1, anchor=center, inner sep=0}, "{f_!}"', "\shortmid"{marking}, from=2-1, to=2-2]
      \arrow[Rightarrow, no head, from=1-1, to=2-1]
      \arrow[Rightarrow, no head, from=1-2, to=2-2]
      \arrow["{1_{f_!}}"{description}, draw=none, from=0, to=1]
    \end{tikzcd}
    \qquad\text{and}\qquad
    \begin{tikzcd}
      x & x \\
      x & y \\
      y & y
      \arrow["f", from=1-2, to=2-2]
      \arrow[""{name=0, anchor=center, inner sep=0}, "{f_!}", "\shortmid"{marking}, from=2-1, to=2-2]
      \arrow[Rightarrow, no head, from=1-1, to=2-1]
      \arrow[""{name=1, anchor=center, inner sep=0}, "{\id_x}", "\shortmid"{marking}, from=1-1, to=1-2]
      \arrow["f"', from=2-1, to=3-1]
      \arrow[Rightarrow, no head, from=2-2, to=3-2]
      \arrow[""{name=2, anchor=center, inner sep=0}, "{\id_y}"', "\shortmid"{marking}, from=3-1, to=3-2]
      \arrow["\eta"{description, pos=0.4}, draw=none, from=1, to=0]
      \arrow["\varepsilon"{description}, draw=none, from=0, to=2]
    \end{tikzcd}
    \quad=\quad
    \begin{tikzcd}
      x & x \\
      y & y
      \arrow[""{name=0, anchor=center, inner sep=0}, "{\id_x}", "\shortmid"{marking}, from=1-1, to=1-2]
      \arrow[""{name=1, anchor=center, inner sep=0}, "{\id_y}"', "\shortmid"{marking}, from=2-1, to=2-2]
      \arrow["f"', from=1-1, to=2-1]
      \arrow["f", from=1-2, to=2-2]
      \arrow["{\id_f}"{description}, draw=none, from=0, to=1]
    \end{tikzcd},
  \end{equation*}
  where in the first equation we suppress the unit isomorphisms
  $\id_x \odot f_! \cong f_!$ and $f_! \odot \id_y \cong f_!$ following our
  usual convention.

  Dually, a \define{conjoint} of an arrow $f: x \to y$ consists of a proarrow
  $f^*: y \proto x$ and a pair of cells
  \begin{equation*}
    \begin{tikzcd}
      x & x \\
      y & x
      \arrow["f"', from=1-1, to=2-1]
      \arrow[""{name=0, anchor=center, inner sep=0}, "{f^*}"', "\shortmid"{marking}, from=2-1, to=2-2]
      \arrow[Rightarrow, no head, from=1-2, to=2-2]
      \arrow[""{name=1, anchor=center, inner sep=0}, "{\id_x}", "\shortmid"{marking}, from=1-1, to=1-2]
      \arrow["\eta"{description}, draw=none, from=1, to=0]
    \end{tikzcd}
    \qquad\text{and}\qquad
    \begin{tikzcd}
      y & x \\
      y & y
      \arrow[""{name=0, anchor=center, inner sep=0}, "{f^*}", "\shortmid"{marking}, from=1-1, to=1-2]
      \arrow["f", from=1-2, to=2-2]
      \arrow[Rightarrow, no head, from=1-1, to=2-1]
      \arrow[""{name=1, anchor=center, inner sep=0}, "{\id_y}"', "\shortmid"{marking}, from=2-1, to=2-2]
      \arrow["\varepsilon"{description}, draw=none, from=0, to=1]
    \end{tikzcd}
  \end{equation*}
  satisfying the equations $\varepsilon \odot \eta = 1_{f^*}$ and
  $\eta \cdot \varepsilon = \id_f$.
\end{definition}

\begin{remark}[Duality] \label{rem:duality}
  The two definitions are indeed dual; they are even dual according to both
  forms of double-categorical duality. A conjoint of an arrow $f: x \to y$ in a
  double category $\dbl{D}$ is just a companion of $f: x \to y$ in
  $\dbl{D}^\rev$, where the \define{reverse} double category $\dbl{D}^\rev$ is
  obtained from $\dbl{D}$ by swapping the source and target functors
  $s,t: \dbl{D}_1 \rightrightarrows \dbl{D}_0$. Alternatively, a conjoint of an
  arrow $f: x \to y$ in $\dbl{D}$ is exactly a companion of $f: y \to x$ in
  $\dbl{D}^\op$, where the \define{opposite} double category $\dbl{D}^\op$ has
  opposite underlying categories $(\dbl{D}^\op)_i = (\dbl{D}_i)^\op$ for
  $i=0,1$. Either way, any statement about companions has a dual statement about
  conjoints and vice versa.
\end{remark}

Defining companions and conjoints equationally, analogous to defining an
adjunction by unit and counit cells obeying the triangle identities, is useful
for performing calculations. Alternatively, companions and conjoints can each be
defined by either of two universal properties, recognizing the unit or counit as
an opcartesian or cartesian cell, respectively. This fundamental three-way
equivalence is stated in \cite[\S\S{1.2--1.3}]{grandis2004} and extended to
equipments in \cite[\mbox{Theorem 4.1}]{shulman2008}. As usual, the universal
properties imply that companions and conjoints are unique up to canonical
isomorphism whenever they exist, which can also be proved equationally
\cite[\mbox{Lemma 3.8}]{shulman2010}.

\subsection{Reshaping cells by sliding}
\label{sec:sliding}

Companions and conjoints allow cells in a double category to be reshaped by
``sliding'' arrows around corners, turning them into proarrows. Since this is
the basic technique on which the paper rests, we will present it in detail. The
following lemma restates the rules for sliding companions and conjoints, as
found in sources such as \cite[\S{1.6}]{grandis2004}, \cite[\mbox{Proposition
  20}]{garner2009}, \cite[\mbox{Proposition 5.13}]{shulman2011}, and
\cite[\S{1}]{pare2023}.

\begin{lemma}[Sliding] \label{lem:sliding-general}
  Suppose $x \xto{f} x' \xto{f'} x''$ and $y \xto{g} y' \xto{g'} y''$ are arrows
  in a double category. If the arrows $f'$ and $g$ have companions and a choice
  of them is made, then there is a bijective correspondence between cells of the
  left and middle shapes:
  \begin{equation*}
    \begin{tikzcd}
      x & y & {y'} \\
      {x'} & {x''} & {y''}
      \arrow["m", "\shortmid"{marking}, from=1-1, to=1-2]
      \arrow["{g'}", from=1-3, to=2-3]
      \arrow["{m''}"', "\shortmid"{marking}, from=2-2, to=2-3]
      \arrow["{f'_!}"', "\shortmid"{marking}, from=2-1, to=2-2]
      \arrow["f"', from=1-1, to=2-1]
      \arrow["{g_!}", "\shortmid"{marking}, from=1-2, to=1-3]
      \arrow["{\alpha_!}"{description}, draw=none, from=1-2, to=2-2]
    \end{tikzcd}
    \qquad\leftrightsquigarrow\qquad
    \begin{tikzcd}
      x & y \\
      {x'} & {y'} \\
      {x''} & {y''}
      \arrow[""{name=0, anchor=center, inner sep=0}, "m", "\shortmid"{marking}, from=1-1, to=1-2]
      \arrow["f"', from=1-1, to=2-1]
      \arrow["{f'}"', from=2-1, to=3-1]
      \arrow["g", from=1-2, to=2-2]
      \arrow["{g'}", from=2-2, to=3-2]
      \arrow[""{name=1, anchor=center, inner sep=0}, "{m''}"', "\shortmid"{marking}, from=3-1, to=3-2]
      \arrow["\alpha"{description}, draw=none, from=0, to=1]
    \end{tikzcd}
    \qquad\leftrightsquigarrow\qquad
    \begin{tikzcd}
      {x'} & x & y \\
      {x''} & {y''} & {y'}
      \arrow["m", "\shortmid"{marking}, from=1-2, to=1-3]
      \arrow["g", from=1-3, to=2-3]
      \arrow["{(g')^*}"', "\shortmid"{marking}, from=2-2, to=2-3]
      \arrow["{m''}"', "\shortmid"{marking}, from=2-1, to=2-2]
      \arrow["{f^*}", "\shortmid"{marking}, from=1-1, to=1-2]
      \arrow["{f'}"', from=1-1, to=2-1]
      \arrow["{\alpha^*}"{description}, draw=none, from=1-2, to=2-2]
    \end{tikzcd}.
  \end{equation*}
  Dually, if the arrows $f$ and $g'$ have chosen conjoints, then there is a
  bijective correspondence between cells of the middle and right shapes.
\end{lemma}
\begin{proof}
  The first bijection, involving companions, is given by the assignment
  \begin{equation} \label{eq:sliding-companion}
    \begin{tikzcd}
      x & y & {y'} \\
      {x'} & {x''} & {y''}
      \arrow["m", "\shortmid"{marking}, from=1-1, to=1-2]
      \arrow["{g'}", from=1-3, to=2-3]
      \arrow["{m''}"', "\shortmid"{marking}, from=2-2, to=2-3]
      \arrow["{f'_!}"', "\shortmid"{marking}, from=2-1, to=2-2]
      \arrow["f"', from=1-1, to=2-1]
      \arrow["{g_!}", "\shortmid"{marking}, from=1-2, to=1-3]
      \arrow["{\alpha_!}"{description}, draw=none, from=1-2, to=2-2]
    \end{tikzcd}
    \quad=\quad
    \begin{tikzcd}[row sep=scriptsize]
      x & x & y & {y'} \\
      {x'} & {x'} & {y'} & {y'} \\
      {x'} & {x''} & {y''} & {y''}
      \arrow[""{name=0, anchor=center, inner sep=0}, "m", "\shortmid"{marking}, from=1-2, to=1-3]
      \arrow["f", from=1-2, to=2-2]
      \arrow["{f'}", from=2-2, to=3-2]
      \arrow["g"', from=1-3, to=2-3]
      \arrow["{g'}"', from=2-3, to=3-3]
      \arrow[""{name=1, anchor=center, inner sep=0}, "{m''}"', "\shortmid"{marking}, from=3-2, to=3-3]
      \arrow[""{name=2, anchor=center, inner sep=0}, "{\id_x}", "\shortmid"{marking}, from=1-1, to=1-2]
      \arrow[""{name=3, anchor=center, inner sep=0}, "{\id_{x'}}", "\shortmid"{marking}, from=2-1, to=2-2]
      \arrow["f"', from=1-1, to=2-1]
      \arrow[""{name=4, anchor=center, inner sep=0}, "{f'_!}"', "\shortmid"{marking}, from=3-1, to=3-2]
      \arrow[Rightarrow, no head, from=2-1, to=3-1]
      \arrow["{g'}", from=2-4, to=3-4]
      \arrow[""{name=5, anchor=center, inner sep=0}, "{\id_{y''}}"', "\shortmid"{marking}, from=3-3, to=3-4]
      \arrow[""{name=6, anchor=center, inner sep=0}, "{\id_{y'}}", "\shortmid"{marking}, from=2-3, to=2-4]
      \arrow[Rightarrow, no head, from=1-4, to=2-4]
      \arrow[""{name=7, anchor=center, inner sep=0}, "{g_!}", "\shortmid"{marking}, from=1-3, to=1-4]
      \arrow["\alpha"{description}, draw=none, from=0, to=1]
      \arrow["{\id_f}"{description, pos=0.4}, draw=none, from=2, to=3]
      \arrow["\eta"{description}, draw=none, from=3, to=4]
      \arrow["{\id_{g'}}"{description}, draw=none, from=6, to=5]
      \arrow["\varepsilon"{description, pos=0.4}, draw=none, from=7, to=6]
    \end{tikzcd}
  \end{equation}
  and conversely
  \begin{equation} \label{eq:sliding-companion-inv}
    \begin{tikzcd}
      x & y \\
      {x'} & {y'} \\
      {x''} & {y''}
      \arrow[""{name=0, anchor=center, inner sep=0}, "m", "\shortmid"{marking}, from=1-1, to=1-2]
      \arrow["f"', from=1-1, to=2-1]
      \arrow["{f'}"', from=2-1, to=3-1]
      \arrow["g", from=1-2, to=2-2]
      \arrow["{g'}", from=2-2, to=3-2]
      \arrow[""{name=1, anchor=center, inner sep=0}, "{m''}"', "\shortmid"{marking}, from=3-1, to=3-2]
      \arrow["\alpha"{description}, draw=none, from=0, to=1]
    \end{tikzcd}
    \quad=\quad
    \begin{tikzcd}[row sep=scriptsize]
      x & y & y \\
      x & y & {y'} \\
      {x'} & {x''} & {y''} \\
      {x''} & {x''} & {y''}
      \arrow[""{name=0, anchor=center, inner sep=0}, "m"', "\shortmid"{marking}, from=2-1, to=2-2]
      \arrow["{g'}", from=2-3, to=3-3]
      \arrow[""{name=1, anchor=center, inner sep=0}, "{m''}", "\shortmid"{marking}, from=3-2, to=3-3]
      \arrow[""{name=2, anchor=center, inner sep=0}, "{f'_!}", "\shortmid"{marking}, from=3-1, to=3-2]
      \arrow["f"', from=2-1, to=3-1]
      \arrow[""{name=3, anchor=center, inner sep=0}, "{g_!}"', "\shortmid"{marking}, from=2-2, to=2-3]
      \arrow["{\alpha_!}"{description}, draw=none, from=2-2, to=3-2]
      \arrow["g", from=1-3, to=2-3]
      \arrow[Rightarrow, no head, from=1-2, to=2-2]
      \arrow[""{name=4, anchor=center, inner sep=0}, "{\id_y}", "\shortmid"{marking}, from=1-2, to=1-3]
      \arrow[""{name=5, anchor=center, inner sep=0}, "m", "\shortmid"{marking}, from=1-1, to=1-2]
      \arrow[Rightarrow, no head, from=1-1, to=2-1]
      \arrow["{f'}"', from=3-1, to=4-1]
      \arrow[""{name=6, anchor=center, inner sep=0}, "{m''}"', "\shortmid"{marking}, from=4-2, to=4-3]
      \arrow[Rightarrow, no head, from=3-3, to=4-3]
      \arrow[Rightarrow, no head, from=3-2, to=4-2]
      \arrow[""{name=7, anchor=center, inner sep=0}, "{\id_{x''}}"', "\shortmid"{marking}, from=4-1, to=4-2]
      \arrow["\eta"{description}, draw=none, from=4, to=3]
      \arrow["{1_m}"{description}, draw=none, from=5, to=0]
      \arrow["{1_{m''}}"{description}, draw=none, from=1, to=6]
      \arrow["\varepsilon"{description}, draw=none, from=2, to=7]
    \end{tikzcd}.
  \end{equation}
  That these assignments are mutually inverse follows directly from the defining
  equations for a companion. The second bijection, involving conjoints, is dual.
\end{proof}

A particular case of the lemma is useful enough to state independently. It
furnishes the globular cells available in a bicategory.

\begin{lemma}[Sliding, globular] \label{lem:sliding-globular}
  Suppose $f: x \to w$ and $g: y \to z$ are arrows with companions or conjoints
  in a double category. Then, for any choice of companions or conjoints of $f$
  and $g$, there is one or the other bijective correspondence between cells:
  \begin{equation*}
    \begin{tikzcd}
      x & y & z \\
      x & w & z
      \arrow["m", "\shortmid"{marking}, from=1-1, to=1-2]
      \arrow["{g_!}", "\shortmid"{marking}, from=1-2, to=1-3]
      \arrow[Rightarrow, no head, from=1-1, to=2-1]
      \arrow["{f_!}"', "\shortmid"{marking}, from=2-1, to=2-2]
      \arrow["n"', "\shortmid"{marking}, from=2-2, to=2-3]
      \arrow[Rightarrow, no head, from=1-3, to=2-3]
      \arrow["{\alpha_!}"{description}, draw=none, from=1-2, to=2-2]
    \end{tikzcd}
    \qquad\leftrightsquigarrow\qquad
    \begin{tikzcd}
      x & y \\
      w & z
      \arrow[""{name=0, anchor=center, inner sep=0}, "m", "\shortmid"{marking}, from=1-1, to=1-2]
      \arrow["f"', from=1-1, to=2-1]
      \arrow["g", from=1-2, to=2-2]
      \arrow[""{name=1, anchor=center, inner sep=0}, "n"', "\shortmid"{marking}, from=2-1, to=2-2]
      \arrow["\alpha"{description}, draw=none, from=0, to=1]
    \end{tikzcd}
    \qquad\leftrightsquigarrow\qquad
    \begin{tikzcd}
      w & x & y \\
      w & z & y
      \arrow["{f^*}", "\shortmid"{marking}, from=1-1, to=1-2]
      \arrow["m", "\shortmid"{marking}, from=1-2, to=1-3]
      \arrow["n"', "\shortmid"{marking}, from=2-1, to=2-2]
      \arrow["{g^*}"', "\shortmid"{marking}, from=2-2, to=2-3]
      \arrow[Rightarrow, no head, from=1-3, to=2-3]
      \arrow[Rightarrow, no head, from=1-1, to=2-1]
      \arrow["{\alpha^*}"{description}, draw=none, from=1-2, to=2-2]
    \end{tikzcd}.
  \end{equation*}
\end{lemma}
\begin{proof}
  The first bijection follows by taking $f$ and $g'$ to be identities in
  \cref{lem:sliding-general} and the second follows by taking $f'$ and $g$ to be
  identities.
\end{proof}

\begin{remark}[Mates] \label{rem:mates}
  A companion-conjoint pair of proarrows in a double category is, in particular,
  an adjunction in the underlying bicategory by \cite[\mbox{Proposition
    1.4}]{grandis2004} or \cite[\mbox{Lemma 3.21}]{shulman2010}. Under the
  bijection in \cref{lem:sliding-globular}, the cells $\alpha_!$ and $\alpha^*$
  are then \emph{mates} with respect to the adjunctions $f_! \dashv f^*$ and
  $g_! \dashv g^*$. The sliding rules for companions and conjoints thus extend
  the calculus of mates for adjunctions in a bicategory, a perspective taken in
  \cite{shulman2011}.
\end{remark}

The following terminology is adopted from \cite[\mbox{Definition
  8.1}]{pare2023}.

\begin{definition}[Commuters] \label{def:commuters}
  A cell $\alpha$ as in \cref{lem:sliding-globular} is a \define{commuter} if $\alpha_!$
  is an isomorphism. Dually, a cell $\alpha$ is a \define{cocommuter} if $\alpha^*$ is an
  isomorphism.
\end{definition}

A commuter cell is interpreted as a square of arrows and proarrows that
``commutes'' up to isomorphism. Note that being a commuter or cocommuter are
independent properties, so that a cell can represent an up-to-iso ``commutative
square'' in two different senses. For a cell to be a commuter or cocommuter, it
need not itself be an isomorphism, although, as the next lemma shows, that is a
sufficient condition.

\begin{lemma}[Isomorphisms are commuters] \label{lem:isos-are-commuters}
  Let $\stdInlineCell{\alpha}$ be a cell in a double category. If the arrows $f$ and
  $g$ have companions and $\alpha$ is an isomorphism (hence $f$ and $g$ are also
  isomorphisms), then $\alpha$ is a commuter.

  Dually, if $f$ and $g$ have conjoints and $\alpha$ is an isomorphism, then
  $\alpha$ is a cocommuter.
\end{lemma}
\begin{proof}
  As observed in \cite[Lemma 3.20]{shulman2010}, if $f$ is an isomorphism with a
  companion $f_!$, then $f_!$ is also a conjoint of the inverse $f^{-1}$, via
  the unit and counit cells
  \begin{equation*}
    \begin{tikzcd}
      w & w \\
      x & x \\
      x & w
      \arrow["{f^{-1}}"', from=1-1, to=2-1]
      \arrow["{f^{-1}}", from=1-2, to=2-2]
      \arrow[""{name=0, anchor=center, inner sep=0}, "{\id_x}"', "\shortmid"{marking}, from=2-1, to=2-2]
      \arrow[""{name=1, anchor=center, inner sep=0}, "{\id_w}", "\shortmid"{marking}, from=1-1, to=1-2]
      \arrow["f", from=2-2, to=3-2]
      \arrow[""{name=2, anchor=center, inner sep=0}, "{f_!}"', from=3-1, to=3-2]
      \arrow[Rightarrow, no head, from=2-1, to=3-1]
      \arrow["{\id_f^{-1}}"{description}, draw=none, from=1, to=0]
      \arrow["\eta"{description, pos=0.6}, draw=none, from=0, to=2]
    \end{tikzcd}
    \qquad\text{and}\qquad
    \begin{tikzcd}
      x & w \\
      w & w \\
      x & x
      \arrow[""{name=0, anchor=center, inner sep=0}, "{f_!}", "\shortmid"{marking}, from=1-1, to=1-2]
      \arrow["f"', from=1-1, to=2-1]
      \arrow[""{name=1, anchor=center, inner sep=0}, "{\id_w}", "\shortmid"{marking}, from=2-1, to=2-2]
      \arrow[Rightarrow, no head, from=1-2, to=2-2]
      \arrow["{f^{-1}}"', from=2-1, to=3-1]
      \arrow["{f^{-1}}", from=2-2, to=3-2]
      \arrow[""{name=2, anchor=center, inner sep=0}, "{\id_x}"', "\shortmid"{marking}, from=3-1, to=3-2]
      \arrow["\varepsilon"{description, pos=0.4}, draw=none, from=0, to=1]
      \arrow["{\id_f^{-1}}"{description}, draw=none, from=1, to=2]
    \end{tikzcd}.
  \end{equation*}
  It is then straightforward to show that an inverse to the cell
  \begin{equation*}
    \begin{tikzcd}[row sep=scriptsize]
      x & y & z \\
      x & w & z
      \arrow["m", "\shortmid"{marking}, from=1-1, to=1-2]
      \arrow["{g_!}", "\shortmid"{marking}, from=1-2, to=1-3]
      \arrow[Rightarrow, no head, from=1-1, to=2-1]
      \arrow["{f_!}"', "\shortmid"{marking}, from=2-1, to=2-2]
      \arrow["n"', "\shortmid"{marking}, from=2-2, to=2-3]
      \arrow[Rightarrow, no head, from=1-3, to=2-3]
      \arrow["{\alpha_!}"{description}, draw=none, from=1-2, to=2-2]
    \end{tikzcd}
    \quad\coloneqq\quad
    \begin{tikzcd}[row sep=scriptsize]
      x & x & y & z \\
      x & w & z & z
      \arrow["f"', from=1-2, to=2-2]
      \arrow["g"', from=1-3, to=2-3]
      \arrow[""{name=0, anchor=center, inner sep=0}, "m", "\shortmid"{marking}, from=1-2, to=1-3]
      \arrow[""{name=1, anchor=center, inner sep=0}, "n"', "\shortmid"{marking}, from=2-2, to=2-3]
      \arrow[""{name=2, anchor=center, inner sep=0}, "{f_!}"', "\shortmid"{marking}, from=2-1, to=2-2]
      \arrow[""{name=3, anchor=center, inner sep=0}, "{\id_x}", "\shortmid"{marking}, from=1-1, to=1-2]
      \arrow[Rightarrow, no head, from=1-1, to=2-1]
      \arrow[""{name=4, anchor=center, inner sep=0}, "{g_!}", "\shortmid"{marking}, from=1-3, to=1-4]
      \arrow[Rightarrow, no head, from=1-4, to=2-4]
      \arrow[""{name=5, anchor=center, inner sep=0}, "{\id_z}"', "\shortmid"{marking}, from=2-3, to=2-4]
      \arrow["\eta"{description}, draw=none, from=3, to=2]
      \arrow["\alpha"{description}, draw=none, from=0, to=1]
      \arrow["\varepsilon"{description}, draw=none, from=4, to=5]
    \end{tikzcd}
  \end{equation*}
  is provided by the cell
  \begin{equation*}
    \begin{tikzcd}[row sep=scriptsize]
      x & w & z \\
      x & y & z
      \arrow["{f_!}", "\shortmid"{marking}, from=1-1, to=1-2]
      \arrow["n", "\shortmid"{marking}, from=1-2, to=1-3]
      \arrow["m"', "\shortmid"{marking}, from=2-1, to=2-2]
      \arrow["{g_!}"', "\shortmid"{marking}, from=2-2, to=2-3]
      \arrow[Rightarrow, no head, from=1-1, to=2-1]
      \arrow[Rightarrow, no head, from=1-3, to=2-3]
      \arrow["{(\alpha^{-1})^*}"{description}, draw=none, from=1-2, to=2-2]
    \end{tikzcd}
    \quad=\quad
    \begin{tikzcd}[row sep=scriptsize]
      x & w & z & z \\
      w & w & y & y \\
      x & x & y & z
      \arrow[""{name=0, anchor=center, inner sep=0}, "{f_!}", "\shortmid"{marking}, from=1-1, to=1-2]
      \arrow["f"', from=1-1, to=2-1]
      \arrow[""{name=1, anchor=center, inner sep=0}, "{\id_w}", "\shortmid"{marking}, from=2-1, to=2-2]
      \arrow[Rightarrow, no head, from=1-2, to=2-2]
      \arrow["{f^{-1}}"', from=2-1, to=3-1]
      \arrow["{f^{-1}}", from=2-2, to=3-2]
      \arrow[""{name=2, anchor=center, inner sep=0}, "{\id_x}"', "\shortmid"{marking}, from=3-1, to=3-2]
      \arrow[""{name=3, anchor=center, inner sep=0}, "n", "\shortmid"{marking}, from=1-2, to=1-3]
      \arrow[""{name=4, anchor=center, inner sep=0}, "{\id_z}", "\shortmid"{marking}, from=1-3, to=1-4]
      \arrow["{g^{-1}}"', from=1-3, to=2-3]
      \arrow[""{name=5, anchor=center, inner sep=0}, "m"', "\shortmid"{marking}, from=3-2, to=3-3]
      \arrow[Rightarrow, no head, from=2-3, to=3-3]
      \arrow["{g^{-1}}", from=1-4, to=2-4]
      \arrow[""{name=6, anchor=center, inner sep=0}, "{g_!}"', "\shortmid"{marking}, from=3-3, to=3-4]
      \arrow["g", from=2-4, to=3-4]
      \arrow[""{name=7, anchor=center, inner sep=0}, "{\id_y}"', "\shortmid"{marking}, from=2-3, to=2-4]
      \arrow["\varepsilon"{description, pos=0.4}, draw=none, from=0, to=1]
      \arrow["{\id_f^{-1}}"{description}, draw=none, from=1, to=2]
      \arrow["{\alpha^{-1}}"{description}, draw=none, from=3, to=5]
      \arrow["{\id_g^{-1}}"{description}, draw=none, from=4, to=7]
      \arrow["\eta"{description, pos=0.6}, draw=none, from=7, to=6]
    \end{tikzcd}.
    \qedhere
  \end{equation*}
\end{proof}

The sliding operations are functorial and natural in every sense that is well
typed. Such properties have often been tacitly used in the literature. For the
sake of completeness, we record the properties of sliding that we will need in a
series of simple but useful lemmas.

\begin{lemma}[External functoriality of sliding, globular]
  \label{lem:sliding-functoriality-external-globular}
  Suppose $f$, $g$, and $h$ are arrows with chosen companions in a double
  category. Then, under the bijection in \cref{lem:sliding-globular}, we have
  \begin{equation*}
    \begin{tikzcd}[sep=scriptsize]
      x & y & z \\
      {x'} & {y'} & {z'}
      \arrow["f"', from=1-1, to=2-1]
      \arrow["g"{description}, from=1-2, to=2-2]
      \arrow["h", from=1-3, to=2-3]
      \arrow[""{name=0, anchor=center, inner sep=0}, "m", "\shortmid"{marking}, from=1-1, to=1-2]
      \arrow[""{name=1, anchor=center, inner sep=0}, "n", "\shortmid"{marking}, from=1-2, to=1-3]
      \arrow[""{name=2, anchor=center, inner sep=0}, "{m'}"', "\shortmid"{marking}, from=2-1, to=2-2]
      \arrow[""{name=3, anchor=center, inner sep=0}, "{n'}"', "\shortmid"{marking}, from=2-2, to=2-3]
      \arrow["\alpha"{description}, draw=none, from=0, to=2]
      \arrow["\beta"{description}, draw=none, from=1, to=3]
    \end{tikzcd}
    \quad\leadsto\quad
    \begin{tikzcd}[sep=scriptsize]
      x & y & z & {z'} \\
      x & y & {y'} & {z'} \\
      x & {x'} & {y'} & {z'}
      \arrow["{g_!}", "\shortmid"{marking}, from=2-2, to=2-3]
      \arrow[""{name=0, anchor=center, inner sep=0}, "m"', "\shortmid"{marking}, from=2-1, to=2-2]
      \arrow["{f_!}"', "\shortmid"{marking}, from=3-1, to=3-2]
      \arrow["{m'}"', "\shortmid"{marking}, from=3-2, to=3-3]
      \arrow[Rightarrow, no head, from=2-1, to=3-1]
      \arrow[Rightarrow, no head, from=2-3, to=3-3]
      \arrow[Rightarrow, no head, from=1-2, to=2-2]
      \arrow["n", "\shortmid"{marking}, from=1-2, to=1-3]
      \arrow[""{name=1, anchor=center, inner sep=0}, "m", "\shortmid"{marking}, from=1-1, to=1-2]
      \arrow[Rightarrow, no head, from=1-1, to=2-1]
      \arrow[""{name=2, anchor=center, inner sep=0}, "{n'}", "\shortmid"{marking}, from=2-3, to=2-4]
      \arrow[Rightarrow, no head, from=2-4, to=3-4]
      \arrow[Rightarrow, no head, from=1-4, to=2-4]
      \arrow["{h_!}", "\shortmid"{marking}, from=1-3, to=1-4]
      \arrow["{\alpha_!}"{description}, draw=none, from=2-2, to=3-2]
      \arrow[""{name=3, anchor=center, inner sep=0}, "{n'}"', "\shortmid"{marking}, from=3-3, to=3-4]
      \arrow["{\beta_!}"{description}, draw=none, from=1-3, to=2-3]
      \arrow["{1_m}"{description}, draw=none, from=1, to=0]
      \arrow["{1_{n'}}"{description}, draw=none, from=2, to=3]
    \end{tikzcd}
    \qquad\text{and}\qquad
    \begin{tikzcd}[sep=scriptsize]
      x & x \\
      {x'} & {x'}
      \arrow["f"', from=1-1, to=2-1]
      \arrow[""{name=0, anchor=center, inner sep=0}, "{\id_x}", "\shortmid"{marking}, from=1-1, to=1-2]
      \arrow["f", from=1-2, to=2-2]
      \arrow[""{name=1, anchor=center, inner sep=0}, "{\id_{x'}}"', "\shortmid"{marking}, from=2-1, to=2-2]
      \arrow["{\id_f}"{description}, draw=none, from=0, to=1]
    \end{tikzcd}
    \;\leadsto\;
    \begin{tikzcd}[sep=scriptsize]
      x & {x'} \\
      x & {x'}
      \arrow[""{name=0, anchor=center, inner sep=0}, "{f_!}", "\shortmid"{marking}, from=1-1, to=1-2]
      \arrow[Rightarrow, no head, from=1-1, to=2-1]
      \arrow[Rightarrow, no head, from=1-2, to=2-2]
      \arrow[""{name=1, anchor=center, inner sep=0}, "{f_!}"', "\shortmid"{marking}, from=2-1, to=2-2]
      \arrow["{1_{f_!}}"{description}, draw=none, from=0, to=1]
    \end{tikzcd}.
  \end{equation*}
  In particular, the external identity cell $\id_f$ is a commuter.
\end{lemma}
\begin{proof}
  Compose the cell $\alpha \odot \beta$ on the left with $\eta$ and on the right with $\varepsilon$, and
  insert the identity $\id_g = \eta \cdot \varepsilon$ in the middle, to obtain the equation
  \begin{equation*}
    \begin{tikzcd}[column sep=scriptsize]
      x & x & y & z & {z'} \\
      x & {x'} & {y'} & {z'} & {z'}
      \arrow["f"', from=1-2, to=2-2]
      \arrow["g"{description}, from=1-3, to=2-3]
      \arrow["h", from=1-4, to=2-4]
      \arrow[""{name=0, anchor=center, inner sep=0}, "m", "\shortmid"{marking}, from=1-2, to=1-3]
      \arrow[""{name=1, anchor=center, inner sep=0}, "n", "\shortmid"{marking}, from=1-3, to=1-4]
      \arrow[""{name=2, anchor=center, inner sep=0}, "{m'}"', "\shortmid"{marking}, from=2-2, to=2-3]
      \arrow[""{name=3, anchor=center, inner sep=0}, "{n'}"', "\shortmid"{marking}, from=2-3, to=2-4]
      \arrow[""{name=4, anchor=center, inner sep=0}, "{\id_x}", "\shortmid"{marking}, from=1-1, to=1-2]
      \arrow[""{name=5, anchor=center, inner sep=0}, "{f_!}"', "\shortmid"{marking}, from=2-1, to=2-2]
      \arrow[Rightarrow, no head, from=1-1, to=2-1]
      \arrow[""{name=6, anchor=center, inner sep=0}, "{h_!}", "\shortmid"{marking}, from=1-4, to=1-5]
      \arrow[Rightarrow, no head, from=1-5, to=2-5]
      \arrow[""{name=7, anchor=center, inner sep=0}, "{\id_{z'}}"', "\shortmid"{marking}, from=2-4, to=2-5]
      \arrow["\alpha"{description}, draw=none, from=0, to=2]
      \arrow["\beta"{description}, draw=none, from=1, to=3]
      \arrow["\eta"{description}, draw=none, from=4, to=5]
      \arrow["\varepsilon"{description}, draw=none, from=6, to=7]
    \end{tikzcd}
    \quad=\quad
    \begin{tikzcd}[sep=scriptsize]
      x & x & y & y & z & {z'} \\
      x & x & y & {y'} & {z'} & {z'} \\
      x & {x'} & {y'} & {y'} & {z'} & {z'}
      \arrow["f"', from=2-2, to=3-2]
      \arrow["h", from=1-5, to=2-5]
      \arrow[""{name=0, anchor=center, inner sep=0}, "n", "\shortmid"{marking}, from=1-4, to=1-5]
      \arrow[""{name=1, anchor=center, inner sep=0}, "{m'}"', "\shortmid"{marking}, from=3-2, to=3-3]
      \arrow[""{name=2, anchor=center, inner sep=0}, "{n'}", "\shortmid"{marking}, from=2-4, to=2-5]
      \arrow[""{name=3, anchor=center, inner sep=0}, "m", "\shortmid"{marking}, from=2-2, to=2-3]
      \arrow[""{name=4, anchor=center, inner sep=0}, "{\id_y}", "\shortmid"{marking}, from=1-3, to=1-4]
      \arrow["g", from=1-4, to=2-4]
      \arrow[""{name=5, anchor=center, inner sep=0}, "{g_!}", "\shortmid"{marking}, from=2-3, to=2-4]
      \arrow["g", from=2-3, to=3-3]
      \arrow[Rightarrow, no head, from=1-3, to=2-3]
      \arrow[""{name=6, anchor=center, inner sep=0}, "m", "\shortmid"{marking}, from=1-2, to=1-3]
      \arrow[Rightarrow, no head, from=2-4, to=3-4]
      \arrow[""{name=7, anchor=center, inner sep=0}, "{\id_{y'}}"', "\shortmid"{marking}, from=3-3, to=3-4]
      \arrow[""{name=8, anchor=center, inner sep=0}, "{f_!}"', "\shortmid"{marking}, from=3-1, to=3-2]
      \arrow[""{name=9, anchor=center, inner sep=0}, "{\id_x}", "\shortmid"{marking}, from=2-1, to=2-2]
      \arrow[Rightarrow, no head, from=2-1, to=3-1]
      \arrow[Rightarrow, no head, from=1-1, to=2-1]
      \arrow[Rightarrow, no head, from=1-2, to=2-2]
      \arrow[""{name=10, anchor=center, inner sep=0}, "{\id_x}", "\shortmid"{marking}, from=1-1, to=1-2]
      \arrow[Rightarrow, no head, from=2-5, to=3-5]
      \arrow[""{name=11, anchor=center, inner sep=0}, "{n'}"', "\shortmid"{marking}, from=3-4, to=3-5]
      \arrow[""{name=12, anchor=center, inner sep=0}, "{h_!}", "\shortmid"{marking}, from=1-5, to=1-6]
      \arrow[Rightarrow, no head, from=1-6, to=2-6]
      \arrow[""{name=13, anchor=center, inner sep=0}, "{\id_{z'}}", "\shortmid"{marking}, from=2-5, to=2-6]
      \arrow[Rightarrow, no head, from=2-6, to=3-6]
      \arrow[""{name=14, anchor=center, inner sep=0}, "{\id_{z'}}"', "\shortmid"{marking}, from=3-5, to=3-6]
      \arrow["\beta"{description}, draw=none, from=0, to=2]
      \arrow["\alpha"{description}, draw=none, from=3, to=1]
      \arrow["\eta"{description}, draw=none, from=4, to=5]
      \arrow["\varepsilon"{description}, draw=none, from=5, to=7]
      \arrow["\eta"{description}, draw=none, from=9, to=8]
      \arrow["{1_m}"{description}, draw=none, from=6, to=3]
      \arrow["{1_{\id_x}}"{description}, draw=none, from=10, to=9]
      \arrow["{1_{n'}}"{description}, draw=none, from=2, to=11]
      \arrow["\varepsilon"{description}, draw=none, from=12, to=13]
      \arrow["{1_{\id_{z'}}}"{description}, draw=none, from=13, to=14]
    \end{tikzcd}.
  \end{equation*}
  Comparing with the assignment $\alpha \mapsto \alpha_!$ in \cref{eq:sliding-companion} proves
  the first statement. The second statement amounts to the defining equation
  $\eta \odot \varepsilon = 1_{f_!}$.
\end{proof}

A transpose of this lemma is also true:

\begin{lemma}[Internal functoriality of sliding, special]
  \label{lem:sliding-functoriality-internal-special}
  Suppose $h$, $h'$, and $h''$ are arrows with chosen companions in a double
  category. Then, under the bijection in \cref{lem:sliding-general}, we have
  \begin{equation*}
    \begin{tikzcd}
      x & y \\
      {x'} & {y'} \\
      {x''} & {y''}
      \arrow[""{name=0, anchor=center, inner sep=0}, "{h_!}", "\shortmid"{marking}, from=1-1, to=1-2]
      \arrow[""{name=1, anchor=center, inner sep=0}, "{h'_!}", "\shortmid"{marking}, from=2-1, to=2-2]
      \arrow[""{name=2, anchor=center, inner sep=0}, "{h''_!}"', "\shortmid"{marking}, from=3-1, to=3-2]
      \arrow["f"', from=1-1, to=2-1]
      \arrow["g", from=1-2, to=2-2]
      \arrow["{f'}"', from=2-1, to=3-1]
      \arrow["{g'}", from=2-2, to=3-2]
      \arrow["{\alpha_!}"{description, pos=0.4}, draw=none, from=0, to=1]
      \arrow["{\alpha'_!}"{description}, draw=none, from=1, to=2]
    \end{tikzcd}
    \quad\leadsto\quad
    \begin{tikzcd}[row sep=scriptsize]
      x & x & x \\
      {x'} & {x'} & y \\
      {x''} & {y'} & {y'} \\
      {y''} & {y''} & {y''}
      \arrow[""{name=0, anchor=center, inner sep=0}, "{\id_x}", "\shortmid"{marking}, from=1-1, to=1-2]
      \arrow["f"', from=1-1, to=2-1]
      \arrow["f", from=1-2, to=2-2]
      \arrow[""{name=1, anchor=center, inner sep=0}, "{\id_{x'}}"', "\shortmid"{marking}, from=2-1, to=2-2]
      \arrow["{f'}"', from=2-1, to=3-1]
      \arrow["{h''}"', from=3-1, to=4-1]
      \arrow["{g'}"', from=3-2, to=4-2]
      \arrow[""{name=2, anchor=center, inner sep=0}, "{\id_{y''}}"', "\shortmid"{marking}, from=4-1, to=4-2]
      \arrow["{h'}"', from=2-2, to=3-2]
      \arrow[""{name=3, anchor=center, inner sep=0}, "{\id_x}", "\shortmid"{marking}, from=1-2, to=1-3]
      \arrow[""{name=4, anchor=center, inner sep=0}, "{\id_{y''}}"', "\shortmid"{marking}, from=4-2, to=4-3]
      \arrow["{g'}", from=3-3, to=4-3]
      \arrow[""{name=5, anchor=center, inner sep=0}, "{\id_{y'}}", "\shortmid"{marking}, from=3-2, to=3-3]
      \arrow["h", from=1-3, to=2-3]
      \arrow["g", from=2-3, to=3-3]
      \arrow["{\id_f}"{description}, draw=none, from=0, to=1]
      \arrow["{\alpha'}"{description}, draw=none, from=1, to=2]
      \arrow["\alpha"{description}, draw=none, from=3, to=5]
      \arrow["{\id_{g'}}"{description}, draw=none, from=5, to=4]
    \end{tikzcd}
    \qquad\text{and}\qquad
    \begin{tikzcd}[sep=scriptsize]
      x & y \\
      x & y
      \arrow[""{name=0, anchor=center, inner sep=0}, "{h_!}", "\shortmid"{marking}, from=1-1, to=1-2]
      \arrow[Rightarrow, no head, from=1-1, to=2-1]
      \arrow[Rightarrow, no head, from=1-2, to=2-2]
      \arrow[""{name=1, anchor=center, inner sep=0}, "{h_!}"', "\shortmid"{marking}, from=2-1, to=2-2]
      \arrow["{1_{h_!}}"{description}, draw=none, from=0, to=1]
    \end{tikzcd}
    \;\leadsto\;
    \begin{tikzcd}[sep=scriptsize]
      x & x \\
      y & y
      \arrow["h"', from=1-1, to=2-1]
      \arrow[""{name=0, anchor=center, inner sep=0}, "{\id_x}", "\shortmid"{marking}, from=1-1, to=1-2]
      \arrow["h", from=1-2, to=2-2]
      \arrow[""{name=1, anchor=center, inner sep=0}, "{\id_y}"', "\shortmid"{marking}, from=2-1, to=2-2]
      \arrow["{\id_h}"{description}, draw=none, from=0, to=1]
    \end{tikzcd}.
  \end{equation*}
\end{lemma}
\begin{proof}
  Dually to \cref{lem:sliding-functoriality-external-globular}, the first
  statement is proved by pre-composing the cell $\alpha_! \cdot \alpha'_!$ with
  $\eta$, post-composing with $\varepsilon$, and inserting the identity $1_{h'_!}= \eta \odot \varepsilon$ in
  the middle. The second statement amounts to the defining equation
  $\eta \cdot \varepsilon = \id_h$.
\end{proof}

The next two lemmas use the functoriality of companions.

\begin{lemma}[External functoriality of sliding, special]
  \label{lem:sliding-functoriality-external-special}
  Suppose $x \xto{k} y \xto{\ell} z$ and $x' \xto{k'} y' \xto{\ell'} z'$ are arrows
  with chosen companions in a double category, and choose the companions of
  their composites to be $k_! \odot \ell_!$ and $k'_! \odot \ell'_!$, respectively. Also,
  choose the companions of the identities $1_x$ and $1_{x'}$ to be $\id_x$ and
  $\id_{x'}$. Then, under the bijection in \cref{lem:sliding-general}, we have
  \begin{equation*}
    \begin{tikzcd}
      x & y & z \\
      {x'} & {y'} & {z'}
      \arrow["f"', from=1-1, to=2-1]
      \arrow["g"{description}, from=1-2, to=2-2]
      \arrow["h", from=1-3, to=2-3]
      \arrow[""{name=0, anchor=center, inner sep=0}, "{k_!}", "\shortmid"{marking}, from=1-1, to=1-2]
      \arrow[""{name=1, anchor=center, inner sep=0}, "{k'_!}"', "\shortmid"{marking}, from=2-1, to=2-2]
      \arrow[""{name=2, anchor=center, inner sep=0}, "{\ell_!}", "\shortmid"{marking}, from=1-2, to=1-3]
      \arrow[""{name=3, anchor=center, inner sep=0}, "{\ell'_!}"', "\shortmid"{marking}, from=2-2, to=2-3]
      \arrow["{\alpha_!}"{description}, draw=none, from=0, to=1]
      \arrow["{\beta_!}"{description}, draw=none, from=2, to=3]
    \end{tikzcd}
    \;\leadsto\;
    \begin{tikzcd}[row sep=scriptsize]
      x & x & x \\
      {x'} & y & y \\
      {y'} & {y'} & z \\
      {z'} & {z'} & {z'}
      \arrow["f"', from=1-1, to=2-1]
      \arrow["g"', from=2-2, to=3-2]
      \arrow["h", from=3-3, to=4-3]
      \arrow["k", from=1-2, to=2-2]
      \arrow["{k'}"', from=2-1, to=3-1]
      \arrow["{\ell'}"', from=3-2, to=4-2]
      \arrow["{\ell'}"', from=3-1, to=4-1]
      \arrow[""{name=0, anchor=center, inner sep=0}, "{\id_{y'}}", "\shortmid"{marking}, from=3-1, to=3-2]
      \arrow[""{name=1, anchor=center, inner sep=0}, "{\id_x}", "\shortmid"{marking}, from=1-1, to=1-2]
      \arrow["\ell", from=2-3, to=3-3]
      \arrow[""{name=2, anchor=center, inner sep=0}, "{\id_{z'}}"', "\shortmid"{marking}, from=4-1, to=4-2]
      \arrow[""{name=3, anchor=center, inner sep=0}, "{\id_y}"', "\shortmid"{marking}, from=2-2, to=2-3]
      \arrow[""{name=4, anchor=center, inner sep=0}, "{\id_{z'}}"', "\shortmid"{marking}, from=4-2, to=4-3]
      \arrow["k", from=1-3, to=2-3]
      \arrow[""{name=5, anchor=center, inner sep=0}, "{\id_x}", "\shortmid"{marking}, from=1-2, to=1-3]
      \arrow["\alpha"{description}, draw=none, from=1, to=0]
      \arrow["{\id_{\ell'}}"{description}, draw=none, from=0, to=2]
      \arrow["\beta"{description}, draw=none, from=3, to=4]
      \arrow["{\id_k}"{description}, draw=none, from=5, to=3]
    \end{tikzcd}
    \qquad\text{and}\qquad
    \begin{tikzcd}[sep=scriptsize]
      x & x \\
      {x'} & {x'}
      \arrow["f"', from=1-1, to=2-1]
      \arrow[""{name=0, anchor=center, inner sep=0}, "{\id_x}", "\shortmid"{marking}, from=1-1, to=1-2]
      \arrow["f", from=1-2, to=2-2]
      \arrow[""{name=1, anchor=center, inner sep=0}, "{\id_{x'}}"', "\shortmid"{marking}, from=2-1, to=2-2]
      \arrow["{\id_f}"{description}, draw=none, from=0, to=1]
    \end{tikzcd}
    \;\leadsto\;
    \begin{tikzcd}[sep=scriptsize]
      x & x \\
      {x'} & x \\
      {x'} & {x'}
      \arrow["{1_{x'}}"', from=2-1, to=3-1]
      \arrow["f"', from=1-1, to=2-1]
      \arrow[""{name=0, anchor=center, inner sep=0}, "{\id_x}", "\shortmid"{marking}, from=1-1, to=1-2]
      \arrow["{1_x}", from=1-2, to=2-2]
      \arrow["{f'}", from=2-2, to=3-2]
      \arrow[""{name=1, anchor=center, inner sep=0}, "{\id_{x'}}"', "\shortmid"{marking}, from=3-1, to=3-2]
      \arrow["{\id_f}"{description}, draw=none, from=0, to=1]
    \end{tikzcd}.
  \end{equation*}
\end{lemma}
\begin{proof}
  By \cite[Lemma 3.13]{shulman2010}, the composite $k_! \odot \ell_!$ is indeed a
  companion for $k \cdot \ell$, with unit and counit cells
  \begin{equation*}
    \begin{tikzcd}[row sep=scriptsize]
      x & x & x \\
      x & y & y \\
      x & y & z
      \arrow["k", from=1-2, to=2-2]
      \arrow[""{name=0, anchor=center, inner sep=0}, "{k_!}"', "\shortmid"{marking}, from=2-1, to=2-2]
      \arrow[Rightarrow, no head, from=1-1, to=2-1]
      \arrow[""{name=1, anchor=center, inner sep=0}, "\shortmid"{marking}, Rightarrow, no head, from=1-1, to=1-2]
      \arrow[""{name=2, anchor=center, inner sep=0}, "{\ell_!}"', "\shortmid"{marking}, from=3-2, to=3-3]
      \arrow["\ell", from=2-3, to=3-3]
      \arrow["k", from=1-3, to=2-3]
      \arrow[""{name=3, anchor=center, inner sep=0}, "\shortmid"{marking}, Rightarrow, no head, from=2-2, to=2-3]
      \arrow[""{name=4, anchor=center, inner sep=0}, "\shortmid"{marking}, Rightarrow, no head, from=1-2, to=1-3]
      \arrow[Rightarrow, no head, from=2-2, to=3-2]
      \arrow[""{name=5, anchor=center, inner sep=0}, "{k_!}"', "\shortmid"{marking}, from=3-1, to=3-2]
      \arrow[Rightarrow, no head, from=2-1, to=3-1]
      \arrow["\eta"{description}, draw=none, from=1, to=0]
      \arrow["{\id_k}"{description}, draw=none, from=4, to=3]
      \arrow["\eta"{description}, draw=none, from=3, to=2]
      \arrow["{1_{k_!}}"{description, pos=0.6}, draw=none, from=0, to=5]
    \end{tikzcd}
    \qquad\text{and}\qquad
    \begin{tikzcd}[row sep=scriptsize]
      x & y & z \\
      y & y & z \\
      z & z & z
      \arrow[""{name=0, anchor=center, inner sep=0}, "{k_!}", "\shortmid"{marking}, from=1-1, to=1-2]
      \arrow["k"', from=1-1, to=2-1]
      \arrow[""{name=1, anchor=center, inner sep=0}, "\shortmid"{marking}, Rightarrow, no head, from=2-1, to=2-2]
      \arrow[Rightarrow, no head, from=1-2, to=2-2]
      \arrow[""{name=2, anchor=center, inner sep=0}, "{\ell_!}", "\shortmid"{marking}, from=1-2, to=1-3]
      \arrow[Rightarrow, no head, from=1-3, to=2-3]
      \arrow[""{name=3, anchor=center, inner sep=0}, "{\ell_!}", "\shortmid"{marking}, from=2-2, to=2-3]
      \arrow["\ell"', from=2-1, to=3-1]
      \arrow[""{name=4, anchor=center, inner sep=0}, "\shortmid"{marking}, Rightarrow, no head, from=3-1, to=3-2]
      \arrow["\ell"', from=2-2, to=3-2]
      \arrow[Rightarrow, no head, from=2-3, to=3-3]
      \arrow[""{name=5, anchor=center, inner sep=0}, "\shortmid"{marking}, Rightarrow, no head, from=3-2, to=3-3]
      \arrow["\varepsilon"{description}, draw=none, from=0, to=1]
      \arrow["{1_{\ell_!}}"{description, pos=0.4}, draw=none, from=2, to=3]
      \arrow["{\id_\ell}"{description}, draw=none, from=1, to=4]
      \arrow["\varepsilon"{description}, draw=none, from=3, to=5]
    \end{tikzcd}.
  \end{equation*}
  The first statement is proved by pre-composing the cell
  $\alpha_! \odot \beta_!$ with the unit for $k_! \odot \ell_!$, post-composing
  with the counit for $k'_! \odot \ell'_!$, and using the inverse assignment
  $\alpha_! \mapsto \alpha$ in \cref{eq:sliding-companion-inv}. The second
  statement is proved similarly in view of \cite[\mbox{Lemma
    3.12}]{shulman2010}.
\end{proof}

\begin{lemma}[Internal functoriality of sliding, globular]
  \label{lem:sliding-functoriality-internal-globular}
  Suppose $x \xto{f} x' \xto{f'} x''$ and $y \xto{g} y' \xto{g'} y''$ are arrows
  with chosen companions in a double category, and choose the companions of
  their composites to be the composites of their companions, and likewise for
  identities. Then, under the bijection in \cref{lem:sliding-globular}, we have
  \begin{equation*}
    \begin{tikzcd}[row sep=scriptsize]
      x & y \\
      {x'} & {y'} \\
      {x''} & {y''}
      \arrow["f"', from=1-1, to=2-1]
      \arrow["g", from=1-2, to=2-2]
      \arrow[""{name=0, anchor=center, inner sep=0}, "m", "\shortmid"{marking}, from=1-1, to=1-2]
      \arrow[""{name=1, anchor=center, inner sep=0}, "{m'}", "\shortmid"{marking}, from=2-1, to=2-2]
      \arrow["{f'}"', from=2-1, to=3-1]
      \arrow["{g'}", from=2-2, to=3-2]
      \arrow[""{name=2, anchor=center, inner sep=0}, "{m''}"', "\shortmid"{marking}, from=3-1, to=3-2]
      \arrow["\alpha"{description}, draw=none, from=0, to=1]
      \arrow["{\alpha'}"{description}, draw=none, from=1, to=2]
    \end{tikzcd}
    \;\rightsquigarrow\;
    \begin{tikzcd}[row sep=scriptsize]
      x & y & {y'} & {y''} \\
      x & {x'} & {y'} & {y''} \\
      x & {x'} & {x''} & {y''}
      \arrow["m", "\shortmid"{marking}, from=1-1, to=1-2]
      \arrow["{m'}", "\shortmid"{marking}, from=2-2, to=2-3]
      \arrow["{m''}"', "\shortmid"{marking}, from=3-3, to=3-4]
      \arrow["{g_!}", "\shortmid"{marking}, from=1-2, to=1-3]
      \arrow[""{name=0, anchor=center, inner sep=0}, "{f_!}", "\shortmid"{marking}, from=2-1, to=2-2]
      \arrow[""{name=1, anchor=center, inner sep=0}, "{g'_!}"', "\shortmid"{marking}, from=2-3, to=2-4]
      \arrow[""{name=2, anchor=center, inner sep=0}, "{g'_!}", "\shortmid"{marking}, from=1-3, to=1-4]
      \arrow[Rightarrow, no head, from=1-3, to=2-3]
      \arrow[Rightarrow, no head, from=1-1, to=2-1]
      \arrow["{\alpha_!}"{description}, draw=none, from=1-2, to=2-2]
      \arrow["{f'_!}"', "\shortmid"{marking}, from=3-2, to=3-3]
      \arrow[Rightarrow, no head, from=2-2, to=3-2]
      \arrow[Rightarrow, no head, from=2-4, to=3-4]
      \arrow[Rightarrow, no head, from=1-4, to=2-4]
      \arrow[""{name=3, anchor=center, inner sep=0}, "{f_!}"', "\shortmid"{marking}, from=3-1, to=3-2]
      \arrow[Rightarrow, no head, from=2-1, to=3-1]
      \arrow["{\alpha'_!}"{description}, draw=none, from=2-3, to=3-3]
      \arrow["{1_{g'_!}}"{description}, draw=none, from=2, to=1]
      \arrow["{1_{f_!}}"{description}, draw=none, from=0, to=3]
    \end{tikzcd}
    \qquad\text{and}\qquad
    \begin{tikzcd}
      x & y \\
      x & y
      \arrow[""{name=0, anchor=center, inner sep=0}, "m", "\shortmid"{marking}, from=1-1, to=1-2]
      \arrow[""{name=1, anchor=center, inner sep=0}, "m"', "\shortmid"{marking}, from=2-1, to=2-2]
      \arrow[Rightarrow, no head, from=1-1, to=2-1]
      \arrow[Rightarrow, no head, from=1-2, to=2-2]
      \arrow["{1_m}"{description}, draw=none, from=0, to=1]
    \end{tikzcd}
    \;\rightsquigarrow\;
    \begin{tikzcd}
      x & y \\
      x & y
      \arrow[""{name=0, anchor=center, inner sep=0}, "m", "\shortmid"{marking}, from=1-1, to=1-2]
      \arrow[""{name=1, anchor=center, inner sep=0}, "m"', "\shortmid"{marking}, from=2-1, to=2-2]
      \arrow[Rightarrow, no head, from=1-1, to=2-1]
      \arrow[Rightarrow, no head, from=1-2, to=2-2]
      \arrow["{1_m}"{description}, draw=none, from=0, to=1]
    \end{tikzcd}.
  \end{equation*}
\end{lemma}

The proof this lemma is similar to that of the previous one and is omitted.

\begin{lemma}[Naturality of sliding]
  \label{lem:sliding-naturality}
  Suppose $f: x \to w$ and $g: y \to z$ are arrows with companions in a double
  category. Then an equation between cells as on the left
  \begin{equation*}
    \begin{tikzcd}
      x & y \\
      w & z \\
      w & z
      \arrow["f"', from=1-1, to=2-1]
      \arrow["g", from=1-2, to=2-2]
      \arrow[""{name=0, anchor=center, inner sep=0}, "q"', "\shortmid"{marking}, from=3-1, to=3-2]
      \arrow[""{name=1, anchor=center, inner sep=0}, "m", "\shortmid"{marking}, from=1-1, to=1-2]
      \arrow[""{name=2, anchor=center, inner sep=0}, "n", "\shortmid"{marking}, from=2-1, to=2-2]
      \arrow[Rightarrow, no head, from=2-1, to=3-1]
      \arrow[Rightarrow, no head, from=2-2, to=3-2]
      \arrow["\alpha"{description}, draw=none, from=1, to=2]
      \arrow["\delta"{description}, draw=none, from=2, to=0]
    \end{tikzcd}
    =
    \begin{tikzcd}
      x & y \\
      x & y \\
      w & z
      \arrow["f"', from=2-1, to=3-1]
      \arrow["g", from=2-2, to=3-2]
      \arrow[""{name=0, anchor=center, inner sep=0}, "p", "\shortmid"{marking}, from=2-1, to=2-2]
      \arrow[""{name=1, anchor=center, inner sep=0}, "q"', "\shortmid"{marking}, from=3-1, to=3-2]
      \arrow[""{name=2, anchor=center, inner sep=0}, "m", "\shortmid"{marking}, from=1-1, to=1-2]
      \arrow[Rightarrow, no head, from=1-1, to=2-1]
      \arrow[Rightarrow, no head, from=1-2, to=2-2]
      \arrow["\gamma"{description}, draw=none, from=2, to=0]
      \arrow["\beta"{description}, draw=none, from=0, to=1]
    \end{tikzcd}
    \qquad\leadsto\qquad
    \begin{tikzcd}
      x & y & z \\
      x & w & z \\
      x & w & z
      \arrow[""{name=0, anchor=center, inner sep=0}, "q"', "\shortmid"{marking}, from=3-2, to=3-3]
      \arrow["m", "\shortmid"{marking}, from=1-1, to=1-2]
      \arrow[""{name=1, anchor=center, inner sep=0}, "n", "\shortmid"{marking}, from=2-2, to=2-3]
      \arrow[Rightarrow, no head, from=2-2, to=3-2]
      \arrow[Rightarrow, no head, from=2-3, to=3-3]
      \arrow[""{name=2, anchor=center, inner sep=0}, "{f_!}", "\shortmid"{marking}, from=2-1, to=2-2]
      \arrow["{g_!}", "\shortmid"{marking}, from=1-2, to=1-3]
      \arrow[Rightarrow, no head, from=1-3, to=2-3]
      \arrow["{\alpha_!}"{description}, draw=none, from=1-2, to=2-2]
      \arrow[Rightarrow, no head, from=1-1, to=2-1]
      \arrow[""{name=3, anchor=center, inner sep=0}, "{f_!}"', "\shortmid"{marking}, from=3-1, to=3-2]
      \arrow[Rightarrow, no head, from=2-1, to=3-1]
      \arrow["\delta"{description}, draw=none, from=1, to=0]
      \arrow["{1_{f_!}}"{description}, draw=none, from=2, to=3]
    \end{tikzcd}
    =
    \begin{tikzcd}
      x & y & z \\
      x & y & z \\
      x & w & z
      \arrow[""{name=0, anchor=center, inner sep=0}, "p"', "\shortmid"{marking}, from=2-1, to=2-2]
      \arrow["q"', "\shortmid"{marking}, from=3-2, to=3-3]
      \arrow[""{name=1, anchor=center, inner sep=0}, "m", "\shortmid"{marking}, from=1-1, to=1-2]
      \arrow[Rightarrow, no head, from=1-1, to=2-1]
      \arrow[Rightarrow, no head, from=1-2, to=2-2]
      \arrow["{f_!}"', "\shortmid"{marking}, from=3-1, to=3-2]
      \arrow[""{name=2, anchor=center, inner sep=0}, "{g_!}"', "\shortmid"{marking}, from=2-2, to=2-3]
      \arrow[Rightarrow, no head, from=2-3, to=3-3]
      \arrow[Rightarrow, no head, from=2-1, to=3-1]
      \arrow["{\beta_!}"{description}, draw=none, from=2-2, to=3-2]
      \arrow[Rightarrow, no head, from=1-3, to=2-3]
      \arrow[""{name=3, anchor=center, inner sep=0}, "{g_!}", "\shortmid"{marking}, from=1-2, to=1-3]
      \arrow["\gamma"{description}, draw=none, from=1, to=0]
      \arrow["{1_{g_!}}"{description}, draw=none, from=3, to=2]
    \end{tikzcd}
  \end{equation*}
  implies the equation between cells on the right.
\end{lemma}
\begin{proof}
  Follows directly by composing on the left and right with $\eta$ and
  $\varepsilon$ as in \cref{eq:sliding-companion}.
\end{proof}

\subsection{Transposition and companions as biadjoints}
\label{sec:companions-biadjoint}

The operation of sliding and its many functoriality properties can be distilled
into a single result, characterizing companions in double categories as a right
biadjoint to transposing \emph{strict} double categories. Let us first recall
the double-categorical transpose.

\begin{construction}[Transpose]
  The \define{transpose} $\dbl{D}^\top$ of a strict double category $\dbl{D}$ is
  the strict double category that exchanges the arrows and proarrows of
  $\dbl{D}$, hence transposes the cells of $\dbl{D}$
  \cite[\S{3.2.2}]{grandis2019}. Moreover, transposition is the object part of a
  2-functor
  \begin{equation*}
    (-)^\top: \StrDbl_\pro \to \StrDbl.
  \end{equation*}
  Here the codomain $\StrDbl$ is the 2-category of strict double categories,
  strict double functors, and (tight) natural transformations. The domain
  $\StrDbl_\pro$ is the same 2-category, except that its 2-morphisms are now
  \emph{loose} or \emph{horizontal} transformations, which we call
  \define{strict protransformations} in \cref{def:lax-protransformation}.
\end{construction}

Let $\StrDbl_{!,\pro}$ denote the 2-category of strict double categories
equipped with a functorial choice of companion for each arrow, strict double
functors preserving the chosen companions, and strict protransformations. The
transposition 2-functor now restricts/extends to a 2-functor
\begin{equation*}
    (-)^\top: \StrDbl_{!,\pro} \to \Dbl,
\end{equation*}
where, as usual, $\Dbl$ is the 2-category of (pseudo) double categories,
(pseudo) double functors, and natural transformations. Note that while it is
usually suspect to ask for a strictly functorial choice of companions, such
choices are made canonically when constructing the following 2-functor.

\begin{construction}[Companions as a 2-functor]
  There is a 2-functor
  \begin{equation*}
    \Comp: \Dbl \to \StrDbl_{!,\pro}
  \end{equation*}
  that sends a double category $\dbl{D}$ to the strict double category
  $\Comp(\dbl{D})$ whose objects are those of $\dbl{D}$, arrows are arrows in
  $\dbl{D}$ that have companions, proarrows are arbitrary arrows in $\dbl{D}$,
  and cells
  \begin{equation*}
    \begin{tikzcd}
      x & y \\
      w & z
      \arrow["f"', from=1-1, to=2-1]
      \arrow[""{name=0, anchor=center, inner sep=0}, "h", "\shortmid"{marking}, from=1-1, to=1-2]
      \arrow[""{name=1, anchor=center, inner sep=0}, "k"', "\shortmid"{marking}, from=2-1, to=2-2]
      \arrow["g", from=1-2, to=2-2]
      \arrow["\alpha"{description}, draw=none, from=0, to=1]
    \end{tikzcd}
    \quad\text{in } \Comp(\dbl{D})
    \qquad\leftrightsquigarrow\qquad
    \begin{tikzcd}[row sep=scriptsize]
      x & x \\
      y & w \\
      z & z
      \arrow["f", from=1-2, to=2-2]
      \arrow["k", from=2-2, to=3-2]
      \arrow["h"', from=1-1, to=2-1]
      \arrow["g"', from=2-1, to=3-1]
      \arrow[""{name=0, anchor=center, inner sep=0}, "{\id_x}", "\shortmid"{marking}, from=1-1, to=1-2]
      \arrow[""{name=1, anchor=center, inner sep=0}, "{\id_z}"', "\shortmid"{marking}, from=3-1, to=3-2]
      \arrow["\alpha"{description}, draw=none, from=0, to=1]
    \end{tikzcd}
    \quad\text{in } \dbl{D}
  \end{equation*}
  as on the left are special cells in $\dbl{D}$ as on the right. Cells in
  $\Comp(\dbl{D})$ compose via the evident pastings. By construction, every
  arrow in $\Comp(\dbl{D})$ has a canonical choice of companion, namely itself,
  with the external identity in $\dbl{D}$ giving both binding cells.

  To complete the construction, a double functor $F: \dbl{D} \to \dbl{E}$ induces
  a strict double functor $\Comp F: \Comp \dbl{D} \to \Comp \dbl{E}$ that acts as
  the underlying 2-functor of $F$. Finally, given double functors
  $F,G: \dbl{D} \to \dbl{E}$, a natural transformation $\alpha: F \To G$ induces a
  strict protransformation $\Comp \alpha: \Comp F \proTo \Comp G$ comprising
  \begin{itemize}[noitemsep]
    \item for each object $x$ in $\dbl{D}$, the component
      $(\Comp \alpha)_x \coloneqq \alpha_x$, which is an arrow in $\dbl{E}$, hence a
      proarrow in $\Comp \dbl{E}$;
    \item for each arrow $f: x \to y$ in $\Comp \dbl{D}$, which is an arrow in
      $\dbl{D}$ with a companion, the component
      \begin{equation*}
        \begin{tikzcd}[column sep=large]
          Fx & Gx \\
          Fy & Gy
          \arrow[""{name=0, anchor=center, inner sep=0}, "{(\Comp\alpha)_x}", "\shortmid"{marking}, from=1-1, to=1-2]
          \arrow["Ff"', from=1-1, to=2-1]
          \arrow[""{name=1, anchor=center, inner sep=0}, "{(\Comp\alpha)_y}"', "\shortmid"{marking}, from=2-1, to=2-2]
          \arrow["Gf", from=1-2, to=2-2]
          \arrow["{(\Comp\alpha)_f}"{description}, draw=none, from=0, to=1]
        \end{tikzcd}
        \quad\coloneqq\quad
        \begin{tikzcd}[row sep=scriptsize]
          Fx & Fx \\
          Gx & Fy \\
          Gy & Gy
          \arrow[""{name=0, anchor=center, inner sep=0}, "{\id_{Fx}}", "\shortmid"{marking}, from=1-1, to=1-2]
          \arrow[""{name=1, anchor=center, inner sep=0}, "{\id_{Gy}}"', "\shortmid"{marking}, from=3-1, to=3-2]
          \arrow["{\alpha_x}"', from=1-1, to=2-1]
          \arrow["Gf"', from=2-1, to=3-1]
          \arrow["Ff", from=1-2, to=2-2]
          \arrow["{\alpha_y}", from=2-2, to=3-2]
          \arrow["\id"{description}, draw=none, from=0, to=1]
        \end{tikzcd},
      \end{equation*}
      the external identity in $\dbl{E}$ on the naturality square of $\alpha$ at
      $f$. \qedhere
  \end{itemize}
\end{construction}

The 2-functors just constructed are biadjoint to each other:

\begin{theorem}[Companions as a biadjoint] \label{thm:companions-biadjunction}
  The 2-category of strict double categories with a functorial choice of
  companions and the 2-category of double categories are related by a
  biadjunction\footnote{The proof constructs only an ``incoherent''
    biadjunction. However, any such biadjunction can be upgraded to a coherent
    one satisfying the swallowtail identities \cite[\mbox{Lemma
      1.3.9}]{verity1992}.}
  \begin{equation*}
    \begin{tikzcd}
      {\StrDbl_{!,\pro}} & \Dbl
      \arrow[""{name=0, anchor=center, inner sep=0}, "{(-)^\top}", curve={height=-18pt}, from=1-1, to=1-2]
      \arrow[""{name=1, anchor=center, inner sep=0}, "\Comp", curve={height=-18pt}, from=1-2, to=1-1]
      \arrow["\dashv"{anchor=center, rotate=-90}, draw=none, from=0, to=1]
    \end{tikzcd}.
  \end{equation*}
  The component of the counit at a double category $\dbl{D}$ is a double functor
  \begin{equation*}
    (\Comp\dbl{D})^\top \to \dbl{D},
  \end{equation*}
  that is the identity on objects and arrows and sends proarrows (arrows in
  $\dbl{D}$ that have companions) to choices of companions in $\dbl{D}$.
\end{theorem}
\begin{proof}
  We first construct the unit of the biadjunction. Given a strict double
  category $\dbl{C}$ with a functorial choice of companions, the component of
  the unit $\eta$ at $\dbl{C}$ is the strict double functor
  \begin{equation*}
    \eta_{\dbl{C}}: \dbl{C} \to \Comp(\dbl{C}^\top)
  \end{equation*}
  that is the identity on objects and proarrows; sends each arrow $f$ in
  $\dbl{C}$ to the chosen companion $f_!$ in $\dbl{C}$, which is an arrow in
  $\dbl{C}^\top$ with a companion, namely $f$; and sends a cell $\stdInlineCell{\alpha}$
  in $\dbl{C}$ to the cell
  \begin{equation*}
    \begin{tikzcd}
      x & y \\
      w & z
      \arrow["{f_!}"', from=1-1, to=2-1]
      \arrow["{g_!}", from=1-2, to=2-2]
      \arrow[""{name=0, anchor=center, inner sep=0}, "m", "\shortmid"{marking}, from=1-1, to=1-2]
      \arrow[""{name=1, anchor=center, inner sep=0}, "n"', "\shortmid"{marking}, from=2-1, to=2-2]
      \arrow["{\eta_{\dbl{C}}(\alpha)}"{description}, draw=none, from=0, to=1]
    \end{tikzcd}
    \quad\text{in } \Comp(\dbl{C}^\top)
    \qquad\leftrightsquigarrow\qquad
    \begin{tikzcd}
      x & y & z \\
      x & w & z
      \arrow["{g_!}", "\shortmid"{marking}, from=1-2, to=1-3]
      \arrow["m", "\shortmid"{marking}, from=1-1, to=1-2]
      \arrow["n"', "\shortmid"{marking}, from=2-2, to=2-3]
      \arrow["{f_!}"', "\shortmid"{marking}, from=2-1, to=2-2]
      \arrow[Rightarrow, no head, from=1-1, to=2-1]
      \arrow["{\alpha_!}"{description}, draw=none, from=1-2, to=2-2]
      \arrow[Rightarrow, no head, from=1-3, to=2-3]
    \end{tikzcd}
    \quad\text{in } \dbl{C}
  \end{equation*}
  obtained by reshaping the cell $\alpha$ in $\dbl{C}$ (\cref{lem:sliding-globular}).
  The strict functoriality of this assignment follows from the functorial choice
  of companions in $\dbl{C}$ along with
  \cref{lem:sliding-functoriality-external-globular,lem:sliding-functoriality-internal-globular}.
  Moreover, the components $\eta_{\dbl{C}}$ assemble into a (strict) 2-natural
  transformation $\eta$ since, by definition, the double functors
  $F: \dbl{C} \to \dbl{C}'$ under consideration preserve chosen companions.

  We now construct the counit of the biadjunction. Given a double category
  $\dbl{D}$, the component at $\dbl{D}$ of the counit $\varepsilon$ is the double functor
  \begin{equation*}
    \varepsilon_{\dbl{D}}: (\Comp\dbl{D})^\top \to \dbl{D}
  \end{equation*}
  that is the identity on objects and arrows; sends an arrow $h: x \to y$ in
  $\dbl{D}$ with a companion to any choice of companion $h_!: x \proto y$; and
  acts on a cell $\inlineCell{x}{y}{w}{z}{h}{k}{f}{g}{\alpha}$ in
  $(\Comp\dbl{D})^\top$ by sliding in $\dbl{D}$ (\cref{lem:sliding-general}):
  \begin{equation*}
    \varepsilon_{\dbl{D}}: \quad
    \begin{tikzcd}[row sep=scriptsize]
      x & x \\
      w & y \\
      z & z
      \arrow["f"', from=1-1, to=2-1]
      \arrow["k"', from=2-1, to=3-1]
      \arrow["g", from=2-2, to=3-2]
      \arrow["h", from=1-2, to=2-2]
      \arrow[""{name=0, anchor=center, inner sep=0}, "{\id_x}", "\shortmid"{marking}, from=1-1, to=1-2]
      \arrow[""{name=1, anchor=center, inner sep=0}, "{\id_z}"', "\shortmid"{marking}, from=3-1, to=3-2]
      \arrow["\alpha"{description}, draw=none, from=0, to=1]
    \end{tikzcd}
    \quad\mapsto\quad
    \begin{tikzcd}
      x & y \\
      w & z
      \arrow["f"', from=1-1, to=2-1]
      \arrow[""{name=0, anchor=center, inner sep=0}, "{h_!}", "\shortmid"{marking}, from=1-1, to=1-2]
      \arrow["g", from=1-2, to=2-2]
      \arrow[""{name=1, anchor=center, inner sep=0}, "{k_!}"', "\shortmid"{marking}, from=2-1, to=2-2]
      \arrow["{\alpha_!}"{description}, draw=none, from=0, to=1]
    \end{tikzcd}.
  \end{equation*}
  The (pseudo) functoriality of this assignment follows from
  \cref{lem:sliding-functoriality-internal-special,lem:sliding-functoriality-external-special}.
  Moreover, the components $\varepsilon_{\dbl{D}}$ asssemble into a
  pseudonatural transformation $\varepsilon$ whose naturality comparisons are
  the canonical isomorphisms witnessing that any (pseudo) double functor
  $F: \dbl{D} \to \dbl{D}'$ preserves companions \cite[\mbox{Lemma
    3.13}]{hansen2019}.

  It remains to show that the triangle identities hold, at least up to
  invertible modifications.
  \begin{equation*}
    \begin{tikzcd}
      {\dbl{C}^\top} & {\Comp(\dbl{C}^\top)^\top} \\
      & {\dbl{C}^\top}
      \arrow[""{name=0, anchor=center, inner sep=0}, "{1_{\dbl{C}^\top}}"', from=1-1, to=2-2]
      \arrow["{\eta_{\dbl{C}}^\top}", from=1-1, to=1-2]
      \arrow["{\varepsilon_{\dbl{C}^\top}}", from=1-2, to=2-2]
      \arrow["\sim", shorten <=4pt, shorten >=4pt, Rightarrow, from=0, to=1-2]
    \end{tikzcd}
    \qquad\qquad
    \begin{tikzcd}
      {\Comp\dbl{D}} & {\Comp((\Comp\dbl{D})^\top)} \\
      & {\Comp\dbl{D}}
      \arrow["{\eta_{\Comp\dbl{D}}}", from=1-1, to=1-2]
      \arrow["{1_{\Comp\dbl{D}}}"', from=1-1, to=2-2]
      \arrow["{\Comp \varepsilon_{\dbl{D}}}", from=1-2, to=2-2]
    \end{tikzcd}
  \end{equation*}
  For any strict double category $\dbl{C}$ with a functorial choice of
  companions, the first triangle can fail to commute because
  $\varepsilon_{\dbl{C}^\top}$ need not choose the companion of an arrow
  $\eta_{\dbl{C}}^\top(f) = f_!$ in $\dbl{C}^\top$ to be $f$ itself, but only
  some proarrow isomorphic to $f$ in $\dbl{C}^\top$. Correcting this discrepancy
  fills the triangle with an invertible icon: a natural isomorphism whose
  component arrows are identities. For any double category $\dbl{D}$, the second
  triangle actually commutes on the nose, because, although
  $\varepsilon_{\dbl{D}}$ is pseudo, $\Comp \varepsilon_{\dbl{D}}$ uses only the
  2-functor underlying it.
\end{proof}

This result refines Grandis and Paré's earlier \cite[\mbox{Theorem
  1.7}]{grandis2004} in several ways. First, the adjunction is between
2-categories rather than mere categories. Also, by using strict double
categories rather than 2-categories on the left side of the adjunction, we can
extract the arrows of a double category that have companions without discarding
the other arrows along the way. Another difference is that, on the right side of
the adjunction, we work with the \emph{property} of having a companion, rather
than the \emph{structure} of a companion pair, and we allow pseudo double
categories and double functors rather merely strict ones. This requires
weakening the adjunction to a biadjunction.

\printbibliography[heading=bibintoc]

\end{document}